\documentclass[11pt]{article}
\usepackage{amsfonts,amsmath,amssymb}
\usepackage{url}
\usepackage{epsfig,latexsym,graphicx,subcaption,amsthm}
\usepackage{esint}
\usepackage{graphicx,latexsym,amssymb,amsmath,amsfonts,fancyhdr,mathrsfs}
\usepackage{euscript}
\usepackage[letterpaper,hmargin=1.0in,vmargin=1.0in]{geometry}

\newtheorem{maintheorem}{Theorem}

\newtheorem{theorem}{Theorem}[section]
\newtheorem{lemma}[theorem]{Lemma}

\newtheorem{proposition}[theorem]{Proposition}

\newtheorem{corollary}[theorem]{Corollary}

\newtheorem{definition}[theorem]{Definition}

\def\XXint#1#2#3{{\setbox0=\hbox{$#1{#2#3}{\int}$ }
\vcenter{\hbox{$#2#3$ }}\kern-.6\wd0}}

\newcommand{\E}{\mathbb{E}}
\renewcommand{\P}{\mathbb{P}}

\newcommand{\Z}{\mathbb{Z}}
\newcommand{\N}{\mathbb{N}}
\newcommand{\R}{\mathbb{R}}

\begin{document}

\title{Last Passage Percolation with a Defect Line \\ and the Solution of the Slow Bond Problem}

\author{R.
Basu
\thanks{Department of Mathematics, Stanford University. Email: rbasu@stanford.edu}
 \and
V. Sidoravicius
\thanks{Courant Institute of Mathematical Sciences, New York, NYU-ECNU Institute of Mathematical Sciences at NYU Shanghai and
Cemaden, Sao Jose dos Campos. Email: vs1138@nyu.edu}
\and
A. Sly
\thanks{University of California, Berkeley. Supported by an Alfred Sloan Fellowship and NSF grant DMS-1208338. Email: sly@stat.berkeley.edu}
}

\date{}
\maketitle

\begin{abstract}
We address the question of how a localized microscopic defect, especially if it
is small with respect to certain dynamic parameters,
affects the macroscopic behavior of a system. In particular we consider two classical exactly solvable models:
Ulam's problem of the maximal increasing sequence and the totally
asymmetric simple exclusion process. For the first model, using its
representation as a Poissonian version of directed last passage percolation on $\mathbb R^2$,
we introduce the defect by placing a positive density of extra points along the diagonal line.  For the latter,
the defect is produced by decreasing the jump rate of each particle when it crosses the origin.

The powerful algebraic tools for studying these processes break down in the perturbed versions of the models.
Taking a more geometric approach we show that in both cases the presence of an arbitrarily small defect  affects the macroscopic
behavior of the system: in Ulam's problem the time constant increases, and for the
exclusion process the flux of particles decreases.  This, in particular,  settles the longstanding ``Slow Bond Problem''.
\end{abstract}

\section{Introduction}

One of the fundamental questions of equilibrium and non-equilibrium dynamics refers to the following
problem: how can a localized defect,  especially if it is small
with respect to certain dynamic parameters,
affect the macroscopic behavior of a system?
Two canonical examples are directed last passage percolation (DLPP) with a diagonal defect line and the one dimensional totally
asymmetric simple exclusion process (TASEP) with a slow bond at the origin.
In their unmodified form, these models are exactly solvable and in the KPZ universality class. 
They have been the subject of intensive study yielding a rich and detailed picture including Tracy-Widom scaling limits~\cite{BDJ99, Jo99}.
Under the addition of small modifications, however, the algebraic tools used to study these models break down.  In this paper we bring a new more geometric 
approach to determine the effect of defects.

For TASEP with a slow bond one asks whether the flux of particles is affected at any arbitrarily small value of slowdown at the origin or if when the defect becomes too weak, the fluctuations in the  bulk destroy the effect of the obstruction so that its presence becomes macroscopically undetectable.  Originally posed by Janowsky and Lebowitz in 1992, this question has proved controversial with various groups of physicists arriving at competing conclusions on the basis of empirical simulation studies and heuristic arguments~(see \cite{clst13} for a detailed background). 
In DLPP the question becomes whether the asymptotic speed is changed in the macroscopic  neighborhood of such a defect at any value of its strength.  Equivalently, one may ask if its asymptotic shape is changed and becomes faceted. 

Such a vanishing presence of the macroscopic effect as a function of the strength of
obstruction represents what sometimes is called, in physics literature, a  {\emph {dynamic phase transition}}.
The existence of such a transition, its scaling properties and the behavior of the system
near  the obstruction are among the most important issues.
In this work we prove that indeed an arbitrarily small defect affects the macroscopic behaviour of these models resolving the longstanding slow bond problem.  We begin with a description of the models and our main results.



\medskip

\noindent {\bf Maximal increasing subsequence.} We consider the classical Ulam's problem of the maximal
increasing subsequence of a random permutation recast in the language of continuum Poissonian last passage percolation:
Let $\Pi$ be a Poisson point process of intensity $1$ on $\mathbb R ^2$.
We let $L_n$ denote the maximum number of points in $\Pi$ along any oriented path from $(0,0)$ to $(n,n)$  calling it the \emph{length} of a maximal
path. Conditional on the number of points in the square $[0,n]^2$ this is distributed as the length of the longest
increasing subsequence of a random permutation. Using a correspondence with Young-Tableaus, Vershik and Kerov~\cite{VerKer77} and Logan and Shepp~\cite{LogShep77}  established that
\begin{equation}
\label{e:unperturbedlln}
\lim_{n\rightarrow \infty} \dfrac{\mathbb E L_n}{n}=2.
\end{equation}
(See also the proof by Aldous and Diaconis using interacting particle systems~\cite{AD95}).
For $\lambda >0$, let $\Sigma_{\lambda}$ be a one dimensional Poisson process of intensity
$\lambda$ on the line $x=y$ independent of $\Pi$ and let $\Pi_{\lambda}$ be the point process obtained
by the union of $\Pi$ and $\Sigma_{\lambda}$. We study the question of how the length of the maximal path is affected by this reinforcing of the diagonal.

Let $L_n^{\lambda}$ denote the maximum number of points of $\Pi_{\lambda}$ on an
increasing path from $(0,0)$ to $(n,n)$. It is easy to observe that taking $\lambda$ sufficiently large
changes the law of large numbers for $L_n^{\lambda}$ from that of $L_n$, i.e., for $\lambda$ sufficiently
large
\begin{equation}
\label{e:largeperturbedlln}
\lim_{n\rightarrow \infty} \dfrac{\mathbb E L_n^{\lambda}}{n}>2.
\end{equation}
An important problem is whether there is a non-trivial phase transition in $\lambda$, i.e., whether
for any $\lambda > 0$ the law of large numbers for $L_n^{\lambda}$ differs from that of $L_n$, or
there exists  $\lambda_c >0$, such that the law of large number for $L_n^{\lambda}$ is same
as that of $L_n$ for $\lambda < \lambda_c$. Our first main result settles this question:
\begin{maintheorem}
\label{t:maintheoremppp}
For every $\lambda>0$,
\begin{equation}
\label{e:perturbedlln}
\lim_{n\rightarrow \infty} \dfrac{\mathbb E L_n^{\lambda}}{n}>2.
\end{equation}
\end{maintheorem}

\medskip
\noindent {\bf The slow bond problem.} Consider DLPP on $\mathbb Z^2_+$, defined by associating with each vertex
$x\in \mathbb Z^2_+$ an independent random variable $\xi_x \sim \exp (1)$. The last passage time is defined as
$$
T_n^{0} = \max_{\pi} \sum_{i=0}^{2n+1} \xi_{x_i};
$$
maximized over all oriented paths in $\mathbb Z^2_+$ from $(0,0)$ to $(n,n)$.  It is well known \cite{Ro81} that
\begin{equation}
\label{e:largeperturbeddlln}
\lim_{n\rightarrow \infty} \dfrac{\mathbb E T_n^{0}}{n} = 4
\end{equation}

\medskip
By a well known mapping (see e.g.\ \cite{sas}) $T_n^0$ also describes passage times of particles in the totally asymmetric exclusion process. Consider the continuous time TASEP $X(t) = (\eta_k(t))^{\infty}_{k=- \infty} \in \{0,1\}^{\mathbb Z}$ for $t\ge 1$. The dynamics of the particles is as follows, a particle at position $k$ (i.e $\eta_k = 1$)  jumps with exponential rate one to $k+1$ provided that position is vacant  (i.e.\ $\eta_{k+1} =0$).
Started from the initial configuration $\mathbb {I}_{( - \infty,0]} (k)$, the so called ``step initial condition'', this process was studied in \cite{Ro81}. In this setting, the time for the particle from position $-n$ to move to 1 is distributed as $T_n^0$. Indeed, it is exactly $T_n^0$ if we couple TASEP and DLPP so that the variable $\xi_{(i,j)}$ represents the time which the particle starting at $-i$ has to wait to perform its $j$-th jump once that position is vacant. The inverse value of the expression in (\ref{e:largeperturbeddlln}) corresponds to the asymptotic rate of particles crossing the bond between $0$ and $1$.

Now let us modify the distribution of passage times, by taking
\begin{equation}
\xi_{(x,y)} \sim
\begin{cases}
\exp (1) \; & \text{if}  \; x \neq y, \\
\exp{(1-\epsilon}) \; & \text{if} \; x=y.
\end{cases}
\end{equation}
and ask the same question: does the law of large numbers for $T_n^{\epsilon}$ change for any $\epsilon > 0$ where $T_n^{\epsilon}$ denotes the last passage time in this setting.

In the TASEP representation this change corresponds to a local modification of the dynamics:
the exponential clock governing particles jumping across the edge $(0,1)$ is decreased from rate 1 to rate $1 - \epsilon$ introducing a \emph{slow bond}.   This version of the process was proposed by Janowsky and Lebowitz~\cite{JL1} (see also  \cite{JL2}), as a model for understanding non-equilibrium
stationary states.

The jump-rate decrease at the origin will increase the particle  density to the immediate left of such
a ``slow bond'' and decrease the density to its
immediate right. The difficulty in analyzing this process comes from the fact that
the effect of any local perturbation in  non-equilibrium systems carrying fluxes of conserved quantities is felt at large scales.  What was not obvious, was if this perturbation, in addition to local effects, may also have a global effect and in particular change the current in the system i.e.\  whether the LLN for $T_n^\epsilon$ changes for any value $\epsilon > 0$ or whether $\epsilon_c$ is strictly greater than 0.

This question generated considerable controversy in theoretical physics and mathematical community, which was
supported from opposite sides by numerical analysis and some theoretical arguments (see \S~\ref{s:background}),
and became known in the literature as the ``Slow Bond Problem'' (\cite{JL1, JL2, timo, mmmthn, htm03}), see \cite{clst13} for a detailed account.
Our second result settles this problem:
\begin{maintheorem}
\label{t:maintheoremdlpp}
In Exponential directed last passage percolation model for every $\epsilon>0$,
\begin{equation}
\label{e:perturbedlln2}
\lim_{n\rightarrow \infty} \dfrac{\mathbb E T_n^{\epsilon}}{n}>4.
\end{equation}
\end{maintheorem}

One of the key features of the exactly solvable models in the KPZ universality class, in particular the two models described above, is that they exhibit fluctuation exponent of $\frac{1}{3}$, i.e.\ $L_n$ and $T_n^0$ have fluctuations of order $n^{1/3}$, see Section \ref{s:background} for more details. Adding defects changes this as well. In fact, it can be shown using our techniques that as a consequence of Theorem \ref{t:maintheoremppp} and Theorem \ref{t:maintheoremdlpp}, for any positive value of $\lambda$ (resp.\ $\epsilon$) there is pinning and the fluctuation of $L_n^{\lambda}$ (resp.\ $T_n^{\epsilon}$) is of the order $n^{1/2}$, and moreover the limiting behaviour is Gaussian as opposed to Tracy-Widom in the exactly solvable cases. We shall not provide a detailed proof of this, but a further discussion is provided at the conclusion of the paper.  

 \bigskip

\subsection{Background}\label{s:background}
Non-equilibrium interfaces with localized defect that display nontrivial scaling properties are common in physical, chemical and biological systems. The problem we are
interested in can be cast in several different, but closely related forms: as a stochastic driven transport
through narrow channels  with obstructions \cite{htm03}, as a growth model with defect line \cite{mmmthn}, or as a polymer pinning problem of a one-dimensional interface \cite{hn, ay}.
Most of these models in two dimensions   (sometimes interpreted as 1+1 dimension) belong, in absence of defects, to the Kardar-Parisi-Zhang (KPZ) universality class. The question if arbitrarily small microscopic obstruction may change local macroscopic behavior of non-equilibrium systems became broadly discussed starting in the late eighties.

For the TASEP model with a slow bond Janowsky and Lebowitz in \cite{JL2} provided a non-rigorous mean field argument, suggesting that if the jump rate at the origin is $1-\epsilon$,  then the current should become equal to $(1-\epsilon)/(2- \epsilon)^2$, thus  supporting the conjecture that $\epsilon_c = 0$. This conjecture was also supported by theoretical renormalization group arguments in the study of a directed polymer pinning transition at low temperatures \cite{hn}. An alternative heuristic argument based on ``influence percolation'' was discussed in \cite{bssv}. In a more recent work \cite{clst13}, based on a non-rigorous theoretical argument and analysis of the first sixteen terms of formal power series expansion of the current, authors predicted that for small values of $\epsilon>0$ the current should behave as $1/4 - \gamma \exp(-a/\epsilon)$ with $a \approx 2$.

On the rigorous side, a first upper bound for the critical value of the slow-down rate was derived in~\cite{crez}  by approximating the slow bond model with an exclusion process whose rates vary more regularly
in space. An alternative bound for the critical slow-down was provided in \cite{lgt2}. Finally the most complete and general hydrodynamic limit results were  obtained in \cite{timo} for all values $0 < \epsilon < 1$ of the slow-down. However the hydrodynamic limit can not make the distinction of whether the slow bond disturbs the hydrodynamic profile for all values of $\epsilon > 0$.  Letting $\kappa_{_{1- \epsilon}}$ denote the inverse maximal current in presence of a $1-\epsilon$ slow bond~\cite{timo}  obtained the following bound:
\begin{equation}
\max\Big \{ 4, \frac{3}{2} + \frac{(1-\epsilon)^2 + 2(2-\epsilon)}{2(1-\epsilon)(2 - \epsilon)}   \Big\} \leq \kappa_{_{1-\epsilon} }\leq 3 + \frac{1}{1-\epsilon}.
\end{equation}


At the same time, a competing set of theoretical arguments, mostly appearing in the theoretical physics literature, supported also by numerical
data,   pointed towards the possibility that $\epsilon_c >0$.
In \cite{km}  early numerical data for a related polynuclear growth model,
involving parallel updating, was interpreted as suggesting that the critical delay value
in TASEP with slow bond model  should be  $\epsilon_c \approx 0.3$. In another study, based on a finite size scaling
analysis of simulation data~\cite{htm03}  concluded that $\epsilon_c \approx 0.2$. For a very recent numerical study suggesting $\epsilon_{c}=0$, see \cite{SPS15}.

An important rigorous step forward was made  by Baik and Rains \cite{br} where, among several
cases of interest,  they also consider the so called  ``symmetrized'' version
of the maximal increasing sequence with a defect line, for which they showed that $\lambda_c =1$.   At first glance this may seem at odds with Theorem~\ref{t:maintheoremppp} showing that $\lambda_c =0$ in the original model.
It is shown in~\cite[Theorem 3.2]{br} that the constant in the
LLN in the symmetrized system with $0 < \lambda \leq 1$ reinforcement on the diagonal coincides
with that of the LLN in the non-symmetrized system with no reinforcement on the diagonal, and is equal to 2.
However, if we look in both processes at the picture of their level lines, sometimes also called Hammersley process trajectories,
(see~\cite{AD95}), we observe that  in the non-symmetrized model with no perturbation the level lines are in equilibrium
and in particular, their intersection with the main diagonal
forms a stationary point process of intensity 2 (see \cite{ssv}). However, in the symmetrized case with no reinforcement
the level lines are ``out of equilibrium'' in vicinity $n^{2/3}$ of the main diagonal.  Adding an extra rate $\lambda$ Poisson point with $0 < \lambda < 1$ on the main diagonal brings this process  closer to equilibrium as $\lambda$ increases from $0$ to $1$.  When $\lambda$ reaches  1, at which point the level lines in symmetrized process become  ``equilibrated''.
After that for any positive increase above the value 1 the LLN in the symmetrized (and now equilibrated) model changes.  In the context of the original non-symmetrized model there is
no need to pay this extra cost in order to equilibrate the system and this corresponds to a change of the LLN at any positive value of reinforcement.

\subsection{Tracy-Widom Limit, Moderate Deviations and $n^{2/3}$ Fluctuations}

The two models that we consider (i.e., the longest increasing subsequence and the exponential last passage percolation) are exactly solvable in absence of a defect and it is possible to obtain scaling limits and precise moderate deviation tail bounds for $L_n$ and $T_n$. 
We shall treat these results from the exactly solvable models as a ``black box'' in our arguments. Using these estimates the problems at hand can be treated as percolation  type questions. Here we collect the results we need for the longest increasing subsequence model which is the model we shall primarily work with in this paper. Similar results are also available in the literature for the exponential directed last passage percolation model, and we shall quote them in \S~\ref{s:discrete} where we explain how to adapt our arguments to the Exponential case.

\subsubsection{Scaling limit}
Baik, Deift and Johansonn \cite{BDJ99} proved the following fundamental result about fluctuations of $L_n$. Let $\Pi$ be a homogeneous Poisson point process on $\R^2$ with rate 1. Let $u_{t}$ be a point on the first quadrant of $\R^2$ such that the area of the rectangle with bottom left corner $(0,0)$ and the top right corner $u_{t}$ is $t$. Let $X_{u_{t}}$ denote the maximum number of points on $\Pi$ on an increasing path from $(0,0)$ to $u_{t}$. By the scaling of Possion point process it is clear that the distribution of $X_{t}=X_{u_{t}}$ depends on $u_{t}$ only through $t$. The following Theorem is the main result from \cite{BDJ99}.
\begin{theorem}
\label{BDJ99}
Let $F_{TW}$ be the GUE Tracy-Widom distribution.
As $t \rightarrow \infty$,
\begin{equation}
\label{e:twlimit}
\dfrac{X_{t}-2\sqrt{t}}{t^{1/6}} \stackrel{d}{\rightarrow} F_{TW}
\end{equation}
where $\stackrel{d}{\rightarrow}$ denotes convergence in distribution.
\end{theorem}

For a definition of the GUE Tracy-Widom distribution (also known as $F_2$ distribution) which also arises as the distribution of the scaling limit of largest eigenvalue in GUE random matrices, see \cite{BDJ99}. 

\subsubsection{Moderate deviation estimates}
We also require estimates from the tails of the distribution and quote the following moderate deviation estimates for upper and lower tails of longest increasing subsequence from \cite{LM01} and \cite{LMS02} respectively.
The following theorem is an immediate corollary of Theorem~1.3 of~\cite{LM01}.

\begin{theorem}
\label{t:moddevuppertail}
There exists absolute constants $C_1$, $s_0$ and $t_0>0$ such that for all $t>t_0$ and $s>s_0$, the following holds.

\begin{equation}
\label{e:moddevuppertail}
\P[X_{t}\geq 2\sqrt{t}+st^{1/6}]\leq e^{-C_1 s^{3/2}}.
\end{equation}
\end{theorem}

The corresponding estimate for the lower tail was proved in \cite{LMS02}, the following theorem is an immediate corollary of Theorem 1.2 from that paper.

\begin{theorem}
\label{t:moddevlowertail}
There exists absolute constants $C_1$, $s_0$ and $t_0>0$ such that for all $t>t_0$ and $s>s_0$, the following holds.

\begin{equation}
\label{e:moddevlowertail}
\P[X_{t}\leq 2\sqrt{t}-st^{1/6}]\leq e^{-C_1 s^{3/2}}.
\end{equation}
\end{theorem}

Observe that $t_0$, $s_0$ and $C_1$ can be taken to be same in Theorem \ref{t:moddevuppertail} and Theorem \ref{t:moddevlowertail}.
It is also clear by the translation invariance of the Poisson process that the same bounds can be obtained for the the number of points on a maximal increasing path on any pair of points that determine a rectangle with area $t$.

\textbf{Remark:}
Observe that the result as stated in Theorem \ref{t:moddevlowertail} is not optimal. Comparing with the tail of the Tracy-Widom distribution, one expects an exponent of $s^{3/2}$ for the upper tail and an exponent $s^{3}$ for the lower tail. Indeed the result from \cite{LMS02} gives the optimal bound for a certain range of $s$, but we do not need it in our work. The optimal tail estimates have also been obtained by Riemann-Hilbert problem approach in certain other KPZ models \cite{BXX01, CLW15}.

%
%

\subsubsection{Transversal Fluctuation}
Consider all increasing paths $\gamma$ from $(0,0)$ to $(n,n)$ in $\Pi$ containing the maximum number of points. From now on we shall often interpret a maximal paths as a piecewise linear function $\gamma: [0,n]\to [0,n]$. The maximum transversal fluctuation $F_n$ is defined as $\max_{x\in [0,n],\gamma}|\gamma(x)-x|$. The scaling exponent for the transversal fluctuation $\xi$ is defined by

$$\xi=\inf\{\theta >0: \liminf_{n}\P[F_{n}\geq n^{\theta}] = 0\}.$$
Johansson \cite{J00} proved the following theorem.
\begin{theorem}
In the above set-up we have $\xi=\frac{2}{3}$.
\end{theorem}

This theorem bounds the maximal fluctuation of the maximal paths from the diagonal as having order $n^{2/3 + o(1)}$. This motivates a lot of our constructions. However for our proof, we need a slightly sharper estimate which we establish using Theorem \ref{t:moddevlowertail} and Theorem \ref{t:moddevuppertail} (see Theorem \ref{t:transversal}).


\subsection{Outline of the proof}
\label{s:outline}
In this subsection we present an outline of the proof  for the  case of the Poissonian last passage percolation.
The proof of Theorem~\ref{t:maintheoremdlpp} follows similarly (see \S~\ref{s:discrete} for details).

We start with the observation that due to superadditivity of the passage times
$$
\mathbb E L^\lambda_{n+m} \geq \mathbb E L^\lambda_{n} + \mathbb E L^\lambda_{m}
$$
for any $\lambda > 0$, so it suffices to prove that for some $n$ we have $\mathbb E [L_n^{\lambda}] > 2n$.
Using the Tracy-Widom  limit from Theorem \ref{BDJ99}  and the moderate deviation  inequalities
from  Theorems~\ref{t:moddevuppertail} and~\ref{t:moddevlowertail} we have
\begin{equation}
\mathbb E [L_n] \geq 2n - 4 n^{1/3}.
\end{equation}
Thus, it is sufficient to prove that for $n$ sufficiently large
\emph{
\begin{align}
\label{e:impr2}
& {\text{reinforcing the diagonal by a rate $\lambda$ Poisson point process increases the length of}} \notag \\
& {\text{longest increasing  path from $(0,0)$ to $(n,n)$ by at least $5n^{1/3}$ in expectation.}}
\end{align}
}
Let us consider two  ways in which the reinforced environment may improve upon the original unperturbed environment.  Based on the transversal fluctuation exponent, the maximal path in the unperturbed environment is expected to spend  $O(n^{1/3})$  of time within distance 1 of
the diagonal.  Then, in expectation, the length of this path should increase by $O(\lambda n^{1/3})$ using only small local changes in the path.
Now consider a second scenario where the maximal path deviates from the diagonal for a macroscopic time.  Suppose further that an alternative path exists which differs in length from the maximal path by only $\delta n^{1/3}$ and spends $c n^{1/3}$ more time close to the diagonal.  This event can be shown to occur with constant probability.  If $0<\delta \ll c\lambda$ then in the reinforced environment the alternative path will make use of more points along the diagonal and be $O(\lambda n^{1/3})$ longer.  Thus we have identified  improvements of $O(\lambda n^{1/3})$ originating from changes to the path on both  the shortest and longest length scales.

In order to establish \eqref{e:impr2} we need an improvement of $5 n^{1/3}$ instead of $O(\lambda n^{1/3})$ so we look for improvements on all length scales between $1$ and $n$.
However, this task  becomes complicated since we do not have a good picture of the distributions of excursions from the diagonal of intermediate sizes.  In light of this limitation, we instead consider the question of reinforcing along translates of the diagonal $\ell_{m} = \{(x,y): y= x + m\} $.  We will do this randomly, along $\ell_{\mathfrak{M}}$  where $\mathfrak{M}$  uniformly distributed in $[-Kn^{2/3}, Kn^{2/3}]$.  This shifts our frame of reference from the excursions of the maximal path away from the diagonal to the local behaviour of the path, which we examine at a range of different length scales.

Let $L^{\lambda, m}_n$ be the length of the longest increasing path from $(0,0)$ to
$(n,n)$ in the environment  reinforced along $\ell_m$. Since $L^{\lambda, m}_n$ is itself superadditive for all fixed $m$, it is enough to show that for some $m$ we have $\mathbb  E [L^{\lambda, m}_n] > 2n$.  Hence it will suffice if for some $n$,
\begin{equation} \label{e:dream}
\mathbb  E [L^{\lambda, \mathfrak{M}}_n] > 2n.
\end{equation}

To obtain (\ref{e:dream})  we analyze the unperturbed
environment at a range of different  scales.
For fixed length scale $r=10^i$ and  spatial location $x=kr, \; k \in [n/r]$ we consider the trajectory of the maximal path $\Gamma$ in $B_{r,x} := [kr, (k+1)r ) \times [0,n]$ and say that a good alternative exists if the following all hold:
\begin{enumerate}
\item Denoting $(x,y)$ as the point $\Gamma$ enters $B_{r,x}$, the transversal fluctuations of $\Gamma$ from $\ell_{y-x}$ is at most $\frac{M}{2} r^{2/3}$.
\item There exists an alternative path $\Gamma^*$ which coincides with $\Gamma$ outside $B_{r,x}$ such that the length of $\Gamma^*$ is only $\delta r^{1/3}$ less than $\Gamma$.
\item The path $\Gamma^*$ has a segment of length at least $c r$ in $B_{x,r}$ between the lines $\ell_{y-x-2M r^{2/3}}$ and $\ell_{y-x-M r^{2/3}}$.
\end{enumerate}
The main work of the proof is to show that for most locations $x$, a good alternative path exists with probability at least $p(\lambda,\delta,c)>0$.

As a consequence of Condition 3 conditional on $\mathfrak{M}\in [y-x-2Mr^{2/3},y-x-Mr^{2/3}]$, the effect of reinforcement increases $\Gamma^*$ by $\frac{c' \lambda r}{r^{2/3}}$ on average. Provided that $c' \gg \delta$ then, conditional on a good alternative and $\mathfrak{M}\in [y-x-2Mr^{2/3},y-x-Mr^{2/3}]$, $\Gamma^*$ improves on the original $\Gamma$ by $\tilde{c} \lambda r^{1/3}$.  Summing over all $n/r$ locations for $x$ at scale $r$ we have a total improvement of $p(\lambda,\delta,c) \cdot \frac{n}{r} \cdot \tilde{c} \lambda r^{1/3} \cdot \frac{M r^{2/3}}{2Kn^{2/3}} = \frac{p(\lambda,\delta,c) \tilde{c} \lambda M}{2K} n^{1/3}$.

To boost the total expected improvement to $5 n^{1/3}$ we take improvements over a range of scales $r$.  Since we chose $r$ to be exponentially growing, a combination of Conditions 1 and 3 ensure geometrically that the use of alternative paths on one scale do not interfere with those on other scales.  With a large constant number of scales we establish (\ref{e:dream}).

The fact that a good alternative path exists with probability independent of the scale is motivated by the self-similar scaling of the process.  The principal difficulty in the proof is the effect of the conditioning in analysing the neighbourhood of the environment around the maximal path.  Our approach is geometric based on two main tools.  One is a series of percolation arguments showing that the neighbourhood of the maximal path must be ``typical'' in most locations and scales.  The second is the use of the FKG inequality since conditioning on the trajectory of the maximal path is a negative event on the remaining configuration.  This is used to show that with some probability there exist barriers around the path which force all alternative paths to be local.  Having localized the problem we show that an alternative path with the prescribed properties exists with probability independent of the scale.

\bigskip

{\bf Organisation of the paper:}
The rest of this paper is organised as follows. As mentioned before we shall provide details only for the proof of Theorem \ref{t:maintheoremppp} while pointing out the adaptations needed for the proof of Theorem \ref{t:maintheoremdlpp}. We start with setting up the notations and terminology in \S~\ref{s:prelim}. In \S~\ref{s:events} we define for a fixed scale $r$, and a fixed location $x$, events $G_x$, $H_x$ and $R_x$ which are key to the construction of an alternative path as explained above, we also explain how we condition on $R_x$. Using estimates of probabilities of these events (Theorem \ref{l:rdecreasing} and Theorem \ref{p:gxhxconditional} whose proofs are deferred until later), in \S~\ref{s:locsuccess}, we show that with a probability bounded uniformly away from 0, an alternative path satisfying the necessary conditions exists which deviates from the topmost maximal path only near $x$. This is the heart of the argument. Using this, and adding extra points on different offset diagonals as explained above, we complete the proof of Theorem \ref{t:maintheoremppp} in \S~\ref{s:finish}. In \S~\ref{s:maxpathnice}, we work out certain percolation-type estimates showing that the maximal path behaves sufficiently regularly at a typical location. Probability bounds on $G_x$ are proved in \S~\ref{s:gx}, and for $H_x$ and $R_x$ in \S~\ref{s:rxcond} which ultimately finishes the proof of Theorem \ref{l:rdecreasing} and Theorem \ref{p:gxhxconditional}. Throughout these proofs we use a number of results, which are consequences of the moderate deviation estimates Theorem \ref{t:moddevlowertail} and Theorem \ref{t:moddevuppertail}. For convenience, we have organized these results in \S~\ref{s:moddev}, \S~\ref{s:para}, \S~\ref{s:trans} and \S~\ref{s:const}. However they are quoted throughout the paper. Finally in \S~\ref{s:discrete} we briefly describe how to modify the arguments for the Poissonian last passage percolation case to prove Theorem \ref{t:maintheoremdlpp}.

\section{Notations and Preliminaries}
\label{s:prelim}
In this section, we introduce certain notations for the Poissonian last passage percolation model. The same notations with minor modifications can be used for the Exponential directed last passage percolation model also, see \S~\ref{s:discrete} for details of the Exponential case.

%

\subsection{Path, length and area}
Define the partial order $<$ on $\R^2$ by $u=(x,y)< u'=(x',y')$ if $x<x'$, and $y<y'$. For $u<u'\in \R^2$, an \emph{increasing path} $\gamma$ from $u$ to $u'$ is a piecewise linear path joining a finite sequence of points $u=u_0<u_1<\cdots < u_k=u'$. For $u_0=(x_0,y_0)<u_{k}=(x_k,y_k)$ and an increasing path $\gamma$ from $u_0$ to $u_k$, and for $x_0\leq x\leq x_k$, let $\gamma_{x}$ be such that $(x,\gamma_x)\in \gamma$. Notice that $\gamma_x$ is uniquely defined. We shall sometimes identify the path with the sequence of points that define it.

We define the length of an increasing path with respect to a background point configuration on $\R^2$. Let $\Omega$ be a  point configuration on $\R^2$. Consider an increasing path $\gamma$ from $u$ to $u'$ given by $\gamma=\{u=u_0<u_1<\cdots < u_k=u'\}$. Then \emph{length} of $\gamma$ in $\Omega$, denoted $\ell_{\gamma}^{\Omega}$ is defined by

$$\ell_{\gamma}^{\Omega}= \#\{0\leq j <k: u_{j}\in \Omega\}.$$

Notice that, in the above definition, for definiteness, we count the starting point of the path, but not the end point. 

For $u<u'$ in $\R^2$, let $A(u,u')$ denote the area of the rectangle $\mbox{Box}(u,u')$ with bottom left corner $u$ and top right corner $u'$. For an increasing path $\gamma$ containing $u$ and $u'$, let $\gamma(u,u')$ denote the restriction of $\gamma$ between $u$ and $u'$. Let $\gamma(u,u')=\{u=u_0<u_1<\cdots < u_k=u'\}$. For a given environment  $\Omega$  let $i_{1}<i_{2}<\cdots i_{\ell} \in [k-1]$ be such that $u_{i_j}$ are all the points on $\gamma\cap \Omega$ (ignoring the end points of $\gamma$). Set $i_0=0$ and $i_{\ell+1}=k$. Then the \emph{region} of $\gamma(u,u')$ in the environment $\Omega$, denoted $O_{\gamma}^{\Omega}(u,u')$, is defined to be the union of the rectangles $\mbox{Box}(u_{i_j},u_{i_{j+1}})$ for $j\in \{0,1,\ldots ,\ell\}$. The \emph{area} of the path $\gamma(u,u')$ in the environment $\Omega$, denoted $A_{\gamma}^{\Omega}$ is the area of the region $O_{\gamma}^{\Omega}(u,u')$, i.e.,
$$A_{\gamma}^{\Omega}=\sum_{j=0}^{\ell} A(u_{i_j},u_{i_{j+1}}).$$

We shall drop the superscript $\Omega$ if the environment is clear from the context.

\subsection{Statistics of the Unperturbed Configuration}
We let $\Pi$ denote a rate 1  Poisson process on $\mathbb{R}^2$ which we refer to   as the \emph{unperturbed configuration}, i.e.\ without reinforcements.

\begin{itemize}
\item
For $u,u'\in \R^2$, let $X_{u,u'}$ denote the length of longest increasing path in $\Pi$ from $u$ to $u'$. While the longest increasing path need not be unique,  $X_{u,u'}$ is well defined.
\end{itemize}

For $u=(x,y)<u'=(x',y')$ in $\R^2$, let $d(u,u')=(x'-x)+(y'-y)$ be the $\ell_1$ distance between $u$ and $u'$. It will be useful for us to consider following centered versions of $X_{u,u'}$.

\begin{itemize}
\item Let $$\tilde{X}_{u,u'}=X_{u,u'}-\E X_{u,u'}.$$

\item Let
$$\hat{X}_{u,u'}=X_{u,u'}-d(u,u').$$
\end{itemize}
Observe that, by Theorem \ref{BDJ99} and superadditivity, it follows that $\E[X_{u,u'}] \leq d(u,u')$. The reason behind the choice of centering by $d(u,u')$ is the following. If the straight line joining $u$ and $u'$ has slope very close to $1$, then $d(u,u')$ gives the right centring up to first order. Also observe that for $u_1<u_2<\cdots <u_k$ we have $\sum_{i=1}^{k-1} \hat{X}_{u_i,u_{i+1}}\leq \hat{X}_{u_1,u_k}.$

\subsubsection{Statistics of constrained paths}


We define the following notations for paths subject to certain constraints.
\begin{itemize}
\item For $u,u'\in \R^2$ with $u<u'$, and $S\subseteq \R^2$, we define $X^S_{u,u'}$ to be the length of the longest increasing path from $u$ to $u'$ that does not go through $S$. The centered length is denoted by $\tilde{X}_{u,u'}^{S}$, i.e., $\tilde{X}_{u,u'}^{S}=X^S_{u,u'}-\E X_{u,u'}$. Similarly we also define $\hat{X}_{u,u'}^{S}$.

\item  For $u,u'\in \R^2$ with $u<u'$, and $S\subseteq \R^2$, we define $^SX_{u,u'}$ to be the length of the maximal increasing path from $u$ to $u'$ that intersects the set $S$. We define $^S\tilde{X}_{u,u'}$ and $^S\hat{X}_{u,u'}$ similarly.
\end{itemize}


\subsection{Choice of Parameters}
Throughout the proof we shall make use of a number of parameters which need to satisfy certain constrains among themselves. We record here the parameters used, the relationship between them, and the order in which we need to fix them. The precise values of the parameters will not be of importance to us.

\textbf{Reinforcement parameter $\lambda$:} $\lambda>0$ will be kept fixed throughout the proof, this is the rate at which the diagonal $\{x=y\}$ (and its translates) are reinforced.

\textbf{Scale $r$:} As explained in the introduction, we shall work out estimates for functions of $\Pi$ at different length scales, the scale will be indexed by $r$. Let $$\mathcal{R}=\biggl\{10^k \frac{n}{\log^{10} n} : 1\leq k\leq \frac{1}{100}\log \log n\biggr\}.$$ We shall take $r$ to be one of the elements of $\mathcal{R}$.

\textbf{Parameters:} We choose the parameters in the following order. All these parameters are positive numbers and are independent of $r\in \mathcal{R}$, but they can depend on $\lambda$.
\begin{enumerate}
\item $\psi$ will be an absolute constant sufficiently large.

\item $\eta$ will be an absolute positive constant sufficiently small.

\item We choose $\tilde{C}$ sufficiently large depending on $\psi$.

\item We choose $M$ sufficiently large depending on other parameters chosen so far.

\item The parameter $C$ will be a sufficiently large constant depending on $M$.
\item  $0<\alpha'<1$ is chosen to be sufficiently small constant depending on $M$.

\item $\rho$ is chosen sufficiently small depending on $C$.

\item $\delta<1$ is chosen to be sufficiently small depending all other constants chosen so far (and $\lambda$).

\item We choose $\varepsilon$ small enough depending on $C$ and $M$ and $\delta$ and $\rho$.

\item $L$ is chosen sufficiently large depending on $\delta$ and $\varepsilon$.

\item $C^*$ is chosen sufficiently large depending on all other parameters.
\end{enumerate}

The functional form of the constraints that these parameters will need to satisfy will be specified later on. Without loss of generality we shall also assume that $n$ is an integer multiple of $r$ and $L^{3/2}r$ which will be convenient for some of our estimates. Also $\log$ will always denote natural logarithm unless mentioned otherwise.


\section{Defining the Key Events}
\label{s:events}
As explained in the introduction, we shall define some key events on which we shall be able to obtain local modifications of the longest path which will lead to improvements in the reinforced environment. These events will be defined for different locations in each scale $r$.

For the rest of this section, let $r\in \mathcal{R}$ be fixed. All of our events will be defined for this fixed $r$.

\subsection{Geometric Definitions: The $(x,y,r)$-Butterfly}
For a fixed $r$, let
$$\mathcal{X}_{r}=\left\{(k+\frac{1}{2})r: k\in \{\frac{n}{3r}, \frac{n}{3r}+1, \ldots , \frac{2n}{3r}-1\}\right\}.$$ For a fixed $x\in \mathcal{X}_r$ and for a fixed $y\in r^{2/3}\Z$ we define a geometric object, which we shall call the {\bf $(x,y,r)$-butterfly}, denoted as $\mathbb{B}(x,y,r)$, which will be a union of parallelograms as described below.

First we need the following notation. For $(x,y)\in \R^2$, $\ell,h \geq 0$, let $\mathcal{P}(x,y,\ell,h)$ denote the parallelogram whose corners are given by $(x-\frac{\ell}{2}, x-\frac{\ell}{2}+y)$, $(x-\frac{\ell}{2}, x-\frac{\ell}{2}+y+h)$, $(x+\frac{\ell}{2}, x+\frac{\ell}{2}+y)$, $(x+\frac{\ell}{2}, x+\frac{\ell}{2}+y+h)$. Unless otherwise mentioned our parallelograms will always be closed.

The butterfly $\mathbb{B}(x,y,r)$ consists of the following parallelograms.

\begin{itemize}
\item The {\bf body} of the $\mathbb{B}(x,y,r)$ is the parallelogram $$T=T_{x,y,r}:=\mathcal{P}(x,y-Lr^{2/3},r, (L+M)r^{2/3}).$$

\item The {\bf left wing} ${W}^1$  and the {\bf right wing} ${W}^2$ of $\mathbb{B}(x,y,r)$ is defined as follows
$${W}^1={W}^1_{x,y,r}=\mathcal{P}(x-\frac{r}{2}(1+L^{3/2}), y-L^{11/10}r^{2/3}, L^{3/2}r, 2L^{11/10}r^{2/3})$$

and $${W}^2={W}^2_{x,y,r}=\mathcal{P}(x+\frac{r}{2}(1+L^{3/2}), y-L^{11/10}r^{2/3}, L^{3/2}r, 2L^{11/10}r^{2/3}).$$
\end{itemize}

The $(x,y,r)$-butterfly is defined as $$\mathbb{B}:=T\cup {W}^1 \cup W^2.$$
Notice that the $(x,y,r)$-butterfly implicitly depends on the parameters $L$ and $M$ which are chosen later satisfying the constraints described above.

We further define some important subsets of the butterfly (omitting the subscript $(x,y,r)$).

\begin{itemize}
\item Let $\mathfrak{C}=\mathcal{P}(x,y-Lr^{2/3},\frac{4r}{5},(L+M)r^{2/3})$. We shall call $\mathfrak{C}$ the \emph{central column} of $\mathbb{B}$.

\item Let $D=\mathcal{P}(x,y-(M+\frac1{10})r^{2/3}, \frac{r}{10}, \frac{r^{2/3}}{10})$.

\item Let $\Lambda=\mathcal{P}(x,y-2Mr^{2/3},\frac{4r}{5},3Mr^{2/3}).$

\item Let $$B_1^*=\mathcal{P}(x-\frac{9r}{20},y-Lr^{2/3}, \frac{r}{10},(L+M)r^{2/3})$$
and
$$B_2^*= \mathcal{P}(x+\frac{9r}{20},y-Lr^{2/3}, \frac{r}{10},(L+M)r^{2/3}).$$

We shall call $B_1^{*}$, $B_2^{*}$ {\bf barriers} of the butterfly $\mathbb{B}$.

\item Let $F=\mathcal{P}(x,y-Lr^{2/3},r,0)$ be called the {\bf floor} of the butterfly $\mathbb{B}$ and let $F^{+}$ denote the region in $\mathbb{B}$ above $F$.
\end{itemize}

Different parts of the anatomy of the butterfly are illustrated in Figure \ref{f:butterfly}. This and other figures we use in this paper are drawn in the tilted coordinate $(x',y')=(x,y-x)$ in which the parallelograms with one pair of sides parallel to the line $x=y$ and other pair of sides parallel to the $y$-axis ( e.g. the parallelograms constituting a butterfly) become rectangles with sides parallel to the axes. This is merely a convenience in drawing and does not have any other significance.

\begin{figure*}[h]
\begin{center}
\includegraphics[width=0.6\textwidth]{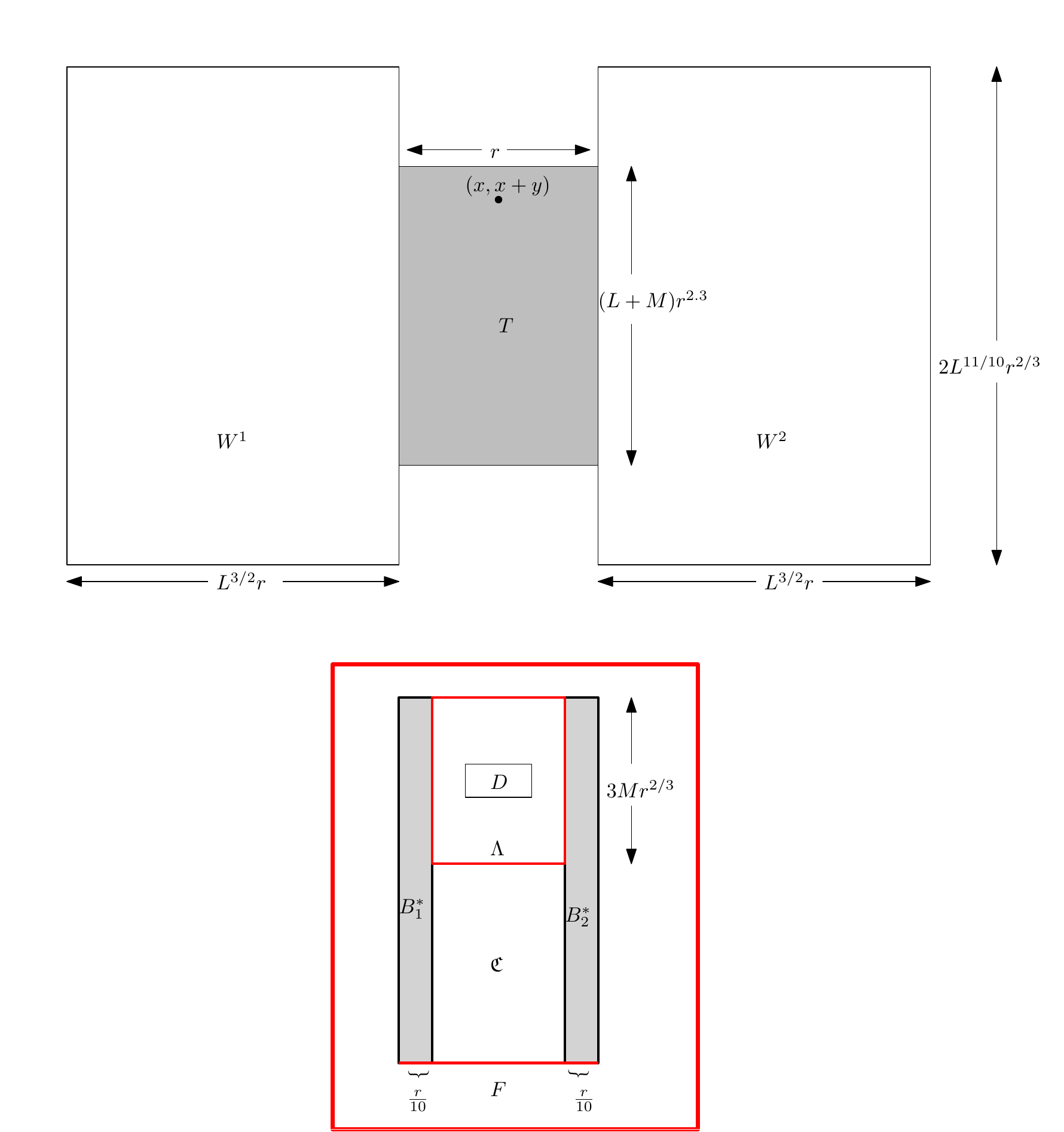}
\caption{Anatomy of a butterfly $\mathbb{B}(x,y,r)$. The figure above shows the parallelograms $W^1,W^2$ and $T$ making up the butterfly. The inset figure below illustrated the detailed anatomy of the body $T$.}
\label{f:butterfly}
\end{center}
\end{figure*}

\subsection{Defining the event $G_{x,y}$:}
Now we are ready to define an event $G_{x,y}$ for $x\in \mathcal{X}_r$ and $y\in r^{2/3}\Z$, which is one of the key events in our proof. We shall say $G_{x,y}$ holds if a long list of conditions are satisfied. For convenience we have divided the conditions into the a number of parts.

Notice that $G_{x,y}$ will consist of conditions that are \textbf{typical}, and for our purposes $G_{x,y}$ will be a good event that holds with large probability. A typical condition in the definition of $G_{x,y}$ will be as follows: for a parallelogram (in the butterfly) we shall ask that for all pairs of points in the parallelogram such that the slope of the line joining them is bounded away from $0$ and $\infty$ the length of the maximal path between these two points is neither too large nor too small (i.e., has on scale fluctuations).
 For our purposes we shall need to consider the above condition (or some variant) for a number of different parallelograms, it might be useful to think about them as \textit{good} parallelograms. Some of these good parallelograms have been illustrated in Figure \ref{f:goodpara}.

\begin{figure*}[!ht]
\begin{center}
\includegraphics[width=0.8\textwidth]{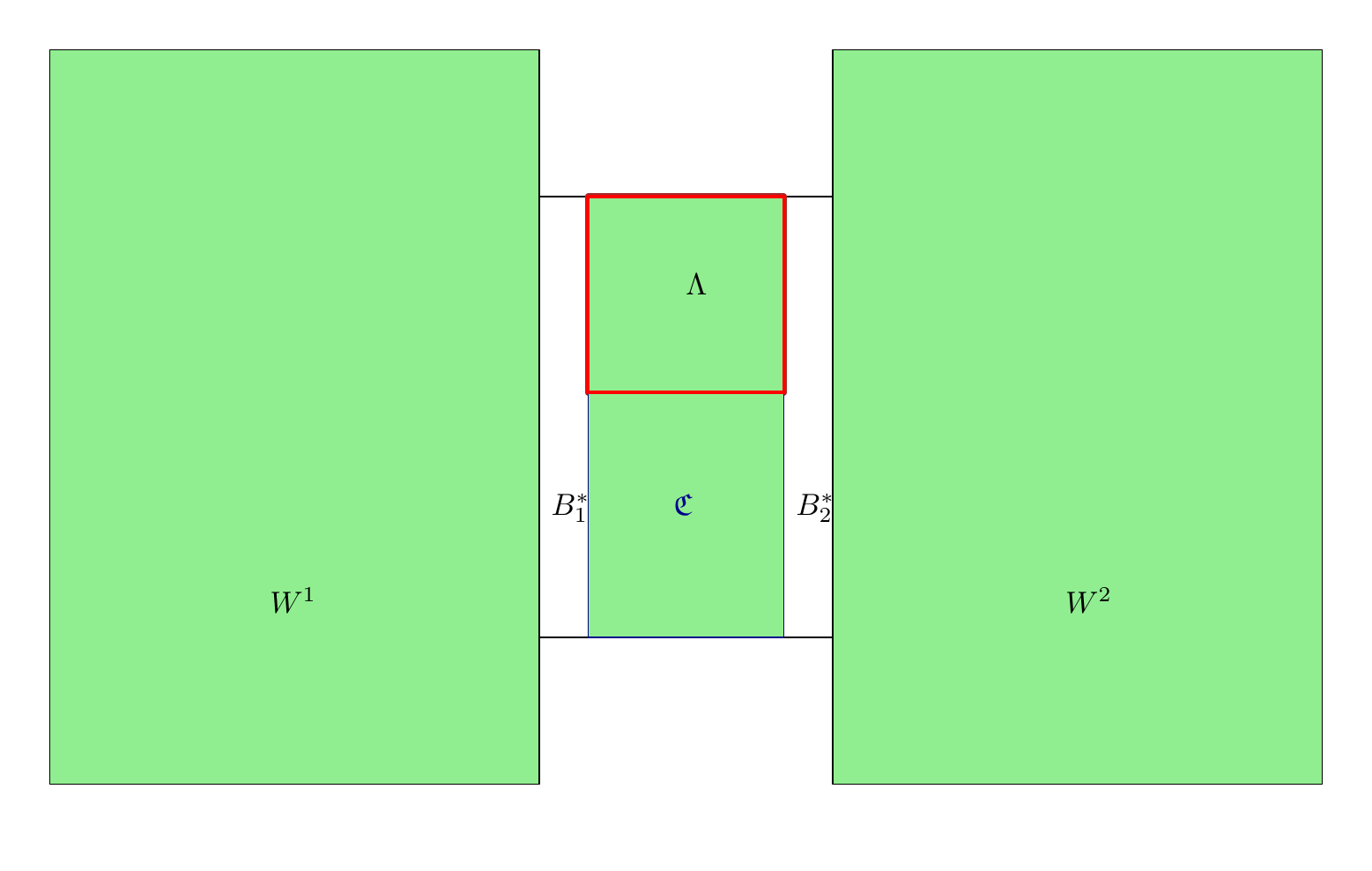}
\caption{Some of the good parallelograms in $\mathbb{B}(x,y,r)$ as described in the conditions defining $G(x,y)$ are marked in green. Notice that even though $\Lambda$ is contained in $\mathfrak{C}$, we need to ask them to be good separately, as $\mathfrak{C}$ has a much larger width and hence the on-scale fluctuations for $\mathfrak{C}$ is much larger than that of $\Lambda$.}
\label{f:goodpara}
\end{center}
\end{figure*}

To state the above condition formally, we shall use the following notation. For a region $U\subseteq \R^2$, we define $\mathcal{S}(U)=\mathcal{S}^{\psi}(U)\subseteq U^2$ as follows. For $u=(x,y)$ and $u'=(x',y')\in U$, $(u,u')\in \mathcal{S}(U)$ iff $\frac{2}{\psi}< \frac{y'-y}{x'-x}\leq \frac{\psi}{2}$.

\subsubsection{The local conditions: $G_{x,y}^{\rm loc}$}
We say $G_{x,y}^{\rm loc}$ holds if the following conditions are satisfied.

\begin{enumerate}
\item
Let $U=\mathcal{P}(x,y-Mr^{2/3}/10,r,2Mr^{2/3}/10)$. For all $(u',u'')\in \mathcal{S}(U)$ we have
\begin{equation}
\label{e:gxloc1}
|\tilde{X}_{u',u''}^{\partial U}|\leq Cr^{1/3}.
\end{equation}

\item We have $\forall (u,u') \in \mathcal{S}(\mathfrak{C})$
\begin{equation}
\label{e:gxloc2}
|\tilde{X}_{u,u'}|\leq CL^{1/2}r^{1/3}.
\end{equation}

\item For all $u,u'\in \mathcal{S}(\Lambda)$ we have
\begin{equation}
\label{e:gxloc3}
|\tilde{X}_{u,u'}|\leq Cr^{1/3}.
\end{equation}

Also let $2D$ denote the dilation that doubles $D$ keeping the centre fixed.  Then we have for all $u,u'\in \mathcal{S}(\Lambda\setminus 2D)$, such that both of $u$ and $u'$ are not in $\{(x^{*},y^{*}):|x-x^*|\leq r/5\}$

\begin{equation}
\label{e:gxloc4}
|\tilde{X}_{u,u'}^D|\leq Cr^{1/3}.
\end{equation}


\item We have $\forall u\in \mathcal{P}(x,y-2Mr^{2/3}, \frac{r}{5}, 2Mr^{2/3})$, $u'\in B_1^*\cap \mathfrak{C}$ (resp.\ $u'\in B_2^*\cap \mathfrak{C}$) (i.e., $u'$ is in the boundary between the barriers and the central column)

\begin{equation}
\label{e:gxloc5}
\hat{X}_{u',u}\leq Cr^{1/3}~ (\text{resp.}~ \hat{X}_{u,u'}\leq Cr^{1/3}).
\end{equation}


\item
We have $\forall u,u'\in F$
\begin{equation}
\label{e:gxloc6}
|\tilde{X}_{u,u'}^{F^{+}}|\leq Cr^{1/3}.
\end{equation}
\end{enumerate}

\subsubsection{The Area Condition: $G_{x,y}^{\rm a}$}
Consider the parallelogram $U=\mathcal{P}(x,y-\frac{3Mr^{2/3}}{2},\frac{4r}{5}, Mr^{2/3})$. We say that $G_{x,y}^{\rm a}$ holds if for all paths $\gamma$ from $u=(x,y)$ to $u'=(x',y')$ with $u<u'\in U$, $|x-x'|\geq \alpha' r$ with $\alpha' r \leq  \ell_{\gamma}^{\Pi} \leq 3r$, we have
\begin{equation}
\label{e:gxarea}
A_{\gamma}^{\Pi} \geq \alpha' \eta r.
\end{equation}
Notice that this condition depends on all the parameters in our construction.



\subsubsection{Resampling Condition: $G_{x,y}^{\rm rs}$}

Taking equally spaced line segments parallel to its sides we divide the parallelogram $D$ into a grid of  $\frac{1}{100\varepsilon^{5/3}}$ many parallelograms of size $\varepsilon r\times (\varepsilon r)^{2/3}$ each.  We denote these by $D_1,D_2,\ldots $ so $D=\cup_{i} D_i$. Let $u_i=(x_i,y_i)$ denote the bottom left corner of $D_i$. Define the parallelogram $\tilde{D}_i$ whose corners are $(x_i,y_i-\varepsilon^{2/3}r^{2/3})$, $(x_i+\varepsilon r, y_i-\varepsilon^{2/3}r^{2/3}+\varepsilon r)$, $(x_i, y_i+2\varepsilon^{2/3}r^{2/3})$ and $(x_i+\varepsilon r, y_i+2\varepsilon^{2/3}r^{2/3}+\varepsilon r)$. Parallelograms $D_i$ and $\tilde{D}_{i}$ are illustrated in Figure \ref{f:resamp}.

\begin{figure*}[!ht]
\begin{center}
\includegraphics[width=0.4\textwidth]{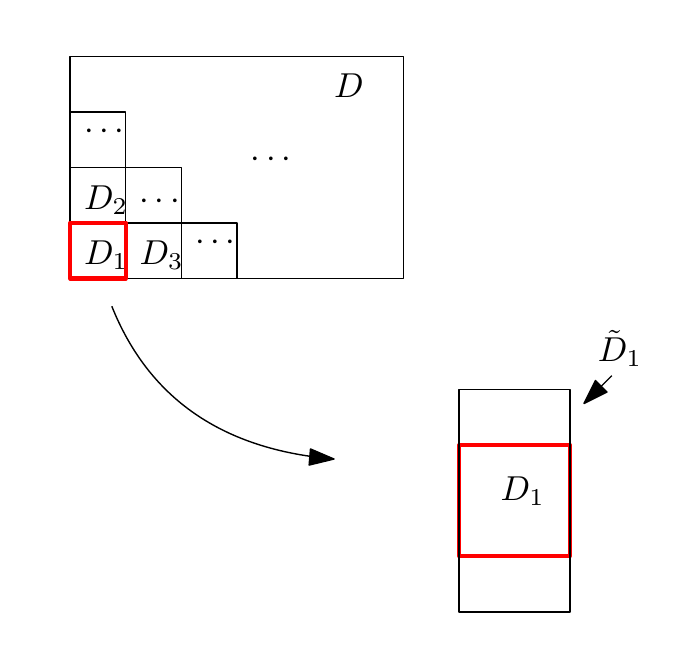}
\caption{Parallelograms $D_i$ and $\tilde{D}_i$ as defined in the resampling condition}
\label{f:resamp}
\end{center}
\end{figure*}

Now let $\Pi^{*}$ be another i.i.d.\ copy of $\Pi$. Let $\Pi^{(*,i)}$ denote the point process obtained by replacing the the point configuration of $\cup_{h=1}^{i} D_h$ in $\Pi$ by the corresponding point configuration in  $\Pi^{*}$. From now on, whenever we write some statistics of a point configuration with a superscript $(i)$, this will denote the statistic for the point configuration $\Pi^{(*,i)}$.

Let $$\Delta_i=\sup_{u,u'\in \mathcal{S}(\tilde{D}_i)} |X_{u,u'}^{(i-1)}-X_{u,u'}^{(i)}|.$$

We say $G_{x,y}^{\rm rs}$ holds if the following condition is satisfied.

\begin{itemize}
\item
\begin{equation}
\label{e:gxrs}
\P\left[\max_{i} \Delta_i \leq \frac{\delta r^{1/3}}{2}\mid \Pi\right]\geq 1-e^{-C/\varepsilon^{1/4}}.
\end{equation}
\end{itemize}

Notice that all the above conditions can be checked by looking at the point configuration in $[x-r/2,x+r/2]\times \R$. i.e., these events will be independent for different values of $x$.

\subsubsection{The Wing condition: $G_{x,y}^{w}$}
We say $G_x^{w}$ holds if the following condition is satisfied.
We have $\forall u,u'\in \mathcal{S}({W}^i)$ for $i=1,2$
\begin{equation}
\label{e:gxnbhd2}
|\tilde{X}_{u,u'}|\leq CL^{3/4}r^{1/3}.
\end{equation}

\begin{itemize}
\item Finally we define
$$G_{x,y}= G_{x,y}^{\rm loc}\cap G_{x,y}^{w}\cap G_{x,y}^{\rm rs}\cap G_{x,y}^{\rm a}.$$
\end{itemize}

\subsection{Defining the event $G_x$:}
\label{s:gxdef}


Let $\Gamma$ be the topmost maximal path in $\Pi$ from $\mathbf{0}=(0,0)$ to $\mathbf{n}=(n,n)$. For $x\in \mathcal{X}_r$ we define the event $G_x$ as follows. Let $y^*=y(x,\Gamma)=\inf\{y\in r^{2/3}\Z: x+y\geq \Gamma_{x}\}$. We shall denote $\mathbb{B}(x,r)=\mathbb{B}(x,y^*,r)$. Also let
$$B_1= \{(x',y')\in B_1^*: x-r/2\leq x' \leq x-2r/5, x'+y^*-Lr^{2/3}\leq y'< \Gamma_{x'}\};$$
$$B_2= \{(x',y')\in B_2^*: x+r/2\geq x' \geq x+2r/5, x'+y^*-Lr^{2/3}\leq y'< \Gamma_{x'}\};$$
For $x\in \mathcal{X}_r$, $i=1,2$, we shall call $B_i=B_i(x,r)$ \emph{walls} in the column $x$. Also for an increasing path $\gamma$ from $(0,0)$ to $(n,n)$, $B_i(x,\gamma,r)$ will be defined similarly, replacing $\Gamma$ by $\gamma$. See Figure \ref{f:wall}.

\begin{figure*}[!ht]
\begin{center}
\includegraphics[width=0.4\textwidth]{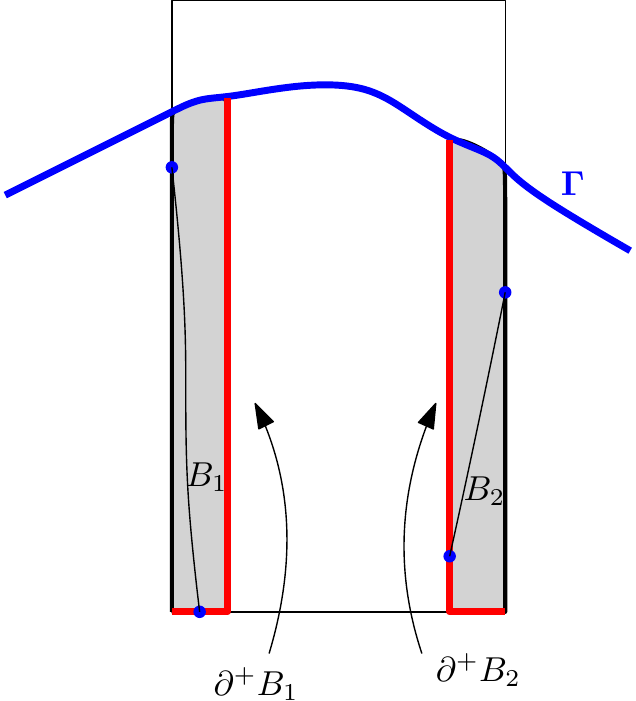}
\caption{Geometry of Walls: Walls $B_1$ and $B_2$ as defined in Section \ref{s:gxdef} where $\Gamma$ is the topmost maximal path. We have marked in red $\partial ^{+} B_1$ and $\partial ^{+} B_2$ as defined in Section \ref{s:rx} in red. Pairs of blue points denote some typical pairs of points which are asked to have not too large distance (i.e.\ length of maximal path from one to the other) in the definition of $R_{x}$ (first condition). In the second condition in definition of $R_{x}$, we stipulate that any path across either wall (i.e., one of the grey regions in the figure) will have much smaller than typical length}
\label{f:wall}
\end{center}
\end{figure*}

We say that $G_x$ holds if all of following conditions are satisfied.

\begin{itemize}
\item {\bf The local conditions: $G_{x}^{\rm loc}$:}
We say that $G_{x}^{\rm loc}$ holds if $G_{x,y^*}^{\rm loc}$ holds with the following two modifications.
\begin{itemize}
\item Instead of condition 1 in the definition of $G_{x,y}^{\rm loc}$ above we have that for all $u'=(x',y'), u''=(x'',y'')\in \Gamma$ with $x',x''\in [x-r/2, x+r/2]$ and $|x'-x''|\geq r^{3/4}$ we have
\begin{equation}
\label{e:gxmaxloc1}
|\tilde{X}_{u',u''}|\leq Cr^{1/3}.
\end{equation}

\item We replace $B_1^*(x,y^*,r)$ and $B_2^*(x,y^*,r)$ in condition 4 by $B_1(x,r)$ and $B_2(x,r)$ respectively.
\end{itemize}

\item {\bf The area condition} $G_{x}^{{\rm a}}:=G_{x,y^*}^{{\rm a}}$.

\item {\bf The Wing condition} $G_{x}^{w}:=G_{x,y^*}^{w}$.

\item {\bf The resampling condition} $G_{x}^{\rm rs}:=G_{x,y^*}^{\rm rs}$.

\item {\bf The fluctuation condition: $G_{x}^{f}$:} We say $G_{x}^{f}$ holds if the following conditions are satisfied.

\begin{enumerate}
\item $|\Gamma_x-x|\leq Cn^{2/3}$, 
\item  We have  for $(x',y')\in \Gamma$
\begin{equation}
\label{e:gxmaxnbhd1}
\frac{|(y'-y)-(x'-x)|}{r^{2/3}} \leq
\begin{cases}
\frac{M}{10}~\text{if}~\frac{|x'-x|}{r}\leq 1,\\
L^{11/10}~\text{if}~|x'-x| =(1/2+L^{3/2})r.
\end{cases}
\end{equation}
\end{enumerate}
\end{itemize}

\subsection{Defining $R_{x,\gamma}$, $R_{x,y}$ and $R_{x}$:}
\label{s:rx}

Let $\gamma$ be an increasing path from $\mathbf{0}$ to $\mathbf{n}$. For $x\in \mathcal{X}_r$, define $y(x,\gamma)=\inf\{y'\in r^{2/3}\Z: x+y'\geq \gamma_{x}\}$. Set $B_i=B_i(x,\gamma,r)$. Also let $\partial^{+}(B_i)$  denote the union of $B_i\cap \mathfrak{C}$ and the bottom boundary of $B_i$. See figure \ref{f:wall}. We define $\partial^{+} B_i^{*}$ similarly.
We say $R_{x,\gamma}$ holds if the following conditions are satisfied.

\begin{enumerate}
\item[(i)] We have $\forall u=(x',y')\in B_1$ (resp. $B_2$)  with $y'\geq x'+y(x,\gamma)-Mr^{2/3}$ and $\forall u'\in \partial^{+} B_1$ (resp.\ $\partial ^{+}B_2$)
\begin{equation}
\label{e:rx1}
\hat{X}_{u,u'}^{B_1^{c}\cup \gamma}\leq Cr^{1/3}~(\text{resp.}~\hat{X}_{u',u}^{B_2^{c}\cup \gamma}\leq Cr^{1/3}) .
\end{equation}

\item[(ii)]
We have $\forall u\in B_1\cap W^1$ (resp.  $\forall u\in B_2\cap \mathfrak{C}$) and $\forall u'\in B_1\cap \mathfrak{C}$ (resp.\ $\forall u'\in B_2\cap W^2)$

\begin{equation}
\label{e:rx2}
\tilde{X}_{u,u'}^{B_1^{c}\cup \gamma}\leq -C^*r^{1/3}~(\text{resp.}~\tilde{X}_{u,u'}^{B_2^{c}\cup \gamma}\leq -C^*r^{1/3}).
\end{equation}
\end{enumerate}

The second condition above means that any path that crosses the walls from left to right are much shorter than typical paths. Recall that $C^{*}$ is chosen sufficiently large depending on other parameters.

We also make the following definitions.
\begin{itemize}
\item For $x\in \mathcal{X}_r$ and $y\in r^{2/3}\Z$, let $R_{x,y}$ denote the event such that \eqref{e:rx1} and \eqref{e:rx2} holds in $\mathbb{B}(x,y,r)$ with $B_i$ replaced by $B_i^*(x,y,r)$.

\item We define $R_x:=R_{x,\Gamma}$ where $\Gamma$ is the topmost maximal path from $\mathbf{0}$ to $\mathbf{n}$ in $\Pi$.
\end{itemize}

\subsection{Defining $H_{x,\gamma}$, $H_{x,y}$ and $H_x$:}
Let $\gamma$ be an increasing path from $\mathbf{0}$ to $\mathbf{n}$. For $x\in \mathcal{X}_r$, define $y(x,\gamma)$ as before.
We say $H_{x,\gamma}$ holds if the following conditions are satisfied in the butterfly $\mathbb{B}(x,y(x,\gamma),r)$.

\begin{enumerate}
\item[(i)] For all $u,u'\in F$, we have
\begin{equation}
\label{e:hx1}
^{\Lambda}\tilde{X}_{u,u'}^{\gamma}\leq -Lr^{1/3}.
\end{equation}

\item[(ii)] For all $u\in F$, $u'\in \mathcal{P}(x,y-2Mr^{2/3},r, 3Mr^{2/3})$ and $u'$ below $\gamma$ we have
\begin{equation}
\label{e:hx2}
\hat{X}_{u,u'}^{\gamma} \leq -Lr^{1/3}~\text{or}~\hat{X}_{u',u}^{\gamma} \leq -Lr^{1/3}
\end{equation}
depending on whether $u<u'$ or $u'<u$.
\end{enumerate}
We also make the following definitions.
\begin{itemize}
\item Let $H_{x,y}$ denote the event such that in $\mathbb{B}(x,y,r)$ \eqref{e:hx1} holds without the requirement of avoiding $\gamma$ and \eqref{e:hx2} holds without the requirement of avoiding $\gamma$ or the requirement $u'\in \gamma$.

\item We define $H_x:=H_{x,\Gamma}$ where $\Gamma$ is the topmost maximal path from $\mathbf{0}$ to $\mathbf{n}$ in $\Pi$.
\end{itemize}


%
%
%
%

\subsection{Conditioning on $R_x$}
We want to show that for a fixed $r$, for a large fraction of  $x\in \mathcal{X}_r$, $G_x\cap H_x\cap R_x$ hold with probability bounded away from $0$ uniformly in $r$. It turns out that each of $G_x$ and $H_x$ holds with probability close to $1$, however $R_x$ only holds with a small probability (bounded away from $0$). For this reason, in many of our probabilistic estimates we shall need to condition on $R_x$ for $x\in \mathcal{X}_r$ and deal with the conditional probability measures. For the sake of clarity we shall use the measure $\mu$ for the the measure on configurations $\Pi$ distributed according to a homogeneous Poisson process of rate $1$. The generic notation $\P$ will also refer to this measure unless specified otherwise.


The following theorem gives a lower bound on the probability of $R_x$.

\begin{theorem}
\label{l:rdecreasing}
Let $\gamma$ be an increasing path from $\mathbf{0}$ to $\mathbf{n}$. Let $r\in \mathcal{R}$ be fixed. For $x\in \mathcal{X}_r$ let $A=A_x^{\gamma}:=\R^2\setminus (B_1(x,\gamma,r) \cup B_2(x,\gamma,r))$. Let $\Pi_{A}$ denote the point configuration $\Pi$ restricted to $A$. Then we have
$$\mu(R_{x,\gamma}\mid \Pi_{A}, \Gamma =\gamma)\geq \mu(R_{x,\gamma})\geq \min_{y\in r^{2/3}\Z}\mu(R_{x,y})\geq \beta>0.$$
\end{theorem}

\begin{definition}[Conditional measure]
\label{d:mustar}
Define the measure $\mu_x^{*}$ on configurations in $\R^2$ by conditioning on the configuration in the walls of column $x$ such that $R_{x}$ holds. That is, denoting $A=A_x^{\gamma}$ and $B=B_1(x,\gamma,r)\cup B_2(x,\gamma, r)$ and for point configurations $\Pi_{A}$ restricted to $A$ and $\Pi_{B}^{*}$ restricted to $B$ we have 
$$\mu_x^{*}(\Pi_A, \Pi^*_B):=\sum_{\gamma}\mu(\Gamma=\gamma, \Pi_A)\mu(\Pi^*_B\mid \Gamma=\gamma, \Pi_A, R_{x,\gamma})I(\gamma)$$
where the sum above is over all increasing paths $\gamma$ from $\mathbf{0}$ to $\mathbf{n}$ and $I(\gamma)$ denotes the indicator that $\gamma$ is the topmost maximal path from $\mathbf{0}$ to $\mathbf{n}$ uniquely determined by $\Pi_{A}$ and $\Pi_{B}^{*}$.
\end{definition}

Observe the following mechanism to sample a point configuration from the measure $\mu_{x}^*$. Sample a point configuration $\Pi$ from the measure $\mu$. Notice that $\Gamma, A$ and $B$ are defined as functions of $\Pi$. Write $\Pi=(\Pi_{A},\Pi_{B})$. Now resample the point configuration on $B$ as follows. Draw a configuration $\Pi_{B}^*$ from the Poissonian measure conditioned on the following event: in the configuration $(\Pi_{A}, \Pi^*_{B})$, $\Gamma$ is the topmost maximal path and $R_{x}$ holds. Replace $\Pi_{B}$ by $\Pi^*_{B}$ to obtain a sample from the measure $\mu_x^*$. 

We record the basic properties of $\mu_x^*$ in the following lemma.

\begin{lemma}
\label{l:mustarbasic}
The measure $\mu_x^{*}$ satisfies the following two properties:\
\begin{enumerate}
\item[(i)] We have $\mu_x^*\preceq \mu$ where $\preceq$ denotes stochastic domination.
\item[(ii)] We have
\begin{equation}
\label{e:mustardensity}
\dfrac{{\rm d}\mu_x^{*}}{{\rm d}\mu}\leq \max_{y}\frac{1}{\mu(R_{x,y})}\leq \frac{1}{\beta}.
\end{equation}
\end{enumerate}
\end{lemma}

\begin{proof}
Notice that $(i)$ follows from the FKG inequality and it is clear from definition that the first inequality in (\ref{e:mustardensity}) holds, the second inequality follows from Theorem \ref{l:rdecreasing}.
\end{proof}

Finally we have the following theorem.

\begin{theorem}
\label{p:gxhxconditional}
There exists $\mathcal{X}_r^*\subseteq \mathcal{X}_r$ with $|\mathcal{X}_r^*|\geq \frac{9}{10} |\mathcal{X}_r|$ such that for all $x\in \mathcal{X}_r^{*}$ we have
$$\mu_x^*(G_x\cap H_x)\geq \frac{9}{10}.$$
\end{theorem}

We shall prove Theorem \ref{l:rdecreasing} and Theorem \ref{p:gxhxconditional} over \S~\ref{s:maxpathnice}, \S~\ref{s:gx} and \S~\ref{s:rxcond}. Before that we show how using these two theorems we can prove Theorem \ref{t:maintheoremppp}.

\section{Resampling in $D$: Getting an almost optimal alternative path}
\label{s:locsuccess}
Let $\Gamma$ be the topmost maximal path in $\Pi$ from $\mathbf{0}$ to $\mathbf{n}$. The aim of this section is to prove for $x\in \mathcal{X}_r$ such that $G_x\cap H_x\cap R_x$ holds, with probability bounded away from $0$ independent of $r$, there exists a sufficiently regularly behaving alternative path, which deviates from $\Gamma$ only in $(x-\frac{r}{2}, x+\frac{r}{2})$ and is shorter than $\Gamma$ by at most an amount of $\delta r^{1/3}$, where $\delta$ is a small constant depending on $\lambda$. This is illustrated in Figure \ref{f:altpath}.

\begin{figure*}[h]
\begin{center}
\includegraphics[width=0.8\textwidth]{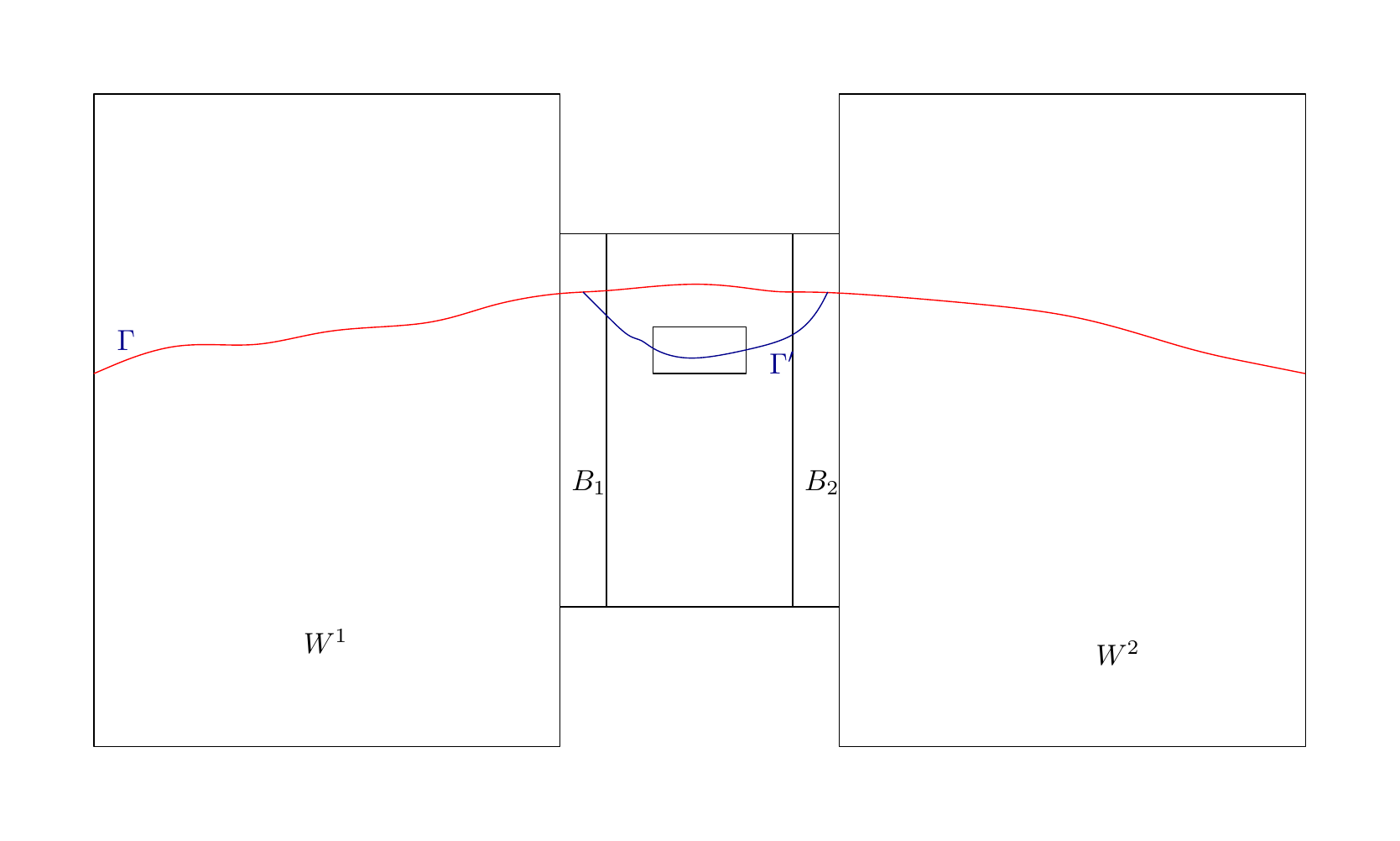}
\caption{An alternative path in $\mathbb{B}(x,r)$: $\Gamma$ is the topmost maximal path and $\Gamma'$ is an alternative path passing through $D$}
\label{f:altpath}
\end{center}
\end{figure*}

The strategy for showing the above is as follows. Consider the butterfly $\mathbb{B}(x,r)=\mathbb{B}(x,y(x,\Gamma),r)$. Resample the rectangles $D_i$ in $D=D(\mathbb{B}(x,r))$ one by one, conditioned on $\Gamma$ and also the configuration outside $D_i$. Since this process is reversible it gives us a way to estimate the probability of such a configuration.  We shall show that by the end of this process with positive probability we get an alternative path satisfying some conditions, to be made precise later.

Before proving that our job is to ensure that the alternative path we get by the above procedure satisfies the required regularity conditions.

\subsection{The alternative path deviates locally}
Fix $r$ and $x\in \mathcal{X}_r$. We first prove that on $G_x\cap H_x\cap R_x$, any competitive (i.e.\ not too short compared to $\Gamma$) alternative path which passes through $D$ will be very likely to deviate from $\Gamma$ only in the interval $[x-\frac{r}{2}, x+\frac{r}{2}]$. To make things precise we need to define the following global event $Q$. 

\begin{definition}[Steepness condition]
\label{d:steep}
An increasing path $\gamma$ from $\mathbf{0}$ to $\mathbf{n}$ is called {\bf steep} if there exists $\frac{n}{10}<x_1<x_2< \frac{9n}{10}$ such that $(x_2-x_1)\vee (\gamma_{x_2}-\gamma_{x_1}) \geq \frac{n^{2/3}}{2\log^{7}n}$  and $\frac{\gamma_{x_2}-\gamma_{x_1}}{x_2-x_1}\notin (\frac{20}{\psi}, \frac{\psi}{20})$.

For a point configuration $\Pi$, let $Q$ denote the event that: (i) the maximal path from $\mathbf{0}$ to $\mathbf{n}$ has length at least $2n-n^{0.35}$ and (ii) for every steep $\gamma$ from $\mathbf{0}$ to $\mathbf{n}$ we have $\ell_{\gamma}\leq 2n-n^{2/5}$.
\end{definition}

The event $Q$ asserts that any path containing a very high or low slope portion and not competitive in length with the global maximal path from $\mathbf{0}$ to $\mathbf{n}$. We shall show later that $Q$ is overwhelmingly likely (see Theorem \ref{t:steep}), but for now let us show that on $G_x\cap H_x\cap R_x\cap Q$, competitive alternative paths deviate locally. We shall need the following notation to state our next lemma.


Let $\gamma$ be another increasing path from $\mathbf{0}$ to $\mathbf{n}$ such that $\gamma$ passes through $D=D(\mathbb{B}(x,r))$. Let {\bf $D$-entry of $\gamma$} be the point $u_1=(x_1,y_1)$ where $\gamma$ intersects $D$ first, i.e., for each $x<x_1$, we have $(x,\gamma_x)\notin D$. Similarly let {\bf $D$-exit of $\gamma$} be the point $u_2=(x_2,y_2)$ where $\gamma$ intersects $D$ last. We define the {\bf split} of $\gamma$ to be the point $u_3=(x_3,y_3)$ such that $x_3=\sup_{x'<x_1}\{x':\gamma_{x'}=\Gamma_{x'}\}$. Similarly the {\bf confluence} of $\gamma$ is defined to be the point $u_4=(x_4,y_4)$ such that $x_4=\inf_{x'>x_2}\{x':\gamma_{x'}=\Gamma_{x'}\}$. It will suffice to consider the paths that deviate from $\Gamma$ only between the {\bf split} and the {\bf confluence}. We have the following lemma.

\begin{lemma}
\label{l:onghr}
Let $\Gamma$ be the topmost maximal increasing path from $\mathbf{0}$ to $\mathbf{n}$ in $\Pi$. Let $x\in \mathcal{X}_r$. Let $\gamma$ be another increasing path from $\mathbf{0}$ to $\mathbf{n}$ passing through $D=D(\mathbb{B}(x,r))$ with {\bf $D$-entry} $u_1=(x_1,y_1)$, {\bf $D$-exit} $u_2=(x_2,y_2)$, {\bf split} $u_3=(x_3,y_3)$ and {\bf confluence} $u_4=(x_4,y_4)$. Suppose $\Gamma=\gamma$ except on $(x_3,x_4)$. Also suppose either $x_3<x-\frac{r}{2}$ or $x_4>x+\frac{r}{2}$. The on $G_x\cap H_x \cap R_x\cap Q$, we have $$\ell_{\gamma}\leq \ell_{\Gamma}- \delta \varepsilon^{-2}r^{1/3}.$$
\end{lemma}
%

\begin{proof}
First let us make some notations. Recall that $T$ denotes the body of the butterfly $\mathbb{B}(x,r)$. We define the {\bf $T$-entry} of $\gamma$ as the point $u_5=(x_5,y_5)\in \gamma$ such that
$$x_5=\sup\{x'<x_1: \forall  x_1> x''\geq x', \ ~(x'',\gamma_{x''})\in T\}.$$
On $G_x\cap H_x\cap R_x$, depending on $T$-entries we can classify $\gamma$ into following three categories. {\bf Enter with $\Gamma$}: if $x_5<x_3$. {\bf Enter through $F$}: if $u_5\in F$. {\bf Enter through wall}: if $u_5$ is on the left boundary of $B_1$. Similarly we define the $T$-exit $u_6=(x_6,y_6)$ of $\gamma$ and classifiy $\gamma$ as {\bf exit with $\Gamma$}, {\bf exit through $F$} and {\bf exit through wall}.

The proof of the lemma is based on analysis of a few cases.


{\bf Case 1. Enter with $\Gamma$:} We shall need to consider two subcases.

\begin{figure*}[h]
\begin{center}
\includegraphics[width=0.8\textwidth]{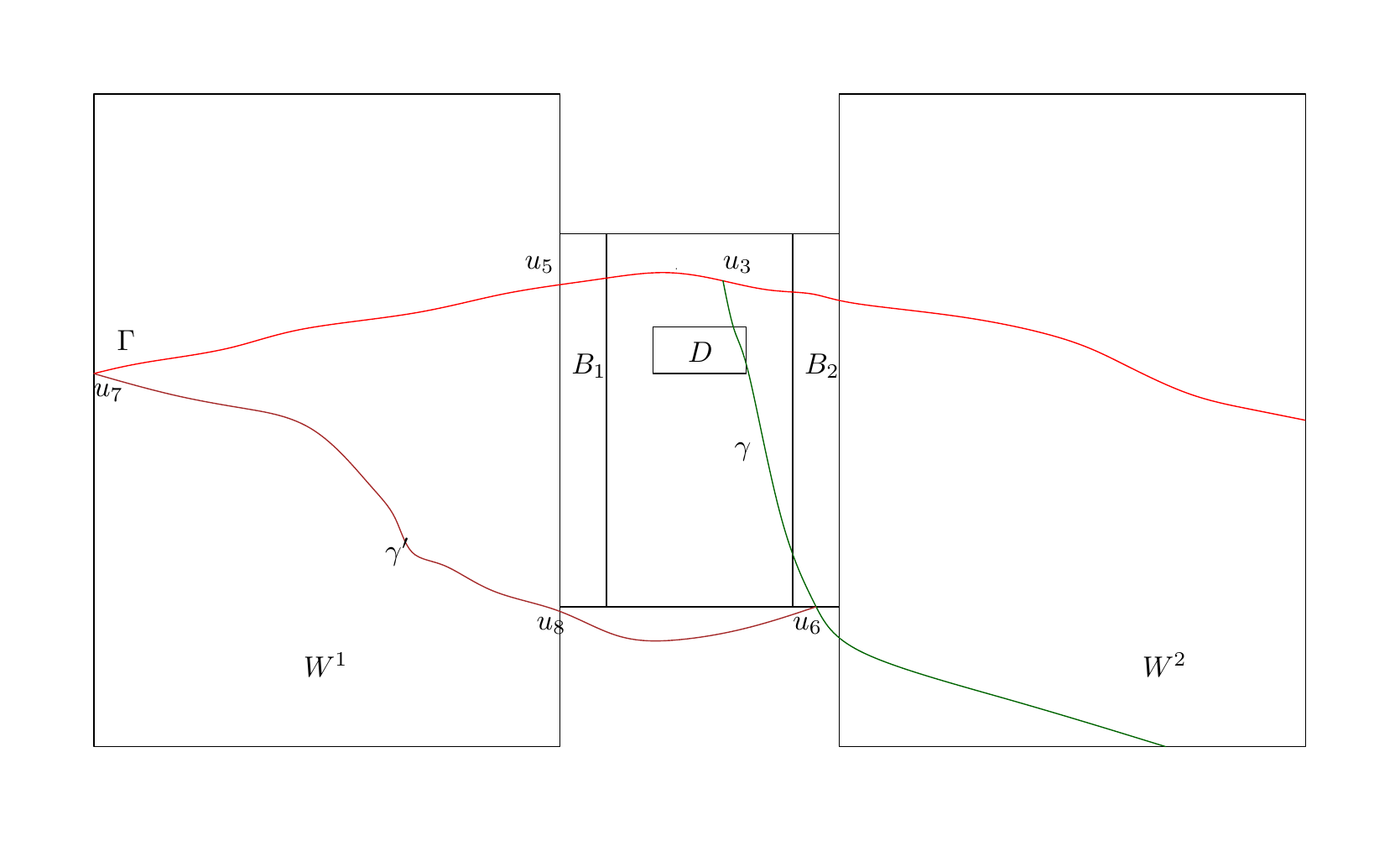}
\caption{Case 1.1: $\gamma$ denotes the path that does not deviate locally, $\gamma'$ is the path it is compared with}
\label{f:gammaf}
\end{center}
\end{figure*}

{\bf Case 1.1. Exit through $F$:} Let $u_7=(x_7,y_7)$ be the point on $\Gamma$ such that $x_7=x-(\frac{1}{2}+L^{3/2})r$ and $u_8=(x_8,y_8)$ be the point on $F$ with $x_8=x-\frac{r}{2}$.

Notice that it sufficies to prove that

\begin{equation}
\label{e:ghr11}
\hat{X}_{u_7,u_5}+\hat{X}_{u_5,u_2}+ \hat{X}_{u_2,u_6}^{\Gamma}\leq \hat{X}_{u_7,u_8} + \hat{X}_{u_8,u_6}^{F^{+}} -\delta \varepsilon^{-2}r^{1/3}.
\end{equation}

See Figure \ref{f:gammaf}. Observe that on $G_x\cap H_x\cap R_x$ we have

$$\hat{X}_{u_7,u_5}\leq 2CL^{3/4}r^{1/3};\qquad \hat{X}_{u_7,u_8}\geq -4CL^{3/4}r^{1/3};$$
$$\hat{X}_{u_5,u_2} \leq 4Cr^{1/3};\qquad \hat{X}^{\Gamma}_{u_2,u_6}\leq -Lr^{1/3};\qquad \hat{X}^{F^{+}}_{u_8,u_6} \geq -2Cr^{1/3}.$$

It follows that (\ref{e:ghr11}) holds since $L$ is sufficiently large (recall that $L$ was chosen sufficiently large depending on $\varepsilon$).

{\bf Case 1.2. Exit through wall:}
In this case, let $u_9=(x+\frac{2r}{5}, \gamma_{x+2r/5})$, $u_{10}=(x_{10}, y_{10})=(x+\frac{r}{2},y_{10})\in F$. Let $u_{11}=(x_{11},y_{11})$ be the point where $\gamma$ last exits $W^2$. Observe that $u_9, u_6 \in B_2$. Also observe that if either $(u_{6}, u_{11})\notin \mathcal{S}(W^2)$ or  $(u_{10}, u_{11})\notin \mathcal{S}(W^2)$, then $\gamma$ is steep and we are done by definition of $Q$. Hence assume otherwise. It suffices to show that

\begin{equation}
\label{e:ghr12}
\hat{X}_{u_7,u_5}+\hat{X}_{u_5,u_9}+ \hat{X}_{u_9,u_6}+ \hat{X}_{u_6,u_{11}} \leq \hat{X}_{u_7,u_8} + \hat{X}_{u_8,u_{10}}^{F^{+}}+ \hat{X}_{u_{10},u_{11}} -\delta \varepsilon^{-2}r^{1/3}.
\end{equation}

Observe that on $G_x\cap H_x\cap R_x\cap Q$ we have
$$ \hat{X}_{u_9,u_6}\leq  -\frac{C^*}{2}r^{1/3};\qquad \hat{X}_{u_{10},u_{11}} \geq -4CL^{3/4}r^{1/3};\qquad \hat{X}_{u_6,u_{11}}\leq 4CL^{3/4}r^{1/3}.$$
Using these and arguments similar to Case 1.1. we see that (\ref{e:ghr12}) holds.

{\bf Case 2. Enter through $F$:} We need to consider three subcases.

{\bf Case 2.1. Exit with $\Gamma$:} This case is similar to Case $1.1$ and we omit the details.

{\bf Case 2.2. Exit through wall:} Define points $u_9$, $u_{10}$, $u_{11}$ as in Case $1.2$. Clearly it suffices to show
$$\hat{X}_{u_5,u_9}+ \hat{X}_{u_9,u_6}+ \hat{X}_{u_6,u_{11}}\leq \hat{X}_{u_5,u_{10}}^{F^{+}}+\hat{X}_{u_{10},u_{11}}-\delta \varepsilon^{-2}r^{1/3}.$$
This is proved in a similar manner to Case 1.2 and we omit the details.

\begin{figure*}[h]
\begin{center}
\includegraphics[width=0.8\textwidth]{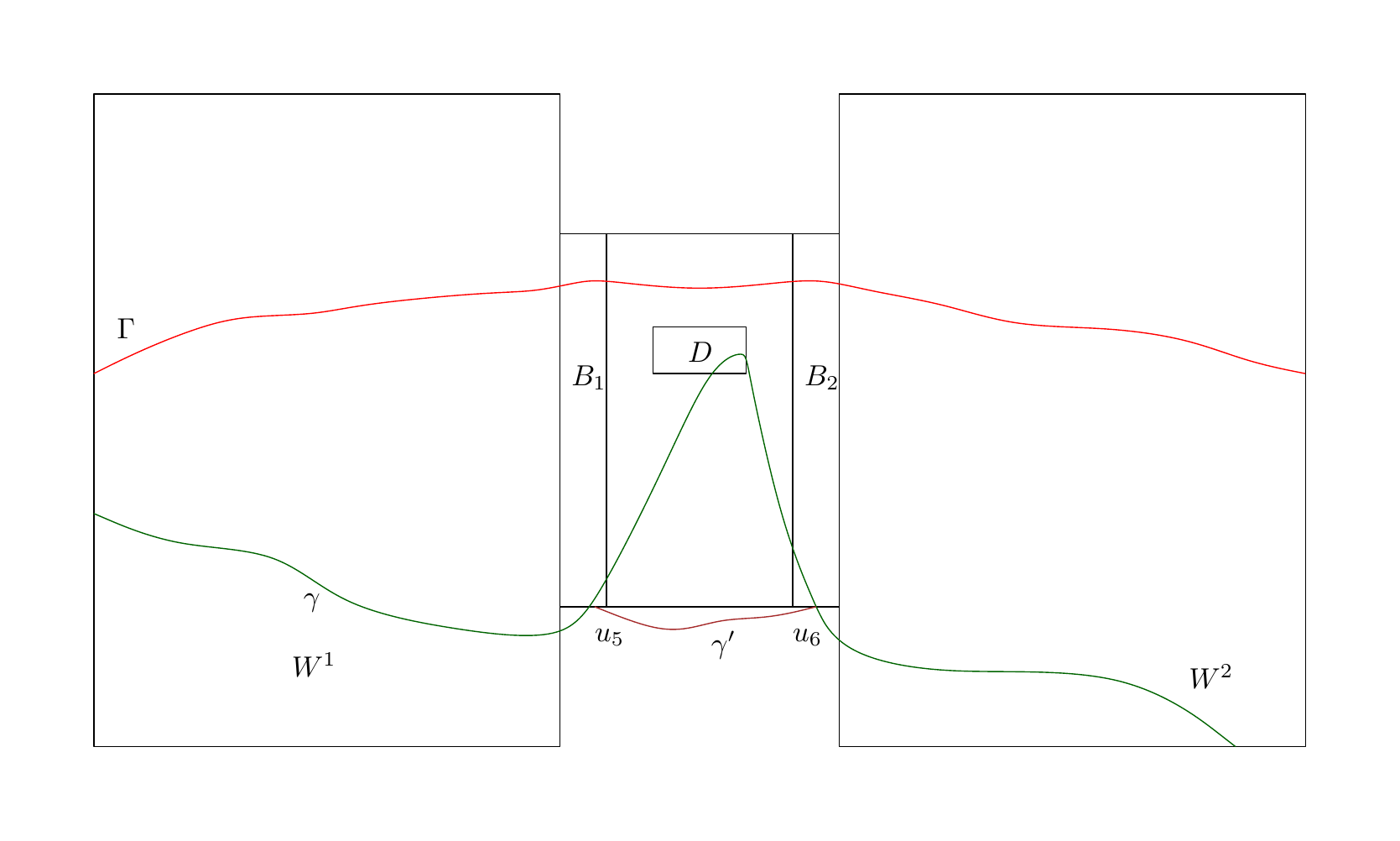}
\caption{Case 2.3: $\gamma$ denotes the path that enters and exits through $F$, it is compared with $\gamma'$}
\label{f:ghrff}
\end{center}
\end{figure*}

{\bf Case 2.3. Exit through $F$:}
In this case it suffices to show
$$\hat{X}_{u_5,u_6}^{\Gamma} \leq \hat{X}_{u_5,u_6}^{F^{+}}-\delta \varepsilon^{-2}r^{1/3}$$
which follows from the definition of $G_x$ and $H_x$ since $L$ is sufficiently large, see Figure \ref{f:ghrff}.

{\bf Case 3. Enter through wall:}
Again we need to consider threes subcases.

{\bf Case 3.1. Exit with $\Gamma$:} This case is similar to Case $1.2$, we omit the details.

\begin{figure*}[h]
\begin{center}
\includegraphics[width=0.8\textwidth]{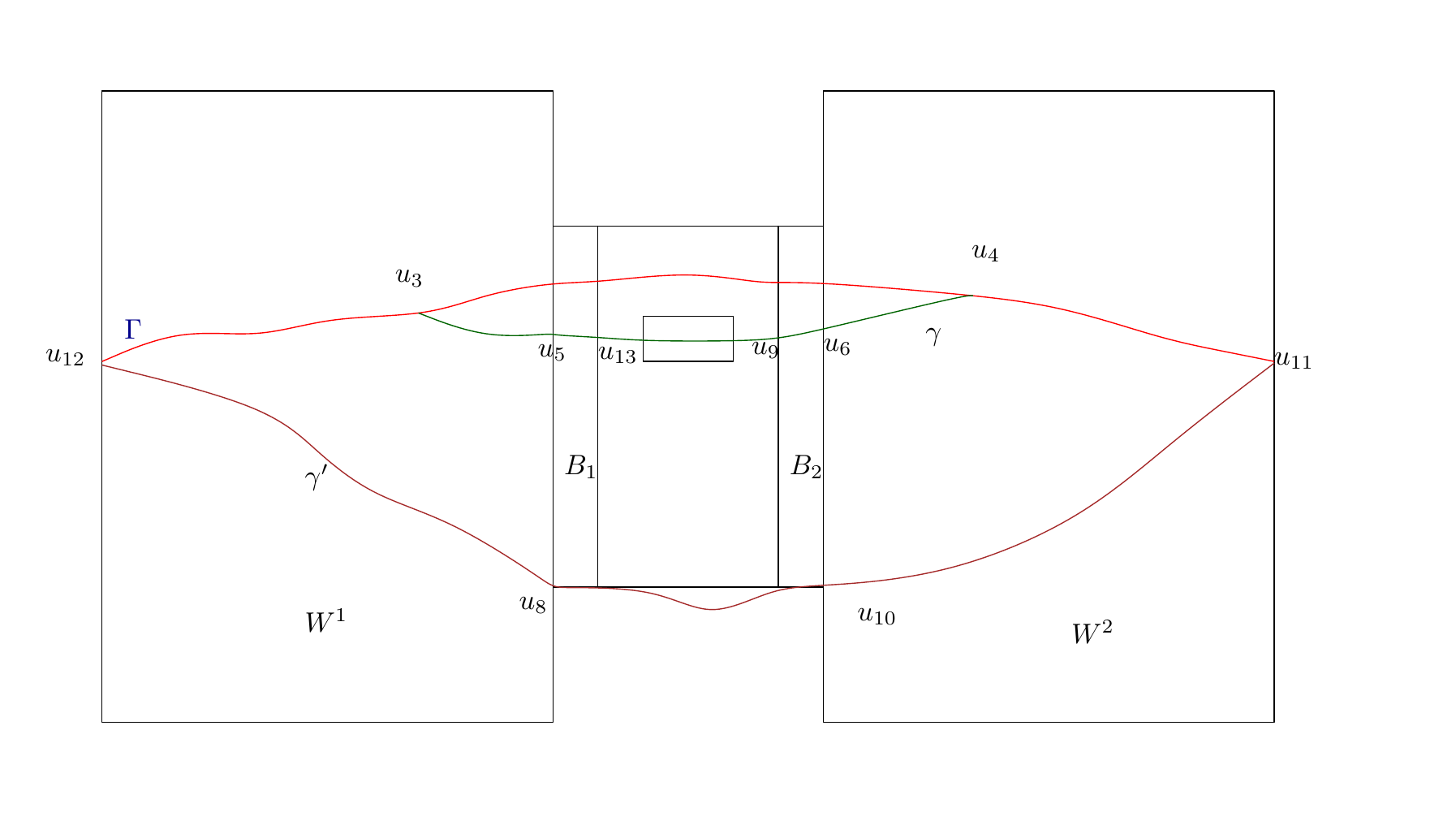}
\caption{Case 3.2: $\gamma$ denotes the path that enters and exits through the wall, compared with $\gamma'$}
\label{f:ghrww}
\end{center}
\end{figure*}

{\bf Case 3.2. Exit through wall:} Let $u_{12}$ denote the point where $\gamma$ first enters $W^1$. Let $u_{13}=(x_{13},y_{13})$ be the point on $\gamma$ such that $x_{13}=x-\frac{2r}{5}$. Clearly $u_{13}\in B_1$ and also without loss of generality we can assume $(u_{12},u_5), (u_{12},u_{8})\in \mathcal{S}(W^1)$, see Figure \ref{f:ghrww}. Clearly it suffices to show that
$$\hat{X}_{u_{12},u_5}+\hat{X}_{u_{5},u_{13}}+\hat{X}_{u_5,u_9}+\hat{X}_{u_9,u_6}+\hat{X}_{u_6,u_{11}}\leq \hat{X}_{u_{12},u_8}+\hat{X}_{u_8,u_{10}}^{F^{+}}+\hat{X}_{u_{10},u_{11}}-\delta \varepsilon^{-2}.$$
The proof can now be completed as in Case 1.2. See Figure \ref{f:ghrww}.

{\bf Case 3.3. Exit through $F$:} This case is analogous to Case 2.2.
\end{proof}

\subsection{The alternate path is not too steep}
The following lemma ensures that an alternative path through $\mathbb{B}(x,r)$ spends sufficiently long time in the region $(x+y(x,\Gamma)-\frac{3M}{2}r^{2/3}, x+y(x,\Gamma)-\frac{M}{2}r^{2/3})$.

\begin{lemma}
\label{l:spendstime}
Let $\Gamma$ be the topmost maximal increasing path from $\mathbf{0}$ to $\mathbf{n}$ in $\Pi$. Let $x\in \mathcal{X}_r$. Let $\partial^{+}(2D)$ denote the union of top, left and bottom boundary of $2D$ in the butterfly $\mathbb{B}(x,r)=\mathbb{B}(x,y(x,\Gamma),r)$. Fix a point $u_0=(x_0,y_0)\in \partial^{+}(2D)$. Let $\gamma$ be the path in $\Pi$ from $\mathbf{0}$ to $u$ of maximal length subject to the conditions
\begin{enumerate}
\item $\gamma$ does not intersect $D$,
\item $\{x':\Gamma_{x'}\neq \gamma_{x'}\}\subseteq [x-\frac{r}{2},x_0]$.
\end{enumerate}
Let $u_2=(x_2,y_2)\in \gamma$ be such that
$$x_2=\inf\left\{x': \gamma_{x''}\in \left[x''+y(x,\Gamma)-\frac{3M}{2}r^{2/3}, x''+y(x,\Gamma)-\frac{M}{2}r^{2/3}\right] \forall x''\in [x',x_0]\right\}.$$
Then on the event $G_{x}\cap H_{x}\cap R_{x}$, we have that $x_0 - x_2  \geq \alpha' r$.
\end{lemma}

\begin{proof}
We prove by contradiction. Let $\gamma$ be a path given by the hypothesis of the Lemma. Suppose $x_2 > x_0-\alpha' r$. We shall prove that on $G_x\cap H_x\cap R_x$, there exists a path $\gamma'$ satisfying the two conditions given in the lemma such that $\ell_{\gamma'}>\ell_{\gamma}$.

Observe that without loss of generality we can assume that there exists $u_1=(x_1,y_1)\in \Gamma$ such that $\Gamma=\gamma$ on $[0,x_1]$ and $\Gamma_{x'}>\gamma_{x'}$ on $(x_1,x_0]$ with $x-\frac{r}{2}<x_1<x_2$.

\textbf{Case 1:} $y_2=x_2+ y(x,\Gamma)-\frac{M}{2}r^{2/3}$.
There are two subscases to consider.

\textbf{Case 1.1:} $x_1\geq x-\frac{2r}{5}$. Set the point $u_4=(x_4,y_4)=(x-\frac{2r}{5}, \Gamma_{x-2r/5})$.
It suffices to prove that $\hat{X}_{u_4,u_0}^{D}\geq \hat{X}_{u_4,u_2}+\hat{X}_{u_2,u_0}$, which will contradict the maximality of $\gamma$. This is what we prove next.

Notice that on $G_x$ we have
$$\tilde{X}_{u_4,u_0}^{D}\geq -2Cr^{1/3};~\hat{X}_{u_4,u_2} \leq 2Cr^{1/3};~\hat{X}_{u_2,u_0} \leq -5Cr^{1/3}.$$

To prove the third inequality, define $u_3=(x_0-2\alpha' r, x_0-2\alpha' r+y(x,\Gamma)-\frac{M}{2}r^{2/3})$, use $\hat{X}_{u_2,u_0} \leq \hat{X}_{u_3,u_0}-\hat{X}_{u_3,u_2}$. Notice that on $G_x$, $\hat{X}_{u_3,u_2}\geq -2Cr^{1/3}$ and $\hat{X}_{u_3,u_0}\leq -10Cr^{1/3}$ since $\alpha'$ is sufficiently small using Lemma \ref{l:penalty}. This completes the proof in this case.

\begin{figure*}[h]
\begin{center}
\includegraphics[height=9cm,width=15.1cm]{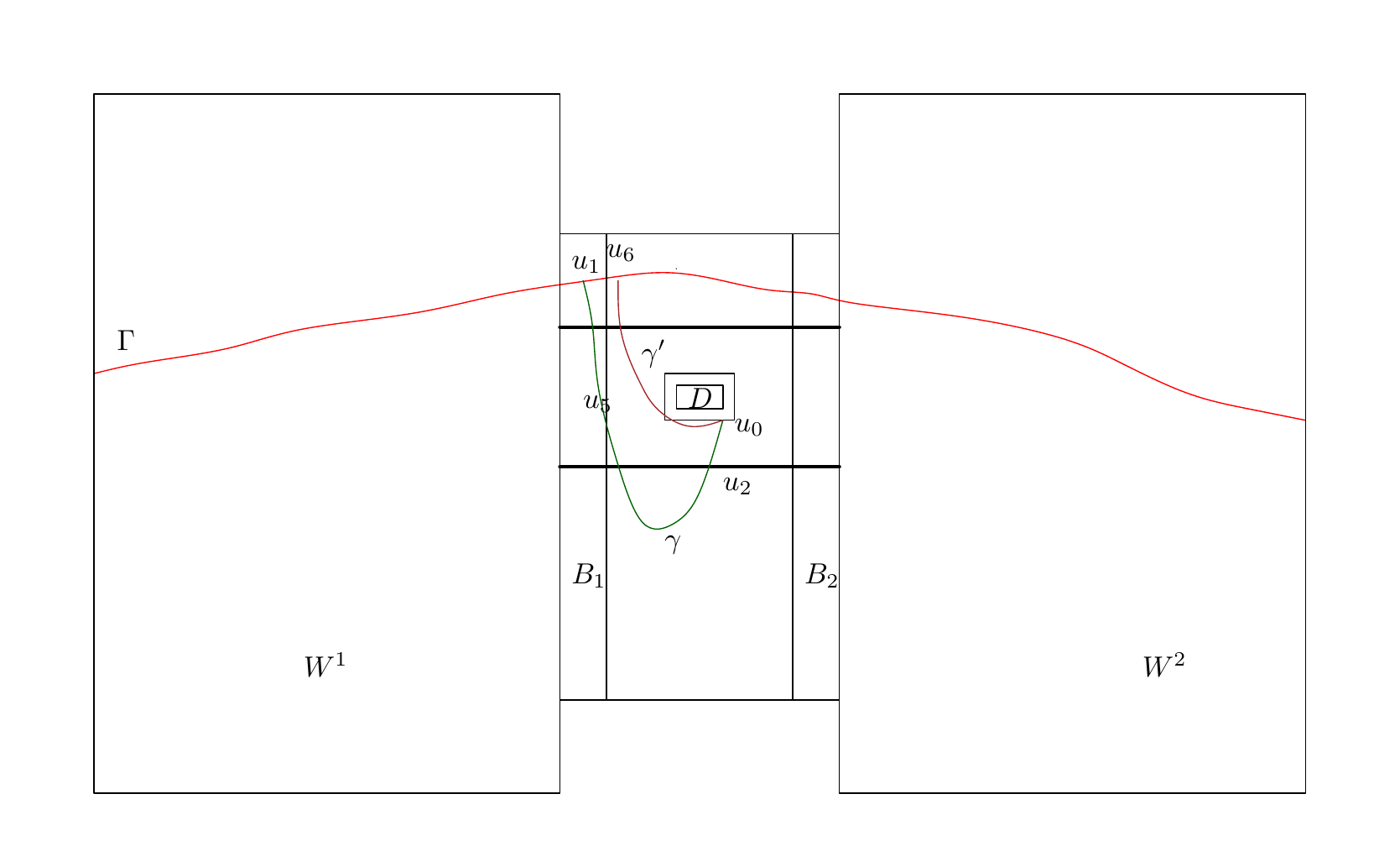}
\caption{Case 1.2: $\gamma'$ is a better path than $\gamma$}
\label{f:ghrsteep}
\end{center}
\end{figure*}

\textbf{Case 1.2:} $x_1< x-\frac{2r}{5}$.
Let $u_5=(x_5,y_5)\in \gamma$ be the first point on $\gamma$ which intersects the boundary of $B_1$, i.e., $x_5=\inf_{x'>x_1: (x',\gamma_{x'})\in \partial B_1}$. Also let $u_6=(x-\frac{r}{3}, \Gamma_{x-\frac{r}{3}})$, see Figure \ref{f:ghrsteep}. As before it suffices to prove,
$$\hat{X}_{u_1,u_5}+\hat{X}^{\Gamma}_{u_5,u_2}+\hat{X}_{u_2,u_0} \leq \hat{X}_{u_1,u_6}+\hat{X}_{u_6,u_0}^{D}.$$
Notice that on $G_x$ we have
$$ \hat{X}_{u_1,u_6}\geq -2Cr^{1/3};\qquad\tilde{X}_{u_6,u_0}^{D}\geq -Cr^{1/3}. $$
Also notice that as before on $G_x\cap H_x \cap R_x$ we further have
$$\hat{X}_{u_1,u_5} \leq Cr^{1/3};\qquad \hat{X}^{\Gamma}_{u_5,u_2}\leq Cr^{1/3};\qquad \hat{X}_{u_2,u_0}\leq -10Cr^{1/3}.$$
In this case also we have a contradiction.

\textbf{Case 2:} $y_2=x_2+y(x,\Gamma)-\frac{3M}{2}r^{2/3}$.
This can be dealt with in the same manner as above and we omit the details.
\end{proof}

\subsection{Sequential Resampling}
Recall our strategy of resampling to get a better path. As always, let $\Gamma$ denote the topmost maximal path from $\mathbf{0}$ to $\mathbf{n}$ in $\Pi$. For $r\in \mathcal{R}$, fix $x\in \mathcal{X}_{r}$ and consider the parallelogram $D$ in the butterfly $\mathbb{B}(x,y(x,\Gamma),r)$. Our first lemma states that on resampling the configuration on $D$ the length of the longest path increases with a chance bounded away from $0$.

\begin{lemma}
\label{l:success1}
Let $\tilde{\Pi}$ be the point configuration on $\R^2$ where $\left. \Pi\right|_{D}$ is replaced by $\left. \Pi^*\right|_D$. Let $\Gamma'$ denote a longest increasing path in $\tilde{\Pi}$. Then
$$\P[\ell_{\Gamma'}^{\tilde{\Pi}} > \ell_{\Gamma}^{\Pi} \mid \Pi, G_x, H_x] \geq \rho.$$
\end{lemma}

\begin{proof}
Let $u_1$ and $u_2$ be the midpoint of the left boundary and the right boundary of $D$ respectively. For $u<u'$, let $Y_{u,u'}$ (resp. $\hat{Y}_{u,u'}$ etc.) denote the length of the longest increasing path from $u$ to $u'$ (resp. the other corresponding statistics) in the environment $\Pi^{*}$. It follows from the definition of $G_x$ and $H_x$ and that it suffices to prove

\begin{equation}
\label{e:constrainedlong}
\hat{Y}^{D}_{u_1,u_2}\geq 10Cr^{1/3}.
\end{equation}

This can be established by using Theorem \ref{BDJ99} arguing along the lines of the proof of Lemma \ref{l:pathbox} and we omit the details.
\end{proof}

Observe that this is not directly useful for us since we shall need to resample conditioning on the fact that $\Gamma$ is the topmost maximal path. To this end we use the parallelograms $D_i$ described in the definition of condition $G_{x,y}^{\rm rs}$.

For $i=1,2,\ldots, \frac{1}{100\varepsilon^{5/3}}$, we define the measure $\mu_x^{*,i}$ on point configurations on $\R^2$ inductively as follows. Let $\mu_{x}^{*,0}=\mu_x^*$. For now, let $\Pi$ denote a point configuration in $\R^2$ sampled from $\mu_x^*$. Let $\Gamma$ denote the topmost maximal path in $\Pi$. Consider the parallelograms $D_i$ in the butterfly $\mathbb{B}(x,r)$. For $i\geq 0$ sample a point configuration $\widehat{\Pi}^{*,i}$ sampled recursively as follows. Let $\widehat{\Pi}^{*,0}=\Pi$. Given $\widehat{\Pi}^{*,i-1}$, obtain $\widehat{\Pi}^{*,i}$ by resampling the point configuration on $D_i$ with law $\tilde{\mu}_{D_i}:=\mu_{D_i}(~\cdot \mid \widehat{\Pi}_{D_i^c}^{*,i-1},\mathscr{H}_{x,i})$ where $\mu_{D_i}$ is the measure $\mu$ restricted on $D_i$ and $\mathscr{H}_{x,i}$ denotes the event that $\Gamma$ is the topmost maximal path in the new point configuration (after resampling $D_i$). Let $\mu_{x}^{*,i}$ be the measure on the point configuration $\Pi^{*,i}$ in $\R^2$ obtained as described above.

\begin{lemma}
\label{l:couplingfails}
Let $x\in \mathcal{X}_r$ be fixed. Let $\Pi$ and $\widehat{\Pi}^{*,i}$ be defined as above. For $i\geq 0$, let $\mathscr{T}_{x,i}$ be the event that there is an increasing path $\Gamma'$ from $\mathbf{0}$ to $\mathbf{n}$ in $\widehat{\Pi}^{*,i}$ passing through $D$ such that $\ell_{\Gamma'}\geq \ell_{\Gamma}-\delta r^{1/3}$. Let $\tau=\inf_{i\geq 0}\{\mathscr{T}_{x,i}~\text{holds}\}$. Then we have
$$\P[\tau\leq \frac{1}{100\varepsilon^{5/3}} \mid \Pi, G_x,H_x]\geq \frac{\rho}{2} >0.$$
\end{lemma}

Before we prove Lemma \ref{l:couplingfails}, let us first explain the idea behind the proof. For $i\geq 1$, we sample $\widehat{\Pi}^{*,i}$ using rejection sampling as follows. Let $\Pi^*$ be an independent sample from $\mu$. For each $i\geq 1$, we take the configuration $\Pi^*_{D_i}$ (whose law is $\mu_{D_i}$). If replacing $\Pi^{*,i-1}_{D_i}$ by $\Pi^*_{D_i}$ does not violate the condition $\mathscr{H}_{x,i}$ we take $\Pi^*_{D_i}$ as a realization of $\tilde{\mu}_{D_i}$ and use it to resample the configuration on $D_i$ and generate $\Pi^{*,i}$,  else generate a configuration according to $\tilde{\mu}_{D_i}$ using some external randomness. Let $\tau'$ be the first time this coupling fails; i.e.\ $\Pi^*_{D_i}$ is rejected as a sample from $\tilde{\mu}_{D_i}$. We shall show that on $G_{x}\cap H_{x}$, $\P[\tau'\leq \frac{1}{100\varepsilon^{5/3}}]$ can be made sufficiently small by taking $\varepsilon$ to $0$ and complete the proof by invoking Lemma \ref{l:success1}. To this end we shall want to make use of the resampling condition in the definition of $G_x$ and for that we need to define the following global event, which is a stronger variant of the steepness condition $Q$ defined earlier. 

\begin{definition}[Stronger Steepness Condition]
For $r\in \mathcal{R}$, $x\in \mathcal{X}_r$ and $y\in r^{2/3}\Z$ and for $i\geq 0$, let $\mathscr{S}_i=\mathscr{S}_i(x,y,r)$ denote the event that \emph{steepness condition} $Q$ as defined in Definition \ref{d:steep} holds for the point configuration $\Pi^{(*,i)}$. Let $Q'$ denote the event that 
$$\P[\cup_{i,x,y,r}  \mathscr{S}_i^c\mid \Pi]\leq e^{-n^{1/100}}.$$
\end{definition}

We have the following theorem showing both $Q$ and  $Q'$ are overwhelmingly likely.

\begin{theorem}
\label{t:steep}
For $n$ sufficiently large we have that $\P[Q\cap Q'] \geq 1-e^{-n^{0.01}}$.
\end{theorem}

We shall postpone the proof of Theorem \ref{t:steep} for now and continue with the proof of Lemma \ref{l:couplingfails}.

\begin{proof}[Proof of Lemma \ref{l:couplingfails}]
%

As explained above we shall show that, on $\{i<\tau \wedge \tau' \}\cap G_x\cap H_x \cap Q \cap Q'$, we have
\begin{equation}
\label{e:couplingfails}
\P[\tau'=i+1] \leq e^{-\frac{C}{2\varepsilon^{1/4}}}.
\end{equation}

To establish (\ref{e:couplingfails}) observe the following. Let $\tilde{\Pi}^{*, i+1}$ be the point configuration obtained from $\widehat{\Pi}^{*,i}$ by replacing $\widehat{\Pi}^{*,i}_{D_{i+1}}$ by $\Pi^*_{D_{i+1}}$. Notice that on $\{i<\tau \wedge \tau' \}\cap G_x\cap H_x \cap Q \cap Q' \cap \{\tau'=i+1\}$, there is an increasing path in $\tilde{\Pi}^{*, i+1}$ from $\mathbf{0}$ to $\mathbf{n}$ passing through $D$ with length more than $\ell_{\Gamma}$, and replacing $\widehat{\Pi}^{*,i}_{D_{i+1}}$ by $\Pi^*_{D_{i+1}}$ increases the length of $\gamma$ by at least $\delta r^{1/3}$. Observe that either $\gamma$ is \emph{steep} or $\Delta_{i+1}> \delta r^{1/3}$ (recall the definition of $\Delta_{i+1}$ from the resampling condition in $G_x$). The former case has probability at most $e^{-n^{0.01}}$ by definition of $Q\cap Q'$ and \eqref{e:couplingfails} follows from the resampling condition in $G_x$. It follows now from Lemma \ref{l:success1} and Theorem \ref{t:steep} that
$$\P[\tau\leq \frac{1}{100\varepsilon^{5/3}}\mid \Pi, G_x,H_x]\geq \rho- \varepsilon^{-2}e^{-\frac{C}{2\varepsilon^{1/4}}}-\P[Q^c \cup (Q')^c]\geq \rho/2$$
since $\varepsilon$ is small enough. This completes the proof of the lemma.
%
\end{proof}

\subsection{Local Success}
Let $\Pi$ be a point configuration on $\R^2$, not necessarily distributed according to $\mu$. Let $\Gamma$ be be the topmost maximal increasing path in $\Pi$ from $\mathbf{0}$ to $\mathbf{n}$. Fix $r\in \mathcal{R}$. For $x\in \mathcal{X}_r$ we define the event {\bf ``Success at $x$ in scale $r$ at cost $\delta$"}, denoted $S_{x,r,\delta}$ to be the event that the following conditions hold.

\begin{enumerate}
\item We have
\begin{enumerate}
\item $|\Gamma_{x'}-\Gamma_x|\leq \frac{Mr^{2/3}}{10}$ for all $x'\in [x-\frac{r}{2},x+\frac{r}{2}]$.
\item $|\Gamma_{x}-x|\leq Cn^{2/3}$.
\end{enumerate}
\item There exists an increasing path $\Gamma'$ in $\Pi$ from $\mathbf{0}$ to $\mathbf{n}$ such that
\begin{enumerate}
\item $\Gamma'$ passes through $D$,
\item $\{x':\Gamma'_{x'}\neq \Gamma_{x'}\}\subseteq [x-\frac{r}{2},x+\frac{r}{2}]$,
\item $\ell_{\Gamma'}\geq \ell_{\Gamma}-\delta r^{1/3}$.
\item There exists points $u_1=(x_1,\Gamma'_{x_1})$ and $u_2=(x_2,\Gamma'_{x_2})$ on $\Gamma'$ such that $x_1,x_2 \in [x-\frac{r}{2},x+\frac{r}{2}]$ and $O_{\Gamma'}(u_1,u_2)$ is contained in the region $$\{(x',y')\in \R^2: (\Gamma_{x}-x)- \frac{2M}{5}r^{2/3} \leq  y'-x'\leq (\Gamma_x-x) -\frac{8M}{5}r^{2/3}\}$$ and $A_{\Gamma'}(u_1,u_2)\geq \alpha ' \eta r$.
%
%
\end{enumerate}
\end{enumerate}

Our goal is to prove the following theorem.
\begin{theorem}
\label{t:successmu}
For $\mathcal{X}_r^*\subseteq \mathcal{X}_r$ with $|\mathcal{X}_r^*|\geq \frac{9}{10}|\mathcal{X}_r|$ given by Theorem \ref{p:gxhxconditional} and for each $x\in \mathcal{X}_r^*$, we have  $\mu(S_{x,r,\delta})>p>0$ where $p$ is a constant independent of $r$.
\end{theorem}

First we prove the following lemma.

\begin{lemma}
\label{l:successmustar}
Let $\mathcal{X}_r^*$ be as given by Theorem \ref{p:gxhxconditional}. For $x\in \mathcal{X}_r^*$, there exists $i\in \{0,1,\ldots, \frac{1}{100\varepsilon^{5/3}}\}$ such that $\mu_x^{*,i}(S_{x,r,\delta})\geq \frac{9\rho\varepsilon^2}{40}.$
\end{lemma}


\begin{proof}
Sample $\Pi$ from the measure $\mu_x^{*}$. Fix $x\in \mathcal{X}_r$ such that $G_x$ and $H_x$ hold. Consider the set-up of Lemma \ref{l:couplingfails}. It follows from Lemma \ref{l:couplingfails} that there exists $i\in \{0,1,\ldots, \frac{1}{100\varepsilon^{5/3}}\}$ such that
$\P[\tau=i\mid \Pi, G_x\cap H_x] \geq \varepsilon^{2}\rho/2$. Let $\Pi^*$ be an independent sample of $\mu$. Now generate a sample from $\mu_{x}^{*,i}$ in the manner described in the proof of Lemma \ref{l:couplingfails}. Recall that $\Pi^{(*,j)}$ denotes the point configuration obtained from $\Pi$ by changing $\Pi_{\cup_{k=1}^{j}D_k}$ to $\Pi^*_{\cup_{k=1}^{j}D_k}$. Also recall the definition of $\Delta_j$ from the resampling condition. Let $A_x$ denotes the event that $\max_{j,j\leq i} \Delta_j \leq \frac{\delta r^{1/3}}{2}$. It follows that
$$\P[\tau'>i, \tau=i, A_x \mid \Pi, G_x\cap H_x] \geq \frac{\varepsilon^{2}\rho}{2}- \varepsilon^{-2}e^{-\frac{C}{2\varepsilon^{1/4}}} \geq \frac{\varepsilon^{2}\rho}{4}$$
since $\varepsilon$ is sufficiently small. 
Consider another global good event $V$ defined as follows. Let $V$ denote the event that each square of side length at least $n^{1/2}$ contained in $[0,n]^2$ contains at least one point of $\Pi$. Clearly, by a union bound one has $\mu(V)\geq 1-e^{-n^{0.01}}$. Now define $Q_0:=Q\cap Q' \cap V$. It follows from Theorem \ref{t:steep} and Lemma \ref{l:mustarbasic} that $\mu^{*}_{x}(Q_0^c)\leq e^{-n^{0.001}}$ for $n$ sufficiently large.

Now observe that on $\{G_x\cap H_x \cap A_x \cap Q_0\}\cap \{\tau=i, \tau' >i\}$ there exists an increasing path $\Gamma'$ in $\Pi^{(*,i)}$ such that $\Gamma'$ satisfies all the conditions in the definition of $S_{x,r,\delta}$. To see this, observe that by construction $\Pi$ and $\Pi^{(*,i)}$ has the same topmost maximal path and hence Condition 1 holds by the definition of $G_x$. Condition $2(a)$ and $2(c)$ holds by the definition of $\tau$. That Condition $2(b)$ holds is a consequence of Lemma \ref{l:onghr}. Condition $2(d)$ is a consequence of Lemma \ref{l:spendstime}, the wing condition in the definition of $G_{x}$ and the global good event $V$. It follows that
$$\mu_x^{*,i}(S_{x,r,\delta}\mid \Pi, G_x\cap H_x)\geq  \frac{\varepsilon^{2}\rho}{4}.$$

It follows from Theorem\ref{p:gxhxconditional} that for $x\in \mathcal{X}_r^*$, we have
$$\mu_x^{*,i}(S_{x,r,\delta}) \geq \mu_x^{*}(G_x\cap H_x)\frac{\varepsilon^{2}\rho}{4}\geq \frac{9\varepsilon^{2}\rho}{40}.$$
\end{proof}

Now we are ready to prove Theorem \ref{t:successmu}.

\begin{proof}[Proof of Theorem \ref{t:successmu}]
For $x\in \mathcal{X}_r^{*}$, choose $i$ as in the previous lemma. Notice that the resampling of $D_i$ under the conditional measure can be interpreted as step of the Glauber dynamics and hence a smoothing operator which implies that $\sup \frac{{\rm d}\mu_x^{*,j}}{{\rm d}\mu}$ is decreasing in $j$, and hence
$$\sup \frac{{\rm d}\mu_x^{*,i}}{{\rm d}\mu}\leq \sup \frac{{\rm d}\mu_x^{*}}{{\rm d}\mu}\leq \frac{1}{\beta}<\infty$$
by Lemma \ref{l:mustarbasic} where $\beta$ is a constant independent of $r$.
The result now follows from Lemma~\ref{l:successmustar}.
\end{proof}

\section{Combining Success at Different Scales and Locations}
\label{s:finish}
Our goal in this section is to improve the length of the local almost optimal paths of the previous section using the extra points from the reinforced configuration, and then put together all these improvements to obtain a path longer than the optimal path in the unperturbed configuration.

\subsection{Reinforcing on different lines}
As explained in the introduction, our strategy is to consider, instead of only one reinforced configuration, a family of reinforced configurations, where the reinforcement is on different translates of the diagonal line $\{x=y\}$. Let $\lambda>0$ be fixed. For each $m\in [-2Cn^{2/3}, 2Cn^{2/3}]$, let $\Sigma_{\lambda}^{(m)}$ denote a one dimensional PPP with intensity $\lambda$ on the line $\mathbb{L}_m:\{y=x+m\}$. Let $\Pi_{\lambda}^{(m)}$ be the point process obtained by superimposing $\Pi$ and $\Sigma_{\lambda}^{(m)}$. Let $L_n^{(\lambda, m)}$ be the length of the maximal increasing path from $\mathbf{0}$ to $\mathbf{n}$ in $\Pi_{\lambda}^{(m)}$. As explained in \S~\ref{s:outline}, to prove Theorem \ref{t:maintheoremppp}, it suffices to show that $\E L_n^{(\lambda,m)}>2n$ for some $n$ and $m$.
For $m\in (-2Cn^{1/3}, 2Cn^{1/3})$, let $\Gamma(m)$ denote the longest path in $\Pi_{\lambda}^{(m)}$ from $(0,0)$ to $(n,n)$. 
For the next lemma we shall use the notation $\ell_{\Gamma(m)}=\ell_{\Gamma(m)}^{\Pi_{\lambda}^{(m)}}$ and $\ell_{\Gamma}=\ell_{\Gamma}^{\Pi}$.

\begin{lemma}
\label{l:difflines}
For some $n$ sufficiently large, there exists $m\in [-2Cn^{2/3},2Cn^{2/3}]$ such that we have $\E[\ell_{\Gamma(m)}-\ell_{\Gamma}]> 100n^{1/3}$.
\end{lemma}

\begin{proof}
Let $r\in \mathcal{R}$ and $x\in \mathcal{X}_r$ be fixed. For a given $\Pi$ with the topmost maximal path $\Gamma$, if $S_{x,r,\delta}$ holds, let $\Gamma'(x,r,\delta)$ denote the alternative path given by the definition of $S_{x,r,\delta}$. Let $u_1$ and $u_2$ be as in the definition of $S_{x,r,\delta}$. Let $\Gamma^*(x,r,\delta,m)$ denote the increasing path in $\Pi_{\lambda}^{(m)}$ which contains all points of $\Gamma'(x,r,\delta)$ that belong to $\Pi$ and also all points of $\Sigma_{\lambda}^{(m)}$ that are contained in $O_{\Gamma'}(u_1,u_2)$. Let us denote $\ell_{\Gamma^*}=\ell_{\Gamma^*(x,r,\delta,m)}^{\Pi_{\lambda}^{(m)}}$ and $\ell_{\Gamma'}=\ell_{\Gamma'(x,r,\delta)}^{\Pi}$. Also set
$$ G^{m}_{x,r,\Gamma}=(\E[\ell_{\Gamma^*}\mid \Pi]-\ell_{\Gamma})\mathbf{1}_{S_{x,r,\delta}}.$$
Observe that on $S_{x,r,\delta}$, we have
$$\int_{(\Gamma_x-x)-\frac{8Mr^{2/3}}{5}}^{(\Gamma_x-x)-\frac{2Mr^{2/3}}{5}} (\E[\ell_{\Gamma^*}\mid \Pi]-\ell_{\Gamma'}) ~dm \geq \frac{\lambda}{10}A_{\Gamma'}(u_1,u_2).$$
It follows from the definition of $S_{x,r,\delta}$ and since $\delta$ is sufficiently small depending on $\lambda$ that
$$\int_{(\Gamma_x-x)-\frac{8Mr^{2/3}}{5}}^{(\Gamma_x-x)-\frac{2Mr^{2/3}}{5}} G^{m}_{x,r,\Gamma}~dm\geq (\frac{\alpha ' \lambda r}{10}-\frac{6M\delta r}{5})\mathbf{1}_{S_{x,r,\delta}} \geq rc_{\lambda}\mathbf{1}_{S_{x,r,\delta}}$$
for some constant $c_{\lambda}>0$.

Now notice that for a fixed $m$, and for a fixed $x\in \cup_{r\in \mathcal{R}}\mathcal{X}_{r}$, we have that
$$\mathbf{1}_{\bigl\{m\in ((\Gamma_x-x)-\frac{8Mr^{2/3}}{5}, (\Gamma_{x}-x)-\frac{2Mr^{2/3}}{5})\bigr\}}$$
is nonzero for at most one value of $r\in \mathcal{R}$. Also notice that if
$$\mathbf{1}_{\bigl\{m\in ((\Gamma_x-x)-\frac{8Mr^{2/3}}{5}, (\Gamma_{x}-x)-\frac{2Mr^{2/3}}{5})\bigr\}}=1$$
and $S_{x,r,\delta}$ holds for some value of $r\in \mathcal{R}$ and $x\in \mathcal{X}_r$, then for any $r_1\in \mathcal{R}$ with $r_1<r$, and for any $x'\in [x-\frac{r}{2},x+\frac{r}{2}]\cap \mathcal{X}_{r_1}$ we have
$$\mathbf{1}_{\bigl\{m\in (\Gamma_{x'}-\frac{3Mr_1^{2/3}}{2}, \Gamma_{x'}-\frac{Mr_1^{2/3}}{2})\bigr\}}=0.$$
Finally notice that for a fixed $r$ and $x\in \mathcal{X}_r$, $\Gamma^*(x,r,\delta,m)$ and $\Gamma$ deviate from one another only in the interval $[x-\frac{r}{2},x+\frac{r}{2}]$ on $S_{x,r,\delta}$. All these together imply

$$\E[\ell_{\Gamma(m)}\mid \Pi]-\ell_{\Gamma}  \geq \sum_{r\in \mathcal{R}}\sum_{x\in \mathcal{X}_r}\mathbf{1}_{\bigl\{m\in (\Gamma_{x}-\frac{3Mr^{2/3}}{2}, \Gamma_{x}-\frac{Mr^{2/3}}{2})\bigr\}}G^{m}_{x,r,\Gamma}.$$

By a series of interchanges of summation, expectation and integration and using Lemma \ref{t:successmu}, it follows that
\begin{eqnarray*}
\int_{-2Cn^{2/3}}^{2Cn^{2/3}}\E[\ell_{\Gamma(m)}-\ell_{\Gamma}]~dm &\geq & \sum_{r}\sum_{x}rc_{\lambda}\mu(S_{x,r,\delta})\\
&\geq & \sum_{r} pc_{\lambda} r \frac{n}{10r}\geq pc_{\lambda}|\mathcal{R}|n.
\end{eqnarray*}

Hence it follows that there exists $m\in [-2Cn^{2/3}, 2Cn^{2/3}]$ such that
$$\E[\ell_{\Gamma(m)}-\ell_{\Gamma}]\geq \frac{pc_{\lambda}}{4C}|\mathcal{R}|n^{1/3}.$$
The lemma follows.
\end{proof}

\subsection{Proof of Theorem \ref{t:maintheoremppp}}
We can now complete the proof of Theorem \ref{t:maintheoremppp}.
\begin{proof}[Proof of Theorem \ref{t:maintheoremppp}]
Recall that $L_{N}^{(\lambda, m)}$ denotes the length of a maximal path in the environment $\Pi_{\lambda}^{(m)}$ from $(0,0)$ to $(N,N)$. Notice that for a fixed $m$, we have that $\E(L_{N}^{(\lambda, m)})$ is superadditive in $N$, i.e.,
$$\E(\ell_{\Gamma _{m,N_1}})+ \E(\ell_{\Gamma _{m,N_2}})\leq \E(\ell_{\Gamma _{m,N_1+N_2}})$$   
for all $N_1,N_2>0$. Now choosing $n$ and $m$ as given by Lemma \ref{l:difflines} and choosing $|\mathcal{R}|$ sufficiently large depending on $C, c_{\lambda}$ and $p$ as given in Lemma \ref{l:difflines} it follows that we have 
$\E(L_{N}^{(\lambda, m)})> 2n$ and by using superadditivity we get 
$$\lim_{N\rightarrow \infty}  \frac{\E(L_{N}^{(\lambda, m)})}{N} >2.$$

Now, notice by translation invariance we have that $\E L^{\lambda}_{N-m} \geq \E L_{N}^{(\lambda, m)} $ and as $m$ is fixed we get that 
$$\lim_{N\rightarrow \infty}  \frac{\E L^{\lambda}_N}{N} >2$$
thereby completing the proof of Theorem \ref{t:maintheoremppp}.
\end{proof}

\textbf{To complete the proof of Theorem \ref{t:maintheoremppp}, it remains to establish Theorem \ref{l:rdecreasing}, Theorem \ref{p:gxhxconditional} and Theorem \ref{t:steep}. The next three sections in the paper are devoted to proving these results.}

\section{Maximal paths behave nicely most of the time}
\label{s:maxpathnice}
In this section we shall prove that the paths of maximal length with large probability, behave regularly at a scale $r$ (i.e., have on scale fluctuations) around most locations $x$. Throughout this section we shall work at a fixed scale $r$. The general strategy for the results in this section is to use the following Peierls type argument: We consider a set of discretized paths, and show that along each such path show that it is exponentially unlikely (in  $\frac{n}{r}$, the number of locations in scale $r$) that the desired property is violated at more than a small fraction of locations. We complete by taking a union bound over a not too large set of discretized paths, to which we show that a path of maximal length belongs with high probability.

We start with defining a discretization for paths. For a fixed $r$ and for $i\in \{0,1,\ldots, n/r\}$, let  $\mathscr{I}_i$ denote the set of all line segments of the form
$$\{(ir,y'):\ell r^{2/3}\leq y'-ir\leq (\ell+1)r^{2/3}\},$$
which represents a discretization of the endpoints of the $i$-th segment of the path.
Let $\mathscr{I}$ denote the set of all sequences of the form $\mathcal{J}=\{J_i\}_{0\leq i \leq n/r}$ where $J_i\in \mathscr{I}_i$. Fix $\mathcal{J}\in \mathscr{I}$ where
\begin{equation}
\label{e:deltagamma}
J_i= \{(ir,y'):j_i r^{2/3}\leq y'-ir\leq (j_i+1)r^{2/3}\}.
\end{equation}
Define $\Delta_i(\mathcal{J})= j_{i+1}-j_{i}$.

Now let $\gamma$ be an increasing path from $\mathbf{0}$ to $\mathbf{n}$. Define $\mathcal{J}=\mathcal{J}(\gamma)\in \mathscr{I}$ as follows. Let $j_i=\lfloor \frac{\gamma_{ir}-ir}{r^{2/3}} \rfloor$. Define
$\mathcal{J}=\{J_i(\gamma)\}_i$ by (\ref{e:deltagamma}). See Figure \ref{f:discrete}. Note that this figure is not drawn in the tilted co-ordinates.

\begin{figure}[ht!]
\begin{center}
\includegraphics[width=0.6\textwidth]{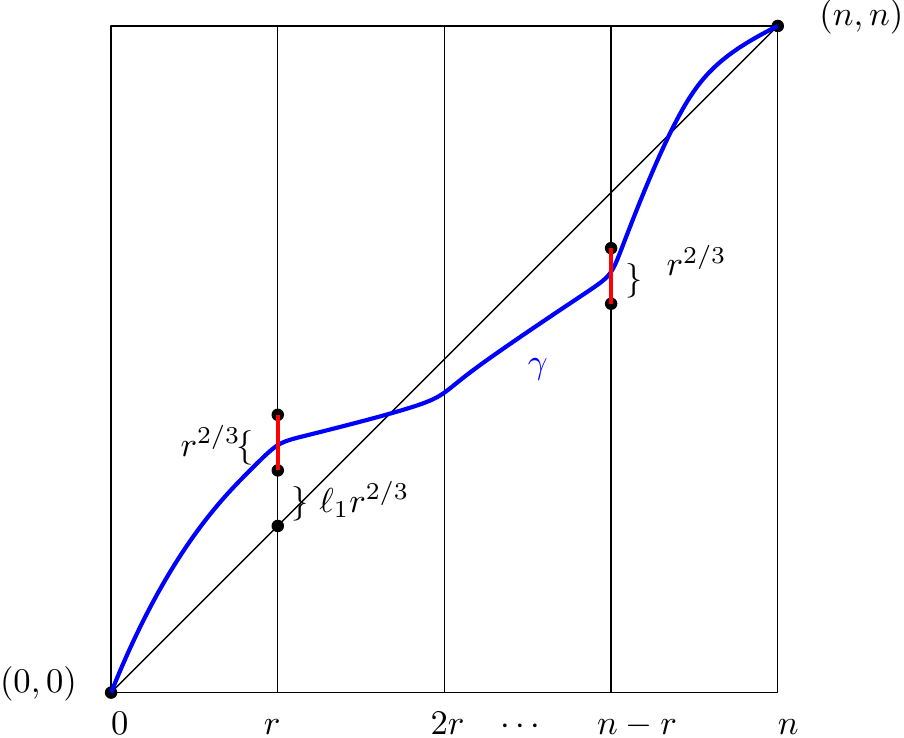}
\caption{Discretization of an increasing path $\gamma$. The line segments $J_{1}(\gamma)$ and $J_{n/r-1}(\gamma)$ are marked in red.}
\label{f:discrete}
\end{center}
\end{figure}

Set $\Delta^{\gamma}_i=\Delta_i(\mathcal{J}(\gamma))$. We define the {\bf total fluctuation} of a path $\gamma$ at scale $r$ to be equal to $\sum_{i} |\Delta^{\gamma}_i|$. We shall need the following easy counting lemma which gives a bound on the number of $J(\gamma)\in \mathscr{I}$ that correspond to an increasing path $\gamma$ of a given total fluctuation. We omit the proof.

\begin{lemma}
\label{l:countmaxpath}
Let $\mathscr{I}(T)=\{\mathcal{J}\in \mathscr{I}: \sum_{i}|\Delta^i(\mathcal{J})|\leq T\}$. Then $|\mathscr{I}(T)|\leq 4^{n/r+T}$. Further if $T\geq \ell\frac{n}{r}$, then $|\mathscr{I}(T)|\leq e^{c(\ell)(n/r+T)}$ where $c(\ell)\rightarrow 0$ as $\ell \rightarrow \infty$.
\end{lemma}
%
Now we show that the topmost maximal path w.h.p. has total fluctuation of the order of $\frac{n}{r}$.

\begin{lemma}
\label{l:goodgradient}
Let $\Gamma$ be the topmost maximal path from $\mathbf{0}$ to $\mathbf{n}$ in $\Pi$. We have w.h.p., $$\sum_i |\Delta^{\Gamma}_i|\leq \frac{\tilde{C}n}{r}.$$
\end{lemma}

\begin{proof}
Observe that by Theorem \ref{t:transversal} it suffices to restrict our attention to the case $\{\sup |\Gamma_x-x|\leq n^{3/4}\}$. Let $\mathscr{I}^*\subseteq \mathscr{I}$ be the set of all sequences $\mathcal{J}(\gamma)$ corresponding to all increasing paths $\gamma$ from $\mathbf{0}$ to $\mathbf{n}$ such that $\{\sup |\gamma_x-x|\leq n^{3/4}\}$. Denote the set of all such $\gamma$ by $\mathscr{G}_*$.

For an increasing path $\gamma$ in $\mathscr{G}_*$ set $\hat{X}_{\gamma}=\ell_{\gamma}^{\Pi}-2n$. It is clear that for each increasing such $\gamma$ we have

\begin{equation}
\label{e:goodgradient1}
\hat{X}_{\gamma}\leq \sum_{i=0}^{n/r} \sup_{u\in J_i(\gamma), u'\in J_{i+1}(\gamma)} \hat{X}_{u,u'}.
\end{equation}

First observe that for all $\gamma$ in $\mathscr{G}_*$ and all $i$, the slope of the line segment joining any $u\in J_i(\gamma)$ to any $u'\in J_{i+1}(\gamma)$ is in $(\frac{2}{\psi}, \frac{\psi}{2})$, by our choice of values taken by $r$. From Corollary \ref{c:deviation} it follows that for $u\in J_i(\gamma)$ and $u'\in J_{i+1}(\gamma)$, we have since $\tilde{C}$ is sufficiently large
$$\hat{X}_{u,u'}- \tilde{X}_{u,u'}\leq \frac{\tilde{C}}{100}r^{1/3}- ((|\Delta_i|-1)\vee 0)^2\frac{r^{1/3}}{100\psi^{3/2}}\leq \frac{\tilde{C}}{50}r^{1/3}- 5|\Delta_i|r^{1/3}.$$

Now let $\mathscr{G}_{T}\subseteq \mathscr{G}_{*}$ be the set of increasing paths from $\mathbf{0}$ to $\mathbf{n}$ such that $\sum_{i}|\Delta^{\gamma}_{i}|=T$. It follows that for each $\gamma\in \mathcal{G}_{T}$, we have

\begin{equation}
\label{e:goodgradient2}
\hat{X}_{\gamma}\leq \left(\frac{\tilde{C}n}{50r} -5T\right)r^{1/3}+\sum_{i=0}^{n/r} \sup_{u\in J_i(\gamma), u'\in J_{i+1}(\gamma)} \tilde{X}_{u,u'}.
\end{equation}

Fix $T\geq \frac{\tilde{C}n}{r}$. Using the exponential tails of $\tilde{X}$ established in Proposition \ref{t:treesup}, we conclude that for some absolute constant $c>0$ we have for each $\gamma \in \mathscr{G}_{T}$,
$$\P\left(\sum_{i=0}^{n/r} \sup_{u\in J_i({\gamma}), u'\in J_{i+1}({\gamma})} \hat{X}_{u,u'} \geq 4Tr^{1/3}\right)\leq K^{n/r} e^{-4cT}\leq e^{-2cT}$$
for some constant $K>0$ where the last inequality follows because $\tilde{C}$ is sufficiently large (depending on $K$). Now using Lemma~\ref{l:countmaxpath} and taking a union bound over all $\gamma\in \mathscr{G}_{T}$, since $\tilde{C}$ is sufficiently large, we get for $T\geq \frac{\tilde{C}n}{r}$,

$$\P(\Gamma\in \mathscr{G}_{T}, \hat{X}_{\Gamma}>-\frac{T}{100}r^{1/3})\leq e^{-cT}.$$

Summing over all $T$, and noticing that $\P(\hat{X}_{\Gamma}<-\ell n^{1/3})\rightarrow 0$ as $\ell \rightarrow \infty$ completes the proof of the lemma.
\end{proof}

\begin{lemma}
\label{l:fluccond1}
Let $r\in \mathcal{R}$ be fixed. For $x\in \mathcal{X}_r$, let $A_x$ denote the event that for all $x'$ with $|x'-x|\leq \frac{r}{2}$, we have $|\Gamma_x-\Gamma_{x'}-(x-x')|\leq \frac{M}{10}r^{2/3}$. Then we have
$$\P\left[\sum_{x\in \mathcal{X}_r}\mathbf{1}_{A_x^{c}}\geq \frac{n}{10000r}\right]\leq e^{-cn/r}+o(1)$$
for some constant $c>0$.
\end{lemma}
\begin{proof}
Let us fix $\mathcal{J}\in \mathscr{I}_*$ such that there exists an increasing path $\gamma \in \cup_{\ell \leq \tilde{C}n/r}\mathscr{G}_{\ell}$ with $\mathcal{J}(\gamma)=\mathcal{J}$ such that $\sup_x|\gamma_x-x|\leq n^{3/4}$. Choose $M$ sufficiently large so that $M\geq 10^6\tilde{C}$. From Markov's inequality it follows that $$\#\{i:|j_{i+1}-j_{i}|\geq \frac{M}{50}\}\leq \frac{n}{20000r}.$$
For $u\in J_i= ir\times (ir+j_ir^{2/3},ir+(j_i+1)r^{2/3})$ and $u'\in J_{i+1}= (i+1)r\times ((i+1)r+j_{i+1}r^{2/3}, i(r+1)+(j_{i+1}+1)r^{2/3})$ we say $F_{u,u'}$ holds if all the maximal paths between $u$ and $u'$ are contained in $\mathcal{P}((i+\frac{1}{2})r,r,(j_{i}-\frac{M}{50})r^{2/3}, (j_{i+1}+\frac{M}{50})r^{2/3})$. Define
$$A_{i,\mathcal{J}}=\bigcap_{u\in J_i,u'\in J_{i+1}}F_{u,u'}.$$
It follows from Corollary \ref{c:transversaldiffslope} that $\P[A_{i,\mathcal{J}}]\geq 1-\epsilon(M)$ where $\epsilon(M)$ can be made arbitrarily small by taking $M$ sufficiently large. Since these are independent events for different values of $i$ it follows that $\P[\sum_{i}\mathbf{1}_{A_{i,\mathcal{J}}^{c}}\geq \frac{n}{20000r}]\leq 10^{-\tilde{C}n/r}$ with $M$ sufficiently large.
Now notice that if $\sum_{i}\mathbf{1}_{A_{i,\mathcal{J}}^{c}}\leq \frac{n}{20000r}$, then we have $\sum_{x\in \mathcal{X}_r}\mathbf{1}_{A_x^{c}}\leq \frac{n}{10000r}$. The lemma now follows by taking $M$ sufficiently large, taking a union bound over $\mathcal{J}$ and using Lemma \ref{l:goodgradient} and Lemma \ref{l:countmaxpath}. 
\end{proof}

Next we want to prove similar results but instead of the central column of the butterfly at a location $x$, we are now concerned with the wings. Since the wings are not disjoint we need to adapt the arguments using some standard dependent percolation techniques. We want to prove the following.

\begin{lemma}
\label{l:fluccond2}
Let $r\in \mathcal{R}$ be fixed. For $x\in \mathcal{X}_r$, let $C_x$ denote the event that for all $x'$ with $|x'-x|\leq (1/2+L^{3/2})r$, we have $|\Gamma_x-\Gamma_{x'}-(x-x')|\leq \frac{1}{10}L^{11/10}r^{2/3}$. Then we have
$$\P\left[\sum_{x\in \mathcal{X}_r}\mathbf{1}_{C_x^{c}}\geq \frac{n}{10000r}\right]\leq e^{-cn/L^{3/2}r}+o(1)$$
for some constant $c>0$.
\end{lemma}

We first need the following lemma, where we are doing a different discretization of increasing paths into strips of width $L^{3/2}r$.

\begin{lemma}
\label{l:wingfluc1}
Let $\gamma$ be an increasing path from $\mathbf{0}$ to $\mathbf{n}$. For a fixed $r$ and for $i\in \{0,1,\ldots, \frac{n}{L^{3/2}r}\}$, let us define $Y^{\gamma}(i)=\lfloor \frac{\gamma_{irL^{3/2}}-irL^{3/2}}{Lr^{2/3}} \rfloor$. Let $\tilde{J}^{\gamma}_i$ denote the line segment
\[
\tilde{J}^{\gamma}_i=\{(iL^{3/2}r,y'):Y^{\gamma}(i)Lr^{2/3}\leq y'-irL^{3/2}\leq (Y(i)+1)Lr^{2/3}\}.
\]
We define $\tilde{\Delta}^{\gamma}_i= Y^{\gamma}(i+1)-Y^{\gamma}(i)$. Let $\tilde{\mathscr{I}}(T)$ denote the set of sequences of line segments $\{\tilde{J}_i\}_{0\leq i \leq \frac{n}{L^{3/2}r}}$ such that  $\sum_{i}|\tilde{\Delta} _i|\leq T\}$. Then $|\tilde{\mathscr{I}}(T)|\leq 4^{\frac{n}{L^{3/2}r}+T}$. Also let $\Gamma$ be the topmost maximal path from $\mathbf{0}$ to $\mathbf{n}$. Then with high probability, $\sum_i |\Delta^{\Gamma}_i|\leq \frac{\tilde{C}n}{L^{3/2}r}$.
\end{lemma}

\begin{proof}
The proof of this lemma is identical to the proofs of Lemma \ref{l:countmaxpath} and Lemma \ref{l:goodgradient} and we omit the proof.
\end{proof}

\begin{lemma}
\label{l:fluccondwing2}
Assume the set-up of Lemma \ref{l:wingfluc1}. Fix $\tilde{\mathcal{J}}=\{\tilde{J}_i\}\in \tilde{\mathscr{I}}(\tilde{C})$ with $\tilde{J}_i= x_i\times (x_i+y_i, x_i+y_i+Lr^{2/3})$. For a fixed $i$, consider the parallelogram $U_i$ whose corners are $(x_i,x_i+y_i-L^{21/20}r^{2/3})$, $(x_i,x_i+y_i+L^{21/20}r^{2/3})$, $(x_{i+3},x_{i+3}+y_{i+3}-L^{21/20}r^{2/3})$, $(x_{i+3},x_{i+3}+y_{i+3}+L^{21/20}r^{2/3})$. Call $i$  `bad' if at least one of the  following two conditions fail to hold.
\begin{enumerate}
\item[(i)] $\tilde{\Delta}_i+\tilde{\Delta}_{i+1}+\tilde{\Delta}_{i+2}\leq 10^6\tilde{C}$.
\item[(ii)] For all $u\in \tilde{J}_i$ and for all $u'\in \tilde{J}_{i+3}$, all the maximal paths from $u$ to $u'$ is contained in $U_i$ (call this event $\mathcal{D}_i$).
\end{enumerate}
Then we have $$\P[\#\{i:i~\text{is bad}\}\geq \frac{n}{20000Lr^{3/2}}]\leq 10^{-\tilde{C}n/L^{3/2}r}e^{-cn/L^{3/2}r}$$ for some constant $c>0$.
\end{lemma}

\begin{proof}
Since $\tilde{\mathcal{J}}\in \tilde{\mathscr{I}}(\tilde{C})$, by Markov's inequality it follows that deterministically  $$\#\{i:i~\text{is bad for failing $(i)$}\}\leq \frac{n}{50000L^{3/2}r}.$$ Also notice that it follows from Corollary \ref{c:transversaldiffslope} that  $\P[\mathcal{D}_i]\geq 1-\epsilon(L)$, where $\epsilon(L)$ can be made arbitrarily small by taking $L$ sufficiently large. Also notice that for each fixed $k\in \Z/3\Z$, the family of events $\{D_{3j+k}\}$ are independent. A large deviation bound followed by a union bound then shows that
$$\P\left[\#\{i:i~\text{is bad for failing $(ii)$}\}\geq \frac{n}{50000L^{3/2}r}\right] \leq 10^{-\tilde{C}n/L^{3/2}r}e^{-cn/L^{3/2}r}$$
since $L$ is sufficiently large, which completes the proof of the lemma.
\end{proof}

\begin{proof}[Proof of Lemma \ref{l:fluccond2}]
Assume the set up of Lemma \ref{l:wingfluc1} and Lemma \ref{l:fluccondwing2}. It is clear that if $\{\tilde{J}_i\}=\{J_i^{\Gamma}\}$ and $i$ is good (i.e., $i$ is not bad), then we have the following. for each $x\in \mathcal{X}_r$ with $x_{i+1}\leq x \leq x_{i+2}$, we have that $C_x$ holds. The lemma now follows from Lemma \ref{l:wingfluc1}, Lemma \ref{l:fluccondwing2} and a union bound over $\tilde{\mathscr{I}}(\tilde{C})$.
\end{proof}

\section{Probability bounds for $G_x$}
\label{s:gx}

Let $r\in \mathcal{R}$ be fixed. In this section, our task is to prove that for a large fraction of $x\in \mathcal{X}_r$, $\P(G_x)$ is close to $1$.
We shall prove the following theorem.

\begin{theorem}
\label{t:gxeverywhere}
For all $n$ sufficiently large we have
$$\P[\#\{x\in \mathcal{X}_r: G_x~\text{does not hold}\}\geq \frac{1}{1000}|\mathcal{X}_r|]\leq 10^{-3}.$$
\end{theorem}

We shall need the following corollary of Theorem \ref{t:gxeverywhere}.

\begin{corollary}
\label{t:gxbound}
There exist $\mathcal{X}_r^*\subseteq \mathcal{X}_r$ with $|\mathcal{X}_r^*|\geq \frac{9}{10} |\mathcal{X}_r|$ such that for all $x\in \mathcal{X}_r^{*}$ we have for all $n$ sufficiently large
$$\P(G_x)\geq \frac{95}{100}.$$
\end{corollary}

\begin{proof}
It follows from Theorem \ref{t:gxeverywhere} that
$$\sum_{x\in \mathcal{X}_r}\P(G_x)\geq \frac{995}{1000}|\mathcal{X}_r|.$$
Corollary \ref{t:gxbound} follows immediately.
\end{proof}

Since the condition $G_x$ has many components we will need a few steps to prove Theorem \ref{t:gxeverywhere}. The general strategy is the following. Since the conditions defining $G_{x,y}$ are all typical, we first show that for a fixed location $(x,y)$, with probability close to $1$, $G_{x,y}$ holds. Hence, by a large deviation estimate, for any increasing path $\gamma$, these events holds at most locations along $\gamma$ with exponentially small failure probability. Now by taking a union bound over all potential maximal paths (the size of this union bound is controlled by the results of the previous section) we get the result.  

For the rest of this section $\chi$ shall denote a small positive constant which can be made arbitrarily small by taking $C$ sufficiently large.

\subsection{Bounding Probabilities of $G_{x,y}$}
First we need to prove that for a fixed $x\in \mathcal{X}_r$, and $y\in r^{2/3}\Z$, $G_{x,y}$ holds with large probability. We have the following lemma.
%

\begin{lemma}
\label{l:gxylocbound}
For $x\in \mathcal{X}_r$, $y\in r^{2/3}\Z$, we have
\begin{equation}
\label{e:gxyprob1}
\P(G_{x,y}^{\rm loc})\geq 1-\chi.
\end{equation}
\end{lemma}

\begin{proof}
It sufficies to prove that for a fixed $x\in \mathcal{X}_r$, $y\in r^{2/3}\Z$, each of the $5$ conditions defining $G_{x,y}^{\rm loc}$ holds with probability at least $1-\frac{\chi}{10}$. We analyse each of the conditions separately.

\textbf{Condition 1:} Let $U=\mathcal{P}(x,y-Mr^{2/3}/10,r,2Mr^{2/3}/10)$. Let $A_{x,y,1}$ denote the event that for all $(u',u'')\in \mathcal{S}(U)$ we have $|\tilde{X}_{u',u''}^{\partial U}|\leq Cr^{1/3}$. It follows from Proposition \ref{t:treesup} and Proposition \ref{t:treebox} that $\P[A_{x,y,1}^{c}]\leq \frac{\chi}{10}$ for $C$ sufficiently large.

\textbf{Condition 2:} Let $A_{x,y,2}$ denote the event that $\forall (u,u') \in \mathcal{S}(\mathfrak{C})$, we have $|\tilde{X}_{u,u'}|\leq CL^{1/2}r^{1/3}$. It follows from Corollary \ref{c:treeinfwidth} and Corollary \ref{c:treesupwidth} that $\P[A_{x,y,2}^{c}]\leq \frac{\chi}{10}$ since $C$ sufficiently large.

\textbf{Condition 3:} Let $A_{x,y,3}$ denote the event that for all $(u,u')\in \mathcal{S}(\Lambda)$ we have $|\tilde{X}_{u,u'}|\leq Cr^{1/3}$. It follows from Proposition \ref{t:treesup} and Proposition \ref{t:treeinf} that $\P[A_{x,y,3}^{c}]\leq \frac{\chi}{20}$ since $C$ sufficiently large.

Now define the following parallelograms. Let $U_1=\mathcal{P}(x-\frac{9r}{40}, y-2Mr^{2/3}, \frac{7r}{20}, 4Mr^{2/3})$, $U_2=\mathcal{P}(x+\frac{9r}{40}, y-2Mr^{2/3}, \frac{7r}{20}, 4Mr^{2/3})$, $U_3=\mathcal{P}(x, y-2Mr^{2/3}, \frac{4r}{5}, (M-\frac{1}{10})r^{2/3})$, $U_4=\mathcal{P}(x, y-Mr^{2/3}, \frac{4r}{5}, Mr^{2/3})$. Let us define the following events. For $i=1,2$, let
$$B_{x,y,i}=\{\forall (u,u')\in \mathcal{S}(U_i)~|\tilde{X}_{u,u'}|\leq \frac{Cr^{1/3}}{10}\}.$$
For $i=3,4$, let
$$B_{x,y,i}=\{\forall (u,u')\in \mathcal{S}(U_i)~|\tilde{X}^{D}_{u,u'}|\leq \frac{Cr^{1/3}}{10}\}.$$

Since $C$ sufficiently large it follows from Proposition \ref{t:treeinf}, Proposition \ref{t:treesup} and Proposition \ref{t:treebox} $\P[\cup_{i=1}^4 B_{x,y,i}^c]\leq \frac{\chi}{100}$.

Now we show that on $\cap_{i=1}^{4} B_{x,y,i}$, condition (\ref{e:gxloc4}) holds. This is illustrated in Figure \ref{f:davoiding}.

\begin{figure}[ht!]
\begin{center}
\includegraphics[width=0.3\textwidth]{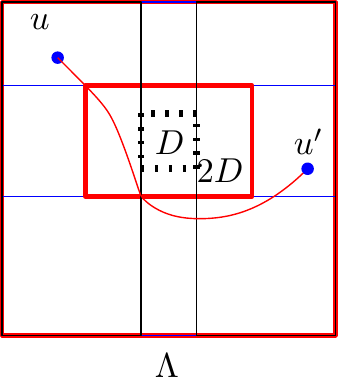}
\caption{A path from $u$ to $u'$ avoiding $D$ as in Condition 3 of the proof of Lemma \ref{l:gxylocbound}.}
\label{f:davoiding}
\end{center}
\end{figure}

To prove this let us fix $(u,u')\in \mathcal{S}(\Lambda\setminus 2D)$ satisfying the hypothesis of the condition. Without loss of generality let us assume $u\in U_1$. There are several cases depending on the position of $u$. If $u'\in U_1$, on $B_{x,y,1}$ it follows that $|\tilde{X}_{u,u'}|\leq \frac{Cr^{1/3}}{10}$. If $u'\in U_3$ (resp. $U_4$) and $u$ is also in $U_3$ (resp. $U_4$), then also it follows that on $\cap_{i=1}^{4} B_{x,y,i}$, $|\tilde{X}_{u,u'}|\leq Cr^{1/3}/10$. Next let us consider the case where $u'\in U_2$. Clearly there exists $u_1^*\in U_1\cap U_3$, $u_2^*\in U_2\cap U_3$ such that $(u,u_1^*)\in \mathcal{S}(U_1)$, $(u_1^*,u_2^*)\in \mathcal{S}(U_3)$ and $(u_2^*,u')\in \mathcal{S}(U_2)$ such that it follows using Lemma \ref{l:penaltysum} that

$$|\tilde{X}_{u,u'}^{D}-\tilde{X}_{u',u_1^*}-\tilde{X}_{u_1^*,u_2^*}^{D}-\tilde{X}_{u_2^*,u}|\leq \frac{C}{2}r^{1/3}$$
since $C$ is sufficiently large. It follows that $|\tilde{X}_{u,u'}|\leq Cr^{1/3}$ on $\cap_{i=1}^{4} B_{x,y,i}$. All other cases can be dealt with similarly and it follows that condition (\ref{e:gxloc4}) holds with probability at least $1-\frac{\chi}{20}$.

\textbf{Condition 4:} Let $A_{x,y,4}$ denote the event that $\forall u\in U=\mathcal{P}(x,y-2Mr^{2/3}, \frac{r}{5}, 2Mr^{2/3})$, $u'\in  B_2^*\cap \mathfrak{C}$ we have $\hat{X}_{u,u'}\leq Cr^{1/3}$. Clearly it suffices to show that $\P[A_{x,y,4}]\geq 1-\frac{\chi}{100}$.



Let $C_{x,y}$ denote the event that for all $u\in U,u'\in \Lambda\cap B_2^{*}$ we have $\hat{X}_{u,u'}\leq \frac{C}{2}r^{1/3}$. Clearly since $C$ is sufficiently large we have $\P[C_{x,y}^c]\leq \frac{\chi}{1000}$.  Now let us define the points $u_j=(x+2r/5, x+2r/5+y-2Mr^{2/3}-jMr^{2/3})$ for $\frac{L} {M} j\geq 0$. Let $C_{x,y,j}$ denote the event that for all $u\in 2D$ and for all $u'$ on the line segment $L_j$ joining $u_j$ and $u_{j+1}$, we have $\hat{X}_{u',u_*}\leq \frac{C}{2}r^{1/3}$. Notice that it follows from Lemma \ref{l:penalty} that since $M$ is sufficiently large we have that for all $u\in U$, and for all $u'\in L_j$, $\hat{X}_{u,u'}\leq \tilde{X}-jr^{1/3}$. Hence it follows from Proposition \ref{t:treesup} that for some constant $c>0$, we have $\P[C_{x,y,j}^{c}]\leq e^{-c(C+j)}$ since $C$ is sufficiently large. Hence it follows that $\P[(\cap_j C_{x,y,j})^c]\leq \frac{\chi}{1000}$. It now follows that $\P[A_{x,y,4}^c]\leq \frac{\chi}{100}$.

\textbf{Condition 5:} Let $A_{x,y,5}$ denote the event that $\forall u,u'\in F$ we have $|\tilde{X}_{u,u'}^{F^{+}}|\leq Cr^{1/3}$.
Using Proposition \ref{t:treebox} it follows that $\P[A_{x,y,5}^c]\leq \frac{\chi}{10}$.


Putting together all the steps above it follows that $\P(G_{x,y}^{\rm loc})\geq 1-\chi$ which completes the proof of the lemma.
\end{proof}

\begin{lemma}
\label{l:gxyareabound}
For $x\in \mathcal{X}_r$, $y\in r^{2/3}\Z$, we have for all $n$ sufficiently large
\begin{equation}
\label{e:gxyprob1area}
\P(G_{x,y}^{\rm a})\geq 1-\chi.
\end{equation}
\end{lemma}

We first need the following lemma.

\begin{lemma}
\label{l:areabound}
Consider the rectangle $U_{h}$ whose opposite corners are $(0,0)$ and $(h, mh)$ where $0.99 \leq m \leq 1.01$. Let  $A(\eta)$ denote the event that there exists an increasing path $\gamma$ from $(0,0)$ to $(h, mh)$ such that $\ell_{\gamma}\in [h,3h]$ and $A_{\gamma}\leq \eta h$. For a fixed absolute constant $\eta>0$ and for $h$ sufficiently large we have $\P(A(\eta))\leq e^{-h}$.
\end{lemma}

\begin{proof}
Notice that it suffices to take $\ell_{\gamma}=\ell$ fixed in the statement of the lemma, since then we can take a union bound over different $\ell \in [h,3h]$. Without loss of generality let us take $\ell_{\gamma}=h$ in the statement of the lemma. Let us first divide $U_h$ into the following subrectangles.
For $i,j\in \{0,1,\ldots, \frac{h}{10}-1\}$, we define $D_{i,j}$ to be the rectangle whose opposite corners are given by $(10i, 10mj)$ and $(10(i+1), 10m(j+1))$.

Let $\mathbb{H}$ denote the set of all oriented paths in $\Z^2$ from $(0,0)$ to $(\frac{h}{10}-1,\frac{h}{10}-1)$. Clearly $|\mathbb{H}|\leq 2^{h}$. For $H\in \mathbb{H}$, let $\mathcal{N}_{H}$ denote the set of all nonnegative integer valued sequences $\{N_{i,j}\}_{(i,j)\in H}$ with $\sum_{(i,j)\in H}N_{i,j}=h$. It is clear that $|\cup_{H\in \mathbb{H}}\mathcal{N}_{H}|\leq 20^{h}$. Now fix $H\in \mathbb{H}$ and $\{N_{i,j}\}\in \mathcal{N}_{H}$.

Let $A(\eta, \{N_{i,j}\})$ denote the event that there exists an increasing path $\gamma$ in $U_{h}$ from $(0,0)$ to $(h,mh)$ with  $\ell_{\gamma}=h$ and $A_{\gamma}\leq \eta h$, and such that $\gamma$ contains exactly $N_{i,j}$ many points in $D_{i,j}$. Observe that, on $A(\eta, \{N_{i,j}\})$, there must exist $\sum_{(i,j)\in H} (N_{i,j}-1)\geq \frac{h}{2}$ points on $\gamma$, such that the point (say $u$) and the next point on $\gamma$ (say $u'$) belong to the same subrectangle $D_{i,j}$ for some $(i,j)\in
H$.

%
%

Now for $D_{i,j}$, let $U_{i,j}$ denote the number of points $u\in D_{i,j}$ such that there is a point $u'\neq u$ in $D_{i,j}$ such $A(u,u')\leq 10\eta$. It follows that on $A(\eta, \{N_{i,j}\})$
$$A(\gamma)\geq 10\eta \times \left(\frac{h}{2}-\sum_{(i,j)\in H, } U_{i,j} \right).$$
Hence it suffices to show that
\begin{equation}
\label{e:area3}
\P\left[\sum_{(i,j)\in H} U_{i,j}\geq \frac{2h}{5}\right]\leq (20e)^{-h}.
\end{equation}

First observe that $\{U_{i,j}\}, (i,j)\in H$ is an independent sequence of random variables. Also observe that
$$\E(e^{20U_{i,j}})\rightarrow 1$$
as $\eta\rightarrow 0$ by the DCT. Since $\eta$ is chosen sufficiently small we have
$$\E(e^{20U_{i,j}})\leq  2.$$
The independence of $U_{i,j}$'s and Markov's inequality then establishes (\ref{e:area3}). This completes the proof of the lemma.
\end{proof}

\begin{proof}[Proof of Lemma \ref{l:gxyareabound}]
It follows from Lemma \ref{l:areabound} and taking a union bound over different pairs of points $(u,u')$ that $\P[G_{x,y}^{{\rm a}}] \geq 1-\frac{\chi}{2}$. 
The lemma follows.
\end{proof}

\begin{lemma}
\label{l:gxyrs}
For each $x\in \mathcal{X}_r$, $y\in r^{2/3}\Z$, we have $\P(G_{x,y}^{\rm rs})\geq 1-\chi$.
\end{lemma}

\begin{proof}
For $i=0,1,2,\ldots \frac{1}{100}\varepsilon^{-5/3}$, let
$$A_{i}=\{\forall (u,u')\in \mathcal{S}(\tilde{D}_i) |\tilde{X}^{(i)}_{u,u'}| \leq \frac{\delta}{8}r^{1/3}\};$$

$$B_{i}=\{\forall (u,u')\in \mathcal{S}(\tilde{D}_i) |\tilde{X}^{(i-1)}_{u,u'}| \leq \frac{\delta}{8}r^{1/3}\}.$$

It is clear from Proposition \ref{t:treesup} that by taking $\varepsilon$ sufficiently small depending on $\delta$ and $C$, we have that $\P[A_i^{c}\cup B_i^{c}]\leq e^{-2C/\varepsilon^{1/4}}$.

Clearly, on $A_i\cap B_i$, we have
$$\Delta_i \leq \sup_{u,u'\in \mathcal{S}(\tilde{D}_i)} (|\tilde{X}^{(i)}_{u,u'}|+|\tilde{X}^{(i-1)}_{u,u'}|)\leq  \frac{\delta}{4}r^{1/3}.$$

It follows by taking a union bound over all $i$ we get
$$\P(\max_i \Delta_i \geq \frac{\delta}{2}r^{1/3})\leq \varepsilon^{-2}e^{-2C/\varepsilon^{1/4}}\leq e^{-C/\varepsilon^{1/4}}\chi^{-1}.$$
by choosing $\varepsilon$ small enough. It follows now from Markov's inequality that

$$\P\left[\P[\max_i \Delta_i \geq \frac{\delta}{2}r^{1/3}\mid \Pi]\geq e^{-C/\varepsilon^{1/4}}\right]\leq \chi.$$
This completes the proof of the lemma.
\end{proof}

For $x\in \mathcal{X}_r$, let $y(\Gamma, x)=\inf_{y\in r^{2/3}\Z}\{y+x\geq \Gamma_x\}$. We have the following proposition.

\begin{proposition}
\label{p:gxyunion}
For all $n$ sufficiently large we have,
$$\P[\#\{x\in \mathcal{X}_r: G_{x,y(\Gamma,x)}^{{\rm loc}}\cap G_{x,y(\Gamma,x)}^{{\rm a}}\cap G_{x,y(\Gamma,x)}^{{\rm rs}} ~\text{does not hold}\}\geq \frac{1}{10000}|\mathcal{X}_r|]\leq 10^{-4}.$$
\end{proposition}

\begin{proof}
The proposition follows from Lemma \ref{l:gxylocbound}, Lemma \ref{l:gxyareabound}, Lemma \ref{l:gxyrs}, the fact that $G_{x,y}^{{\rm loc}}\cap G_{x,y}^{{\rm a}}\cap G_{x,y}^{{\rm rs}}$ are independent events for different values of $x$, a Chernoff bound using $\chi$ sufficiently small and a union bound using Lemma \ref{l:goodgradient} and Lemma \ref{l:countmaxpath}.
\end{proof}

\subsection{Proof of Theorem \ref{t:gxeverywhere}}
To prove Theorem \ref{t:gxeverywhere} we still need to estimate the probabilities of the {\bf wing condition} and the {\bf fluctuation condition}.

\begin{proposition}
\label{p:gxgammawing}
Let $\Gamma$ be the topmost maximal path in $\Pi$ from $\mathbf{0}$ to $\mathbf{n}$. For $x\in \mathcal{X}_r$, let $y(\Gamma,x)=\inf_{y}\{y\in r^{2/3}\Z: y\geq \Gamma_x\}$. Then
$$\P[\#\{x\in \mathcal{X}_r: G_{x,y(\Gamma,x)}^{w}~\text{holds}\}<\frac{9999}{10000}|\mathcal{X}_r|]\leq 10^{-4}.$$
\end{proposition}

The proof of Proposition \ref{p:gxgammawing} is similar to the proof of Proposition \ref{p:gxyunion} but we need to work harder as the \emph{Wings} are not disjoint for diffirent values of $x\in \mathcal{X}_r$. We first need the following lemma.

\begin{lemma}
\label{l:wingfluc2}
Assume the set-up of Lemma \ref{l:wingfluc1}. For $\tilde{\mathcal{J}}=\{\tilde{J}_j\}\in \tilde{\mathscr{I}}=\tilde{\mathscr{I}}(\infty)$ with $\tilde{J}_j= x_j\times (x_j+y_j, x_j+y_j+Lr^{2/3})$, let $W^*_{1,j}=\mathcal{P}(x_j-\frac{3L^{3/2}r}{2},y_j-10L^{11/10}r^{2/3}, 3L^{3/2}r, 20L^{11/10}r^{2/3})$ and $W^*_{2,j}=\mathcal{P}(x_j+\frac{3L^{3/2}r}{2},y_j-10L^{11/10}r^{2/3}, 3L^{3/2}r, 20L^{11/10}r^{2/3})$. Let $\mathcal{X}_{r,j}=\{x\in \mathcal{X}_r: x_j-L^{3/2}r \leq x \leq  x_j+L^{3/2}r\}$. Let $\gamma$ be an increasing path from $\mathbf{0}$ to $\mathbf{n}$ such that $\tilde{J}_j^{\gamma}=\tilde{J}_j$
 For $x\in \mathcal{X}_{r,i}$, let $A_x^{\gamma}$ denote the event that $W^1(\mathbb{B}(x,y(x,\gamma), r))\cup W^1(\mathbb{B}(x,y(x,\gamma), r))\subseteq W^*_{1,i}\cup W^*_{2,i}$. Call $i$ 'good for $\gamma$' if $\cap_{x\in \mathcal{X}_{r,i}} A_x^{\gamma}$ holds. Then for $L$ sufficiently large, $$\P[\#\{i:i~\text{`good' for}~\Gamma\}\leq (1-10^{-5})\frac{n}{L^{3/2}r}]\leq e^{-cn/L^{3/2}r}+o(1).$$
\end{lemma}
\begin{proof}
The proof is essentially similar to the proof of Lemma \ref{l:fluccond2} and we omit the details.
\end{proof}

\begin{lemma}
\label{l:wingfluc3}
Fix $\tilde{\mathcal{J}}=\{\tilde{J}_j\}\in \tilde{\mathscr{I}}$ and define $W^*_{1,i}$ and $W^*_{2,i}$ as in Lemma \ref{l:wingfluc2}. Let $A_i$ denote the event that for all $(u,u')\in \mathcal{S}(W^*_{1,i}\cup W^*_{2,i})$, we have $|\tilde{X}_{u,u'}|\leq \frac{C}{2}L^{3/4}r^{1/3}$. Then $$\P[\#\{i:A_i~\text{does not hold}\}\geq 10^{-6}\frac{n}{L^{3/2}r}]\leq 10^{-\tilde{C}n/L^{3/2}r}e^{-cn/L^{3/2}r}.$$
\end{lemma}

\begin{proof}
Notice that $A_{i_1}$ and $A_{i_2}$ are independent if $i_1-i_2\geq 6$. more generally, we also have $\{A_{6i+k}\}_{i\geq 0}$ is independent for each $k=0,1,2,3,4,5$. By Corollary \ref{c:treesupwidth} and Corollary \ref{c:treeinfwidth} it follows that for each $i$, $\P[A_i]\geq 1-\chi$, where $\chi$ can be made arbitrarily small by choosing $C$ sufficiently large. It follows that for a fixed $k\in \Z/6\Z$ we have with exponentially high probability, $\#\{i:A_{6i+k}~\text{does not holds}\}\leq \frac{1}{10^6}\frac{n}{L^{3/2}r}$. The lemma follows by taking a union bound over $k\in \Z/6\Z$.
\end{proof}

\begin{proof}[Proof of Proposition \ref{p:gxgammawing}]
Observe the following. If $i$ is `good' for $\Gamma$ and $A_i$ holds for $\tilde{J}=\{\tilde{J}_j^{\Gamma}\}$ then $G_{x,y(\Gamma,x)}^{w}$ holds for all $x\in \mathcal{X}_{r,i}$. The proposition now follows similarly to Proposition \ref{p:gxyunion} by taking a union bound over $\tilde{J}$ and using Lemma \ref{l:wingfluc1}, Lemma \ref{l:wingfluc2} and Lemma \ref{l:wingfluc3}.
\end{proof}
%

\begin{proposition}
\label{p:gxgammafluctuation}
We have for all $n$ sufficiently large
$$\P[\#\{x\in \mathcal{X}_r: G_{x}^{f}~\text{holds}\}<\frac{9998}{10000}|\mathcal{X}_r|] \leq 10^{-4}$$
\end{proposition}

\begin{proof}
Since $C$ is sufficiently large, it follows from Theorem \ref{t:moddevlowertail}, Proposition \ref{t:treesup} and Theorem \ref{t:transversal} that with probability at least $1-10^{-5}$ the first condition in the definition of $G_x^{f}$ holds for all $x$ . It follows from Lemma \ref{l:fluccond1} and Lemma \ref{l:fluccond2} that
$$\P[\#\{x\in \mathcal{X}_r: G_{x}^{f}~\text{ does not hold hold for failing }~(\ref{e:gxmaxnbhd1})\}>\frac{2n}{10000r}]\leq 10^{-5}.$$
The lemma follows.
\end{proof}

Now we are ready to prove Theorem \ref{t:gxeverywhere}.

\begin{proof}[Proof of Theorem \ref{t:gxeverywhere}]
Observe that for $x\in \mathcal{X}_r$ if $G_{x,y(x,\Gamma)}^{{\rm loc}}$ holds and $G_x^{f}$ holds then $G_x^{{\rm loc}}$ holds. The theorem now follows from Proposition \ref{p:gxyunion}, Proposition \ref{p:gxgammawing} and Proposition \ref{p:gxgammafluctuation}.
\end{proof}


\section{Probability bounds on $R_{x}$, $H_x$, $Q$ and the conditional measure}
\label{s:rxcond}
In this section, we work out estimates of probabilities of $R_x$ and prove Theorem \ref{l:rdecreasing}, and also estimates for probabilities of $H_x$ conditional on $R_x$ and ultimately prove Theorem \ref{p:gxhxconditional}. We shall also prove Theorem \ref{t:steep}.

\subsection{Bounds on $R_x$}
First we prove Theorem \ref{l:rdecreasing}. We start with the following proposition.

\begin{proposition}
\label{p:rxy}
For each $x\in \mathcal{X}_r$ and for each $y\in r^{2/3}\Z$ we have $\P[R_{x,y}]>\beta>0$ where $\beta$ is a constant independent of $r$.
\end{proposition}

This proposition will follow from the next two lemmas.

\begin{lemma}
\label{l:rxycond1}
Let $x\in \mathcal{X}_r$ and $y\in r^{2/3}\Z$ be fixed. Let $R_{1,xy}$ denote the event that $\forall u=(x',y')\in B_1^{*}$ with $y'\geq y-Mr^{2/3}$ and $\forall u'\in \partial^{+}(B_1^{*})$ we have $\hat{X}_{u,u'}^{(B_1^{*})^{c}}\leq Cr^{1/3}$. Then we have $\P[R_{1,xy}]\geq 99/100$ for $C$ sufficiently large.
\end{lemma}

\begin{proof}
Let $U=\mathcal{P}=(x-\frac{9r}{20}, y-Mr^{2/3}, \frac{r}{10}, 3Mr^{2/3})$. Let $A^{1}_{x,y}$ denote the event that for all $(u,u')\in \mathcal{S}(U)$ we have $|\tilde{X}_{u,u'}|\leq \frac{C}{10}r^{1/3}$. Let $A^2_{x,y}$ denote the event that for all $u,u'$ in the line segment joining $(x-\frac{r}{2}, x-\frac{r}{2}+y-Lr^{2/3})$ and  $(x-\frac{4r}{5}, x-\frac{4r}{5}+y-Lr^{2/3})$ (i.e., the bottom boundary of $B_1^*$) we have $|\tilde{X}_{u,u'}|\leq \frac{C}{10}r^{1/3}$. It follows from Proposition \ref{t:treeinf} and Proposition \ref{t:treesup} that $\P[A^{1}_{x,y}\cap A^{2}_{x,y}] \geq 1-10^{-3}$ since $C$ is sufficiently large.

Let $L_{U}$ denote the left boundary of $U$. For $0 \leq \ell \leq L-M$ let $L_{\ell}$ denote the vertical line segment joining $(x-\frac{4r}{5}, x-\frac{4r}{5}+y-Mr^{2/3}-\ell r^{2/3})$ and $(x-\frac{4r}{5}, x-\frac{4r}{5}+y-Mr^{2/3}-(\ell+1)r^{2/3})$. Let $A_{\ell}$ denote the event that for all $u\in L_{U},u'\in L_{\ell}$ we have $\tilde{X}_{u,u'}\leq \frac{C+\ell}{10} r^{1/3}$. It follows from Proposition \ref{t:treesup} that $\P[A_{\ell}^{c}]\leq e^{-c(C+\ell)}$ for some absolute constant $c>0$. It follows by taking a union bound over all $\ell$ that $\P[\cap_{\ell} A_{\ell}]\geq \frac{999}{1000}$ since $C$ is large enough.

It suffices now to show that $$A^1_{x,y}\cap A^2_{x,y} \cap \left(\bigcap_{\ell}A_{\ell}\right) \subseteq R_{1,xy}.$$

To show this observe that if $u\in L_U$ and $u'$ is on the right boundary of $B_1^*$, this follows from Lemma \ref{l:penalty} since $C$ is sufficiently large. Otherwise, set $u_1\in L_U$ such that the line joining $u_1$ and $u$ has slope 1. Set $u_2=u'$ if $u'$ is on the right boundary of $B_1^*$, otherwise set $u_2=(x-\frac{4r}{5}, x-\frac{4r}{5}+y-Lr^{2/3})$. Then observe that
$$ \hat{X}_{u,u'} \leq \hat{X}_{u_1,u_2}-\hat{X}_{u_1,u}-\hat{X}_{u',u_2}. $$

The lemma now follows from the definition of $A^1_{x,y}$ and $A^2_{x,y}$ and Lemma \ref{l:penalty}.
%
%
%
\end{proof}

\begin{lemma}
\label{l:rxycond2}
Let $x\in \mathcal{X}_r$ and $y\in r^{2/3}\Z$ be fixed. Let $R_{2,xy}$ denote the event that for all $u\in B_1^{*}\cap W^{1}$, $\forall u'\in B_1\cap \mathfrak{C}$ we have $\tilde{X}_{u,u'}^{(B_1^{*})^{c}}\leq -C^*r^{1/3}$. Then we have $\P[R_{2,xy}]\geq \beta'>0$ where $\beta'$ is independent of $r$.
\end{lemma}

\begin{proof}
For $s,t=0,1,\ldots (L+2M)$, define points $u_{s}=(x-r/2,x-r/2+y-Lr^{2/3}+sr^{2/3})$, and $u'_{t}=(x-2r/5, x-2r/5 + y-Lr^{2/3}+tr^{2/3})$. Let $L_{1}^{s}$ denote the line segment joining $u_{s}$ and $u_{s+1}$ and $L_2^{t}$ denote the line segment joining $u'_t$ and $u'_{t+1}$. Let $A_{s,t}$ denote the following event.

$$A_{s,t}=\{\sup_{u\in L_1^{s}, u'\in L_2^{t}}\tilde{X}_{u,u'}\leq -C^*r^{1/3}\}.$$

\begin{figure}[ht!]
  \centering
  \begin{tabular}{cp{.05\textwidth}c}
  \hspace{-1cm}
     \includegraphics[width=.3\linewidth]{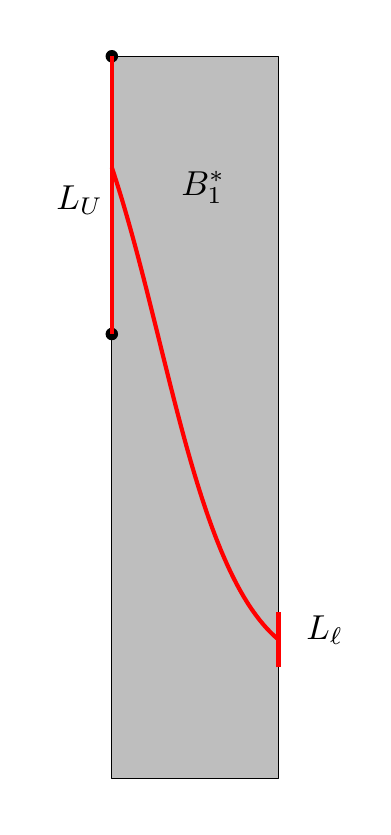} 
     \hspace{-2.75cm}(a)
     \hspace{1cm}
     &&
     \includegraphics[width=.3\linewidth]{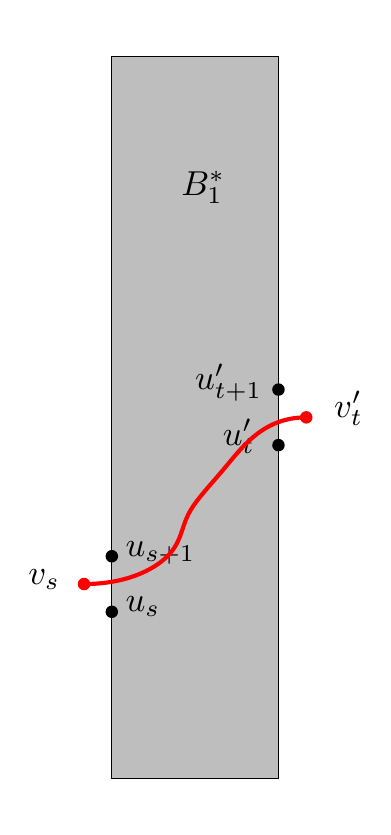}
      \hspace{-2.75cm}(b)
      \\
  \end{tabular}
\caption{Constructions in the proofs of Lemma \ref{l:rxycond1} and Lemma \ref{l:rxycond2}. In Lemma \ref{l:rxycond1} (see (a)) we ask for each $\ell$, a maximal path from any point in $L_{U}$ to any point in $L_{\ell}$ is not too much longer than it's typicall length. For Lemma \ref{l:rxycond2} (see (b)), for each pair $s,t$, we ask that the maximal path between $v_{s}$ and $v'_{t}$ has much smaller than typical length which implies that the maximal path from any point between $u_{s}$ and $u_{s+1}$ to any point between $u'_{t}$ and $u'_{t+1}$ is also likely to have a shorter than typical length.}
\label{f:rcond}
\end{figure}

%
%

First we prove that $\P[A_{s,t}]$ is bounded away from $0$ uniformly in $r$. Fix $\theta>0$. Define the points $v(s)=(x-r/2-\theta r, x-r/2-\theta r+y-Lr^{2/3}+(s+\frac{1}{2})r^{1/3})$ and $v'(t)=(x-2r/5+\theta r, x-2r/5+\theta r+ y-Lr^{2/3}+(t+\frac{1}{2})r^{1/3})$. Observe that for $C^*$ sufficiently large, for all $u\in L_1^{s}$ and all $u'\in L_2^{t}$ we have
$$\tilde{X}_{v(s),u}+\tilde{X}_{u,u'}+\tilde{X}_{u',v'(t)}\leq \tilde{X}_{v(s),v'(t)}-\frac{C^*r^{1/3}}{10}.$$

It clearly follows that
\begin{eqnarray*}
\label{e:r2cond}
\P[\sup_{u\in L_1^{s},u'\in L_2^{t}}\tilde{X}_{u,u'} &\leq & -C^*r^{1/3}]\geq \P[\tilde{X}_{v(s),v'(t)}\leq -2C^*r^{1/3}]-\P[\inf_{u\in L_1^{s}}\tilde{X}_{v(s),u}\leq -C^*r^{1/3}/5]\\
&-&\P[\inf_{u'\in L_2^{s}}\tilde{X}_{u',v'(t)}\leq -C^*r^{1/3}/5].
\end{eqnarray*}

Now observe that by Theorem \ref{BDJ99}, it follows that there exists a constant $\kappa$ (depending on $C^*$) such that for all sufficiently large $r$, we have $\P[\tilde{X}_{v(s),v'(t)}\leq -2C^*r^{1/3}] > 10\kappa$. By choosing $\theta$ sufficiently small and using Proposition \ref{t:treeinf} we get that $\P[A_{s,t}]\geq \kappa >0$. Now notice that since $A_{s,t}$ is a decreasing event for all $s$ and $t$, by the FKG inequality it follows that
$$\P[R_{2,xy}]=P[\cap_{s,t}A_{s,t}]\geq \kappa^{(L+2M)^2}\geq \beta'>0$$
where $\beta'$ is independent of $r$. This completes the proof of the lemma.
\end{proof}

\begin{proof}[Proof of Proposition \ref{p:rxy}]
Observe that since $R_{1,xy}$ and $R_{2,xy}$ are both decreasing events, it follows by the FKG inequality that that $\P[R_{1,xy}\cap R_{2,xy}]\geq \frac{99 \beta'}{100}$. By symmetry we establish the same bounds for the right barrier $B_2^{*}$ and since the two barriers are independent it follows that $\P[R_{xy}]\geq \frac{9(\beta') ^2}{10}\geq \beta >0$, which completes the proof of the Proposition.
\end{proof}

Now we are ready to prove Theorem \ref{l:rdecreasing}.

\begin{proof}[Proof of Theorem \ref{l:rdecreasing}]
The first inequality follows from FKG inequality by noting that both $\{\Gamma=\gamma\}$ and $R_{x,\gamma}$ are decreasing events on the configuration on $\R^2\setminus \{\gamma\}$. The second inequality is trivial and the last inequality was already proved in Proposition~\ref{p:rxy}.
\end{proof}

\subsection{Probability bounds for $H_x$}
\begin{theorem}
\label{t:hxmustar}
For each $x\in \mathcal{X}_r$, we have $\mu_{x}^{*}(H_x)\geq \frac{98}{100}$.
\end{theorem}

To prove Theorem \ref{t:hxmustar} we need the following Proposition.

\begin{proposition}
\label{p:hxybound}
For $x\in \mathcal{X}_r$, $y\in r^{2/3}\Z$, we have
\begin{equation}
\label{e:hxyprob1}
\P(H_{x,y})\geq 98/100.
\end{equation}
\end{proposition}

This proposition will follow from the next two lemmas.

\begin{lemma}
\label{l:hxycond1}
For $x\in \mathcal{X}_r$, $y\in r^{2/3}\Z$ and $\mathbb{B}(x,y,r)$, let $A_1$ denote the event that for all $u,u'\in F$, we have $^{\Lambda}\tilde{X}_{u,u'}\leq -Lr^{1/3}$. Then we have $\P[A_1]\geq 999/1000$.
\end{lemma}

\begin{proof}
Let $u_1=(x-r/2,x-r/2+y-Lr^{2/3})$ and $u_2=(x+r/2, x+r/2+y-Lr^{2/3})$. Let $B$ denote the event that for all $u,u'\in F$ (note that $F$ is the line segment joining $u_1$ and $u_2$), $\tilde{X}_{u,u'}\geq -Cr^{1/3}$. Let $G$ denote the event that $^{\Lambda}\tilde{X}_{u_1,u_2}\leq -2Lr^{1/3}$. It is then clear that $A\supseteq B\cap G$ since $L$ is sufficiently large.

Now it follows from Proposition \ref{t:treeinf} that $\P[B]\geq 1-10^{-5}$ since $C$ is sufficiently large. Now let $G_1$ (resp. $G_2$) denote the event that for all $u\in \Lambda$, we have $\tilde{X}_{u_1,u}\leq C\sqrt{L}r^{1/3}$  (resp. $\tilde{X}_{u,u_2}\leq C\sqrt{L}r^{1/3}$). Observe that since $C$ and $L$ are sufficiently large we have the
$$\tilde{X}_{u_1,u_2}\leq \tilde{X}_{u_1,u}+ \tilde{X}_{u,u_2}-10Lr^{2/3}. $$
Using the above fact and Corollary \ref{c:treeinfwidth} it follows that
$\P[G]\geq P[G_1\cap G_2]\geq 1-10^{-5}$ which completes the proof of the lemma.
\end{proof}

\begin{lemma}
\label{l:hxycond2}
For $x\in \mathcal{X}_r$, $y\in r^{2/3}\Z$ and $\mathbb{B}(x,y,r)$, let $A_2$ denote the event that for all $u\in F$, $u'\in U=\mathcal{P}(x,y-2Mr^{2/3},r,3Mr^{2/3})$  we have $\hat{X}_{u,u'}\leq -Lr^{1/3}$ if $u<u'$ and $\hat{X}_{u',u}^{\gamma}\leq -Lr^{1/3}$ if $u'<u$. Then $\P[A_2]\geq 99/100$.
\end{lemma}

\begin{proof}
Let $A_2^{*}$ denote the event that for all $u\in F$, all $u'\in U$, we have $\hat{X}_{u,u'}\leq -Lr^{1/3}$. Let $u_1$ be the leftmost point on $F$ as in the proof of Lemma \ref{l:hxycond1}. Let $R_{U}$ be the right boundary of $U$. Consider the following events
$$G_1=\{\forall u,u' \in F: \hat{X}_{u,u'}\geq -Cr^{1/3}\}.$$
$$G_2=\{\forall (u,u')\in \mathcal{S}(U): |\tilde{X}_{u,u'}|\leq Cr^{1/3}\}.$$
$$G_3=\{\forall u\in R_U: \hat{X}_{u_1,u} \leq -2Lr^{1/3}\}.$$
It can then be proved along the lines of Lemma \ref{l:rxycond1} that $\P[A_2^{*}]\geq 1-10^{-3}$. The other cases can be dealt with similarly and the lemma follows.
\end{proof}

%


\begin{proof}[Proof of Theorem \ref{t:hxmustar}]
Let $\Gamma$ be the topmost maximal path in $\Pi$ from $\mathbf{0}$ to $\mathbf{n}$ in $\Pi$. Fix $x\in \mathcal{X}_r$. Observe that, for an increasing path $\gamma$ from $\mathbf{0}$ to $\mathbf{n}$, the event $\{\Gamma=\gamma\}$, $R_{x,\gamma}$  $H_{x,\gamma}$ are all decreasing in the configuration on $\R^2\setminus \{\gamma\}$, Hence it follows from the FKG inequality and Lemma \ref{l:mustarbasic} that

$$\mu_x^{*}(H_x\mid \Gamma=\gamma)\geq \mu(H_x\mid \Gamma=\gamma)\geq \mu[H_{x,\gamma}]\geq  \min_{y} \P[H_{x,y}].$$
The theorem follows by averaging over $\gamma$ and using Proposition \ref{p:hxybound}.
\end{proof}

\subsection{Bound on $Q$ and $Q'$}
In this section we prove Theorem \ref{t:steep}.
We start with the following lemma.

\begin{lemma}
\label{l:steepatend}
An increasing path $\gamma$ from $\mathbf{0}$ to $\mathbf{n}$ is called to be \emph{steep at end} if either $\frac{\gamma_{n/10}}{n/10} \notin (\frac{15}{\psi}, \frac{\psi}{15})$ or  $\frac{n-\gamma_{9n/10}}{n/10} \notin (\frac{5}{\psi}, \frac{\psi}{5})$. Let $A$ denote the event that there exists a steep at end path $\gamma$  from $\mathbf{0}$ to $\mathbf{n}$ with $\ell_{\gamma}> 2n-n^{2/5}$. Then $\P[A]\leq e^{-n^{0.1}}$.
\end{lemma}

\begin{proof}
This lemma is proved by showing that since $\psi$ is large enough the expected length of an increasing path which is steep at end is much smaller than  the maximal increasing path and using Theorem \ref{t:moddevuppertail}. The proof is similar to Lemma \ref{l:nosteep} and we omit the details here.
\end{proof}

\begin{lemma}
\label{l:steepdiscretize}
Suppose $\gamma$ is a steep increasing path from $\mathbf{0}$ to $\mathbf{n}$ that is not steep at end. Then there exists $u_1=(x_1,y_1)$ and $u_2=(x_2,y_2)$ in $\Z^2 \cap [0,n]^2$ satisfying the following conditions.
\begin{enumerate}
\item $\frac{n}{10} \leq x_1 < x_2 \leq \frac{9n}{10}$.
\item Either $\frac{y_2-y_1}{x_2-x_1}\in (\frac{\psi}{10}, \frac{\psi}{2})$ or $\frac{y_2-y_1}{x_2-x_1}\in (\frac{2}{\psi}, \frac{10}{\psi})$.
\item $\frac{y_1}{x_1}, \frac{n-y_2}{n-x_2} \in (\frac{1.01}{\psi}, 0.99 \psi)$
\item $(x_2-x_1)\wedge (y_2-y_1) \geq \frac{n^{2/3}}{\log^8 n}$.
\item $\gamma_{x_1}\in [y_1, y_1+1)$ and $\gamma_{x_2}\in (y_2-1,y_2]$.
\end{enumerate}
\end{lemma}

\begin{proof}
This lemma follows from the definition of steep path and steepness at ends.
\end{proof}

A pair of points $u_1$ and $u_2$ in $\Z^2 \cap [0,n]^2$ satisfying the first 4 conditions in Lemma \ref{l:steepdiscretize} is called \emph{inadmissible}. We have the following lemma.

\begin{lemma}
\label{l:nosteep}
For a pair of inadmissible points $u_1=(x_1,y_1)$ and $u_2=(x_2,y_2)$, let $u'_1=(x_1,y_1+1)$ and $u'_2=(x_2,y_2-1)$.  Let $\mathcal{A}$ denote the event that there exists a pair $(u_1,u_2)$ of inadmissible points such that
$$X_{\mathbf{0},u'_1}+ X_{u_1,u_2}+X_{u'_2,\mathbf{n}}\geq 2n-n^{0.49}.$$
Then $\P[\mathcal{A}]\leq e^{-n^{0.1}}$.
\end{lemma}

\begin{proof}
Fix a pair $(u,u')$ of inadmissible points. Since $\psi$ is sufficiently large it follows from an elementary computation that
$$\E X_{\mathbf{0},u'_1}+ \E X_{u_1,u_2}+ \E X_{u'_2,\mathbf{n}} <2n-n^{0.5}.$$
The lemma now follows from using Theorem \ref{t:moddevuppertail} and taking a union bound over all pairs of inadmissible points.
\end{proof}

Now we are ready to prove Theorem \ref{t:steep}.

\begin{proof}[Proof of Theorem \ref{t:steep}]
From Lemma \ref{l:steepatend}, Lemma \ref{l:steepdiscretize} and Lemma \ref{l:nosteep} it follows that
\begin{equation}
\label{e:steep}
\P[Q^c] \leq e^{-n^{0.1}}
\end{equation}

Since resampling at a fixed location does not change the law of the point configuration it follows by using (\ref{e:steep}) and taking a union bound over $r\in \mathcal{R}$, $x\in \mathcal{X}_{r}$, $y\in r^{2/3}\Z$ and $i\in [\frac{1}{100\varepsilon^{5/3}}]$ that for $n$ sufficiently large
$$\P[\cup_{r,x,yi} \mathscr{S}_i^c] \leq e^{-n^{0.05}}.$$
It follows from Markov's inequality that
$$\P[\P[\cup_{r,x,y,i} \mathscr{S}_i^c\mid \Pi] \geq e^{-n^{1/100}}] \leq e^{-n^{0.02}}.$$

The theorem follows by a union bound.
\end{proof}

\subsection{Proof of Theorem \ref{p:gxhxconditional}}

Finally we are ready to prove Theorem \ref{p:gxhxconditional}.

\begin{proof}[Proof of Theorem \ref{p:gxhxconditional}]
Notice that by definition of $G_x$ does not depend on the configuration in the interior of the walls of  $\mathbb{B}(x,r)$. Hence it follows that $\mu_x^*(G_x)=\mu(G_x)$. The theorem now follows from Corollary \ref{t:gxbound}, Theorem \ref{t:hxmustar} and Theorem \ref{t:steep}.
\end{proof}

\section{First order and second order approximation of $X_{u,u'}$}
\label{s:moddev}

In this section we establish useful probability bounds on $X_{u,u'}$ for certain pairs of points $(u,u')$ using the moderate deviation estimates Theorem \ref{t:moddevlowertail} and Theorem \ref{t:moddevuppertail}. We shall mostly have to deal with pairs of points $u,u'\in \R^2$ with $u<u'$ such that the slope of the line joining $u$ and $u'$ is neither too large nor too small. We shall work with first and second order approximations of $\E[X_{u,u'}]$ in this case. We start with the following easy corollary of Theorem \ref{t:moddevuppertail} and Theorem \ref{t:moddevlowertail}.

\begin{corollary}
\label{c:moddevtilde}
Let $\psi>0$ be fixed. There exist constants $C_1=C_1(\psi)$, $r_0=r_0(\psi)$, $\theta_0=\theta_0(\psi)>0$ such that
for points  $u=(x,y)$ and $u'=(x',y')$  in $\R^2$ such that $x'-x=r \geq r_0$, and
$$\frac{1}{\psi} \leq \frac{y'-y}{x'-x} \leq \psi,$$
we have that
$$\E[X_{u,u'}]=2\sqrt{r(y-y')}+O(r^{1/3}).$$
Further, for $\theta>\theta_0$ we have
\begin{equation}
\label{e:moddevlowertailnew}
\P[|\tilde{X}_{u,u'}| > \theta r^{1/3}]\leq e^{-C_1\theta}.
\end{equation}
\end{corollary}


The following expression for $\E[X_{u,u'}]$ will be useful.

\begin{lemma}
\label{l:penalty}
Let $u=(x,y)<u'=(x',y')\in \R^2$ be such that $|x'-x|=r$ and $\frac{y'-y}{x'-x}=m$ where $m\in (\frac{1}{\psi}, \psi)$. Suppose $u_0=(x,y+h_0r^{2/3})$ and $u_1=(x',y'+h_1r^{2/3})$ be such that the slope of the line joining $u_0$ and $u_1$ is in  $(\frac{1}{\psi}, \psi)$, and $|h_1-h_0|\leq r^{1/10}$.  Then for $r$ sufficiently large  $$\E[X_{u_0,u_1}]=2\sqrt{m}r+ \frac{h_1-h_0}{\sqrt{m}}r^{2/3} +O(r^{1/3})-\frac{(h_{1}-h_0)^2}{4m^{3/2}}r^{1/3}.$$
\end{lemma}

\begin{proof}
Follows from Corollary \ref{c:moddevtilde} and observing that for $x\in (-1,1)$ we have $(1+x)^{1/2}=1+\frac{x}{2}-\frac{x^2}{8}+O(x^{3})$.
\end{proof}

The following corollary is a special case of Lemma \ref{l:penalty} which will be useful to us and hence we state it separately.

\begin{corollary}
\label{c:deviation}
In the set-up of Lemma \ref{l:penalty} with $m=1$ we have
$$\hat{X}_{u_0,u_1}\leq \tilde{X}_{u,u'}+ O(r^{1/3})-\frac{(h_{1}-h_0)^2}{8}r^{1/3} .$$
\end{corollary}

The quadratic term above may be viewed as a penalty term which is incurred for deviating from the straight line path as illustrated in the next lemma.

\begin{lemma}
\label{l:penaltyproper}
Let $u=(x,y)<u'=(x',y')\in \R^2$ be such that $|x'-x|=r$ and $\frac{y'-y}{x'-x}=m$ where $m\in (\frac{2}{\psi}, \frac{\psi}{2})$. Let $u_0=(x_0,y_0)=(x+\frac{r}{2},y+\frac{mr}{2}+hr^{2/3})$ be such that slope of the lines joining $u_0$ to $u$ and $u'$ are in  $(\frac{1}{\psi}, \psi)$. Then
$$ \E[X_{u,u_0}]+\E[X_{u_0,u'}]-\E[X_{u,u'}] \leq O(r^{1/3})-\frac{h^2}{8(m\vee 1)^{3/2}}r^{1/3}.$$
\end{lemma}

\begin{proof}
Proof is similar to that of Lemma \ref{l:penalty} and we omit the details.
\end{proof}

We also need the following similar lemma.
\begin{lemma}
\label{l:penaltysum}
Let $u=(x,y)<u'=(x',y')\in \R^2$ be such that $|x'-x|=r$ and $\frac{y'-y}{x'-x}=m$ where $m\in (\frac{2}{\psi}, \frac{\psi}{2})$. Consider points $u=u_0<u_1<u_2<\cdots <u_{\ell}=u'$ such that $u_{i}=(x_i,y_i)$, $y_i=(y_0+m(x_i-x_0)+h_i)$ where $|h_i|\leq h (|x_{i}-x_{i-1}|^{2/3}\wedge |x_{i+1}-x_{i}|^{2/3})$. Then there exists $r_0=r_0(\psi,h)>0$ and $\theta=\theta(\psi, h)>0$ such that if $\min_{i}|x_{i+1}-x_i|\geq r_0$, then
\[
\left|\left(\sum_{i} \E[X_{u_i,u_{i+1}}]\right )- \E[X_{u,u'}] \right|\leq \theta (r^{1/3}+\sum_{i} |x_{i+1}-x_{i}|^{1/3}).
\]
\end{lemma}

\begin{proof}
Follows from Corollary \ref{c:moddevtilde} and Lemma \ref{l:penalty}.
\end{proof}

\section{Bounds on path lengths between points in a parallelogram}
\label{s:para}
Observe that Theorem \ref{t:moddevlowertail} and Theorem \ref{t:moddevuppertail} provide us with nice tail bounds for point to point distances in the Poissonian last passage percolation environment. However, for our purposes we shall need to obtain similar estimates for 
$\sup_{u,u'} \tilde{X}_{u,u'}$ and $\inf_{u,u'} \tilde{X}_{u,u'}$ where $u$ and $u'$ are varied over points in a parallelogram of suitable length and height (such that the slope of the line joining $u$ and $u'$ is neither too small nor too large).

\subsection{Shorter paths are unlikely in a parallelogram}
We need the following notations to make a precise statement. Consider the parallelogram $U=U_{r,m,\ell}$ whose four corners are $(0, -\ell r^{2/3})$, $(0,\ell r^{2/3})$, $(r,mr-\ell r^{2/3})$, $(r, mr+\ell r^{2/3})$. Recall the definition of $\mathcal{S}(U)\subseteq U^2$. For $u=(x,y)$ and $u'=(x',y')\in U$, $(u,u')\in \mathcal{S}(U)$ iff $\frac{2}{\psi}< \frac{y'-y}{x'-x}\leq \frac{\psi}{2}$. We have the following proposition.

\begin{proposition}
\label{t:treeinf}
Consider the parallelogram $U=U_{r,1,1}$. There exists an absolute constant $c_1>0$, $r_0=r_0(\psi)>0$ and $\theta_0=\theta_0(\psi)>0$ such that we have for all $r>r_0$ and $\theta> \theta _0$
\begin{equation}
\label{e:treeinfgeneral}
\P\left(\inf_{(u,u')\in \mathcal{S}(U)} \tilde{X}_{u,u'}\leq -\theta r^{1/3}\right)\leq e^{-c_1\theta}.
\end{equation}
\end{proposition}

The proof of Proposition \ref{t:treeinf} is done in two steps. The first step is to prove the following easier lemma which asserts the statement of Proposition \ref{t:treeinf}, but only for pairs of points such that one is `close' to the left boundary of $U$ and the other is `close' to the right boundary of $U$. 

\begin{lemma}
\label{l:treeinfbasic}
Consider the parallelogram $U=U_{r,m,1}$ where $m\in (\frac{4}{\psi}, \frac{\psi}{4})$. Define $L(U)=U\cap \{x\leq r/8\}$ and $R(U)=U\cap \{x\geq 7r/8\}$. There exist constants $r_1>0$, $\theta_1>0, c_2>0$ such that for all $r>r_1$ and $\theta>\theta_1$ we have
\begin{equation}
\label{e:treeinfbasic2}
\P\left(\inf_{u\in L(U), u'\in R(U)} \tilde{X}_{u,u'}\leq -\theta r^{1/3}\right)\leq e^{-c_2\theta}.
\end{equation}
\end{lemma}

%
%

\begin{proof}
We shall restrict to the case $m=1$ without loss of generality, the same argument works for other values of $m$. Let $u_{*}=(\frac{r}{2},\frac{r}{2})$ denote the center of $U$. Observe using Lemma \ref{l:penaltysum} it follows that for all $u\in L(U)$, $u'\in R(U)$ we have
$$\left |\E X_{u,u_*} +\E X_{u_*, u'} -\E X_{u,u'}  \right|\leq \frac{\theta r^{1/3}}{3}$$
for $\theta$ sufficiently large. Hence it follows that  
$$\tilde{X}_{u,u'}\geq \tilde{X}_{u,u_{*}}+\tilde{X}_{u_*,u'} -\frac{\theta r^{1/3}}{3}$$
and by symmetry it suffices to prove that for $r$ sufficiently large 
\begin{equation}
\label{e:treeinfbasic1}
\P\left(\inf_{u\in L(U)} \tilde{X}_{u,u_*}\leq -\frac{\theta r^{1/3}}{3}\right)\leq e^{-c\theta}
\end{equation}
for some absolute constant $c>0$. This is what we shall establish.

Before proceeding with the proof of \eqref{e:treeinfbasic1}, let us first explain informally the idea of the argument. Suppose, for the moment, we are only interested in paths from the left boundary of $U$ to $u_{*}$. We shall define a sequence of points in the left half of $U$, which will form a planar tree rooted at $u_{*}$ and have a large number of leaves on the left boundary of $U$. See Figure \ref{f:tree} (For convenience, we draw it in the tilted co-ordinates so the parallelogram $U$ becomes a rectangle in the figure). We shall ask that along each edge $(u_1,u_2)$ of the tree, $\tilde{X}_{u_1,u_2}$ is not too small. We shall show that we can choose the points in this tree in such a way that (a) the above event holds with large probability and (b) on these events, $\inf_{u} \tilde{X}_{u,u_*}$ is not too small. The idea is to make the edges of the tree shorter and shorter as the points get closer to the left boundary of $U$. Since larger deviations become much more unlikely with decreasing edge length, it is possible to take a union bound over a larger set of points. Formally we do the following.   

\begin{figure}[ht!]
\begin{center}
\includegraphics[width=\textwidth]{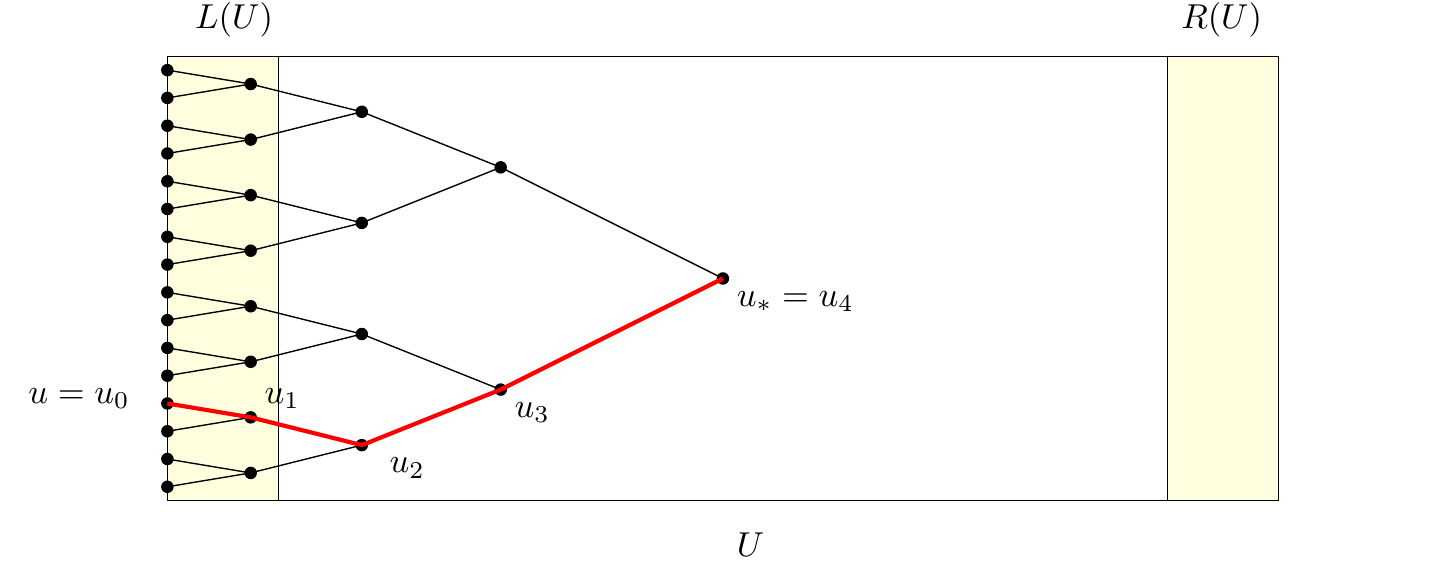}
\caption{A path from $u$ to $u_*$ along the tree, $X_{u,u_*}\geq \sum_{i=0}^{3} X_{u_i,u_{i+1}}$}
\label{f:tree}
\end{center}
\end{figure}

For $r$ sufficiently large and $r_0$ given by Corollary \ref{c:moddevtilde}, fix $a$ such that $r\gg a \gg r_0$ and $\frac{r}{2a}=8^{K}$ for some integer $K>0$. For each $k\in \{0,\ldots, K \}$ we define the following sets of points. Let $S_k = \{a\ell 8^{k}:\ell=0,1,\ldots, 8^{K-k}\}$ and $T_k =  \{a^{2/3}\ell 4^k: \ell=-2\times 4^{K-k},\ldots, 2\times 4^{K-k}\}$. Define $V_{k}$ to be the set of all points $(x,y)\in U$ such that $x\in S_{k}$ and $y-x\in T_{k}$.

At level $k$, define a graph $\mathcal{T}_k$ with the vertex set $V_k$ where $(x,y), (x',y')\in V_k$ is connected by an edge if $x\neq x'$, $|x-x'|\leq 20\cdot 8^ka$ and $|(y-y')-(x-x')|\leq 30\cdot 4^ka^{2/3}$. That is, $\mathcal{T}_{k}$ connects pairs of points in $V_k$ that are close by.

Let $\mathcal{E}_k$ denote the following event.
$$\mathcal{E}_k:=\biggl\{\tilde{X}_{v,v'}\geq -\frac{\theta r^{1/3}}{100}(1.5)^{k-K}~\forall (v,v')\in \mathcal{T}_k \biggr\}.$$

It follows from Lemma \ref{l:ek} below that for $r$ and $\theta$ sufficiently large we have
\begin{equation}
\label{e:ekclaim}
\biggl\{\inf_{u\in L(U)} \tilde{X}_{u,u_*}\geq -\frac{\theta r^{1/3}}{3}\biggr\} \supseteq \bigcap_{k=1}^{K} \mathcal{E}_k.
\end{equation}

To complete the proof of the lemma it remains to obtain a lower bound for $\P[\cap_k \mathcal{E}_k]$. From Corollary \ref{c:moddevtilde} we have for $(v,v')\in \mathcal{T}_k$ and for $\theta$ and $r$ sufficiently large
$$\P\left[\tilde{X}_{v,v'}\leq -\frac{\theta r^{1/3}}{100}(1.5)^{(k-K)}\right]\leq e^{-c\theta (4/3)^{K-k}}$$
for some absolute constant $c>0$. Now the number of edges $(v,v')\in \mathcal{T}_k$ is polynomial in $8^{K-k}$, so taking a union bound over all $(v,v')\in \mathcal{T}_{k}$ and over $k$ we get that $\P[\cap_{k}\mathcal{E}_k]\geq 1-e^{-c_3\theta}$ provided that $\theta$ and $r$ are sufficiently large. This establishes \eqref{e:treeinfbasic1}.  Proof of Lemma \ref{l:treeinfbasic} can then be completed as discussed above.
\end{proof}

It remains to establish \eqref{e:ekclaim}, which we do in the next lemma.
\begin{lemma}
\label{l:ek}
In the set-up of the above proof, for $r$ and $\theta$ sufficiently large we have 
$$\biggl\{\inf_{u\in L(U)} \tilde{X}_{u,u_*}\geq -\theta r^{1/3}\biggr\} \supseteq \bigcap_{k=1}^{K} \mathcal{E}_k.$$
\end{lemma}

We start with an informal sketch of the proof. Fix $u=u_{-1}=(x_{-1},y_{-1})\in L(U)$. Our objective is to find a sequence of points $\{u_i=(x_i,y_i)\}_{i=1}^{i_0}$ (see Figure \ref{f:tree}) such that 
\begin{enumerate}
\item[\rm i.] $u_0$ is very close to $u_{-1}$: the distance between $u_{-1}$ and $u_0$ is $O(a)$. 
\item[\rm ii.] For each $i\geq 0$, we have $u_i\in V_i$ and there is an edge in $\mathcal{T}_i$ between $u_i$ and $u_{i+1}$.  
\item[\rm iii.] $u_{i_0}=u_{*}$.
\end{enumerate}
The informal idea to construct such a sequence is as follows. Consider the line segment $\mathcal{L}$ joining $u$ and $u_{*}$. We construct the points recursively going from left to right. Suppose we have constructed up to point $u_i$. Then we look at the next vertical line to the right of $u_i$ on which points of $V_{i+1}$ lie. We look where $\mathcal{L}$ intersects this line and find a close by point on $V_{i+1}$. Since both the points $u_i$ and $u_{i+1}$ are not too far from $\mathcal{L}$, it can be shown that there exists an edge between $u_i$ and $u_{i+1}$ in $\mathcal{T}_{i}$. Also since the distance from the left boundary of $U$ to $u_i$ keeps increasing exponentially, eventually (say at step $i_0$) this becomes $\frac{r}{2}=a8^{K}$, so at this point we hit $u_*$, and set $u_{i_0}=u_{*}$. More formally we do the following.

\begin{proof}[Proof of Lemma \ref{l:ek}]
We assume $\bigcap_{k=1}^{K} \mathcal{E}_k$ holds and show that $\biggl\{\inf_{u\in L(U)} \tilde{X}_{u,u_*}\geq -\theta r^{1/3}\biggr\}$ holds.

Let $\mbox{Int}(z) = \lfloor z\rfloor$ if $z > 0$ and
$\lceil z \rceil$ if $z < 0$. Define points $u_i=(x_i,y_i)$ for $i\geq 0$ recursively as follows.

$$x_i=a\biggl(\lfloor \frac{x_{i-1} 8^{-i}}{a}\rfloor+ 1\biggr)8^i;\qquad y_i =x_i+  \mbox{Int}\left((y_{-1}-mx_{-1})\left(\frac{r/2-x_i}{r/2-x_{-1}}\right)a^{-2/3}4^{-i}\right)a^{2/3}4^i.$$

Observe that by the above definition $0<x_{i+1}-x_i<8^{i+1}a$ and 
$$|(y_{i+1}-y_i)-(x_{i+1}-x_i)|\leq 5\cdot 4^ia^{2/3}+ 20\frac{8^ia}{r^{1/3}}\leq 25 \cdot 4^ia^{2/3}$$
and hence there exists an edge in $\mathcal{T}_i$ between $u_i$ and $u_{i+1}$ and the points $\{u_i\}$ satisfy the conditions i.-iii. described above. 

Notice that since the distance between $u_{-1}$ and $u_0$ is $O(a)$ it follows that if $r\gg a$ we have
$$\tilde{X}_{u_{-1},u_0}\geq -O(a)\geq -\frac{\theta}{100}r^{1/3}.$$
We can lower bound $X_{u,u_{*}}$ by,
$$X_{u,u_{*}}\geq \sum_{i=0}^{i_0} X_{u_{i-1},u_i}.$$
So to obtain a lower bound on $\tilde{X}_{u,u_{*}}$ we need to obtain a bound on 
$$\left|\sum_{i=0}^{i_0}\E X_{u_{i-1},u_i}-\E X_{u,u_{*}}\right|.$$ 
To this end we apply Lemma \ref{l:penaltysum} and using $x_{i+1}-x_i< 8^{i+1}a$ get that for $\theta$ sufficiently large, on $\bigcap_{k=1}^{K} \mathcal{E}_k$ we have
$$\tilde{X}_{u,u*}- \sum_{i=0}^{i_{0}}\tilde{X}_{u_{i-1},u_{i}}\geq -\frac{\theta}{100}r^{1/3}-\frac{\theta}{100}\sum_{i=0}^{i_{0}}2^ia^{1/3} \geq -\frac{\theta}{3}r^{1/3}.$$
This completes the proof.
\end{proof}


Finally we are ready to give the proof of Proposition \ref{t:treeinf}. The idea is to cover the parallelogram $U$ with a number of smaller parallelograms $U^*$ such that for any pair of points $(u,u')\in S(U)$, there exists a $U^*$ such that $u\in L(U^*)$ and $u'\in R(U^*)$, and then use Lemma \ref{l:treeinfbasic} for the parallelograms $U^*$.

\begin{proof}[Proof of Proposition \ref{t:treeinf}]
Pick $r$ sufficiently large such that $r^{1/6}\gg r_1$ where $r_1$ is given by Lemma \ref{l:treeinfbasic}. Let $U=U_{r,1,1}$ be as in the statement of the proposition and let us define the following sets of points in $U$. For $k\geq 0$, define
$$S_{k}=\{\ell r2^{-3}2^{-k/4}: \ell=0,1,\ldots , 8\times2^{k/4}\}$$
and
$$T_{k}=\{\ell r^{2/3}2^{-k/6}: \ell=-2^{k/6}, \ldots , 2^{k/6}\}.$$

Consider the set of points $V_{k} \subseteq U$ such that $(x,y)\in V_{k}$ iff $x\in S_k$ and $y-x\in T_k$. Now consider the following parallelograms with vertices in $V_{k}$ having width $2^{-k/4}r$ and height $2^{-k/6}r^{2/3}$. For $\ell \in [8\times 2^{k/4}]$, $s_1,s_2\in \{-2^{k/6},\ldots ,2^{k/6}\}$ with $s_1<s_2$ and 
$$\frac{4}{3\psi} <1+\frac{s_2-s_1}{2^{-k/12}r^{1/3}} < \frac{3\psi}{4}$$
define $U^*_{k,\ell, s_1,s_2}$ as the parallelogram with vertices 
$(\ell 2^{-3}2^{-k/4}r, s_1 2^{-k/6}r^{2/3})$, $(\ell 2^{-3}2^{-k/4}r, (s_1+1)2^{-k/6}r^{2/3})$, $((\ell+8)2^{-3}2^{-k/4}r, s_2 2^{-k/6}r^{2/3})$ and $((\ell+8)2^{-3}2^{-k/4}r, (s_2+1)2^{-k/6}r^{2/3})$. See Figure \ref{f:paracover}, which again we have drawn in the tilted co-ordinates.

\begin{figure}[ht!]
\begin{center}
\includegraphics[width=\textwidth]{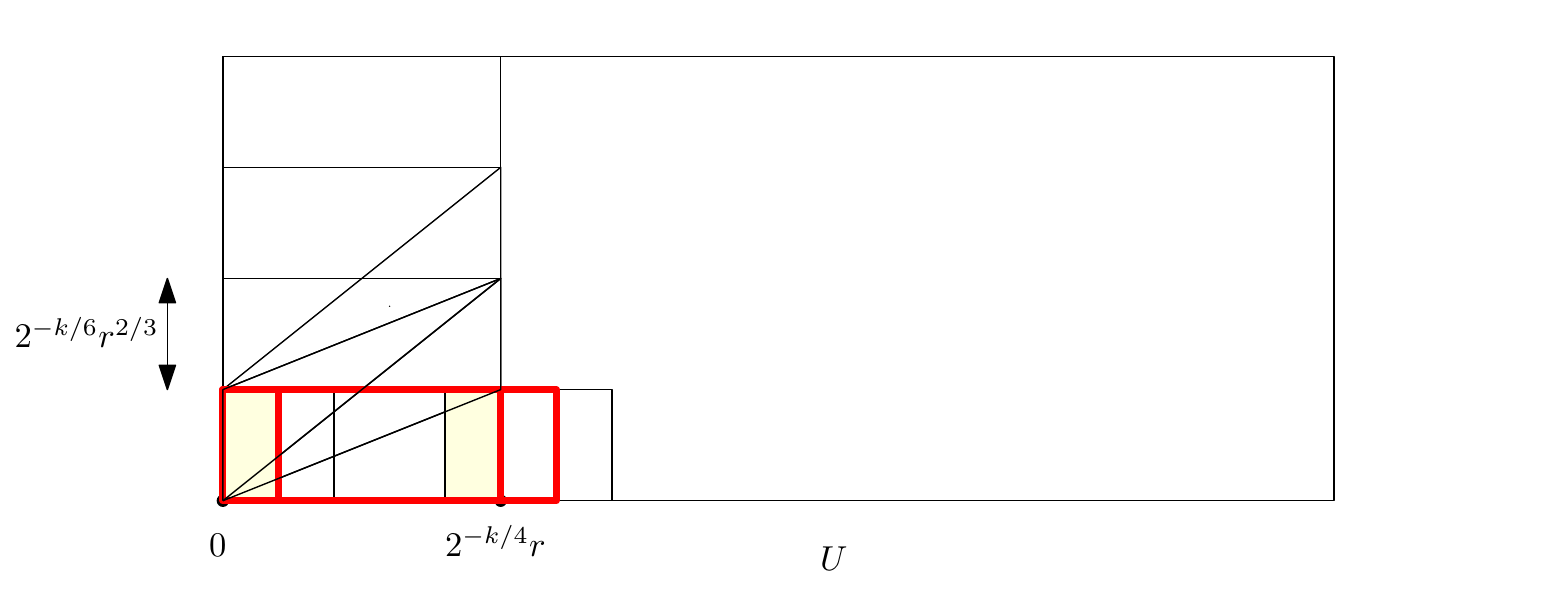}
\caption{Parallelograms constructed in the proof of Proposition \ref{t:treeinf}: Boundaries of $\tilde{U}=U^*_{k,0,-2^{k/6}, -2^{k/6}}$ and $U^*_{k,1,-2^{k/6}, -2^{k/6}}$ are marked in red; $L(\tilde{U})$ and $R(\tilde{U})$ are also marked}
\label{f:paracover}
\end{center}
\end{figure}
Denote the family of such parallelograms at level $k$ by $\mathcal{U}_{k}$. Note that for $r$ sufficiently large, any pair of points $u=(x,y)$ and $u'=(x',y')$ with $u<u'$ such that $(u,u')\in \mathcal{S}(U)$ and $|x-x'|\in [0.75\times 2^{-k/4}r, 2^{-k/4}r]$ there exists $U^*\in \mathcal{U}_{k}$ such that $u\in L(U^*)$ and $u'\in R(U^*)$. 

Let $\mathcal{G}_k$ denote the following event.
$$\mathcal{G}_{k}=:\left\{\forall U^*\in \mathcal{U}_k, \inf_{u\in L(U^*), u'\in R(U^*)} \tilde{X}_{u,u'}\geq -\theta r^{1/3}\right\}.$$

Observe that, if $(u,u')\in \mathcal{S}(U)$ is such that $|x-x'|\leq r^{1/6}$ then we must have $\tilde{X}_{u,u'}\geq -\theta r^{1/3}$ for $r$ sufficiently large. Hence it follows that  
$$\bigcap_{k=0}^{\lceil \frac{20\log r}{6\log 2} \rceil}\mathcal{G}_{k} \subseteq \left\{\inf_{(u,u')\in \mathcal{S}(U)}\tilde{X}_{u,u'}\geq -\theta r^{1/3}\right\}.$$

It remains to estimate $\P[\mathcal{G}_{k}]$.
Notice that $|\mathcal{U}_k|\leq 8^{k+1}$. Using Lemma \ref{l:treeinfbasic} and a union bound it follows that for $r$ sufficiently large (with $r^{1/6}\gg r_1$) and for all $\theta$ sufficiently large we have for all $k$
$$\P[\mathcal{G}_k^{c}]\leq 8^{k+1}e^{-c_2\theta2^{k/12}}.$$

Taking a union bound over $k\in \{0,1,\ldots , \frac{20\log r}{6\log 2}\}$ we get the assertion of the proposition.
\end{proof}

Proposition \ref{t:treeinf} has the following immediate corollary.

\begin{corollary}
\label{c:treeinfwidth}
Consider $U=U_{r,1,\ell}$ with $\ell >1$. There exists an absolute constant $c_1>0$, $h_0>0$ and $\theta_0=\theta_0(\psi)>0$ such that we have for all $h>h_0$ and $\theta> \theta _0$
\begin{equation}
\label{e:treeinfgeneralwidth}
\P\left(\inf_{(u,u')\in \mathcal{S}(U)} \tilde{X}_{u,u'}\leq -\theta\sqrt{\ell} r^{1/3}\right)\leq e^{-c_1\theta}.
\end{equation}
\end{corollary}

\subsection{Longer paths are unlikely too}
In this subsection we prove results analogous to the the previous subsection concerning upper tails of $\sup_{u,u'} \tilde{X}_{u,u'}$ where the supremum is taken over `most' points in appropriate parallelograms. Recall the notation $U_{r,m,\ell}$ from the previous subsection. We have the following proposition.

\begin{proposition}
\label{t:treesup}
Consider the parallelogram $U=U_{r,m,1}$ where $\frac{4}{\psi}<m<\frac{\psi}{4}$. There exists an absolute constant $c_1>0$, $r_0>0$ and $\theta_0>0$ such that we have for all $r>h_0$ and $\theta> \theta _0$
\begin{equation}
\label{e:treesupgeneral}
\P\left(\sup_{(u,u')\in \mathcal{S}(U)} \tilde{X}_{u,u'}\geq \theta r^{1/3}\right)\leq e^{-c_1\theta}.
\end{equation}
\end{proposition}

Observe that $r_0, \theta_0$ and $c_1$ in the above proposition can be taken to be the same as in Proposition \ref{t:treeinf}. The proof of Proposition \ref{t:treesup} follows from the following lemma in an identical manner to the proof of  Proposition \ref{t:treeinf} using Lemma \ref{l:treeinfbasic}. We omit the proof.

\begin{lemma}
\label{l:treesupbasic}
Consider the parallelogram $U=U_{r,m,1}$ where $m\in (\frac{4}{3\psi}, \frac{3\psi}{4})$. Define $L(U)=U\cap \{x\leq r/8\}$ and $R(U)=U\cap \{x\geq 7r/8\}$ as before. There exist constants $r_1>0$, $\theta_1>0, c_2>0$ such that for all $r>r_1$ and $\theta>\theta_1$ we have
\begin{equation}
\label{e:treesupbasic1}
\P\left(\sup_{u\in L(U), u'\in R(U)} \tilde{X}_{u,u'}\geq \theta r^{1/3}\right)\leq e^{-c_2\theta}.
\end{equation}
\end{lemma}

Let us explain first the idea of the proof. We shall take points $v$ and $v'$ slightly to the left $L(U)$ and to the right of $R(U)$ respectively; see Figure \ref{f:treesup}. Observe that if $\sup_{u\in L(U), u'\in R(U)} \tilde{X}_{u,u'}$ is too large then at least one of the following three events must occur: (a) $\tilde{X}_{v,v'}$ is large, (b) $\inf_{u\in L(U)} \tilde{X}_{v,u}$ is small, or (c) $\inf_{u'\in R(U)} \tilde{X}_{u'v'}$ is small. Observe that (a) is unlikely by Theorem \ref{t:moddevuppertail} and (b) and (c) are unlikely by Proposition \ref{t:treeinf}. It will follow from this that it is unlikely that $\sup_{u\in L(U), u'\in R(U)} \tilde{X}_{u,u'}$ is large. Formally we have the following.

\begin{figure}[ht!]
\begin{center}
\includegraphics[width=0.8\textwidth]{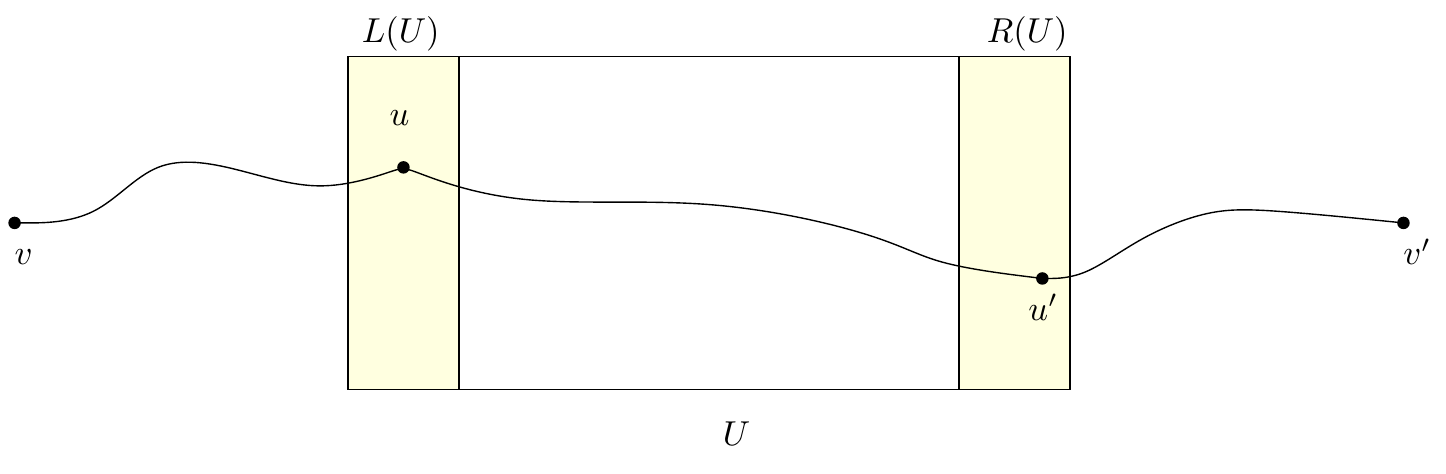}
\caption{$v$ and $v'$ as in the proof of Lemma \ref{l:treesupbasic}; $X_{v,v'}\geq X_{v,u}+X_{u,u'}+X_{u',v'}$}
\label{f:treesup}
\end{center}
\end{figure}

\begin{proof}[Proof of Lemma \ref{l:treesupbasic}]
As before, without loss of generality we shall restrict to the case $m=1$.
Consider the point $v=(-7r/8, -7r/8)$ and $v'=(15r/8, 15r/8)$. Now observe that it follows from Lemma \ref{l:penaltysum} that for $r$ sufficiently large and for $\theta$ sufficiently large 
$$\left| \E X_{v,u}+\E X_{u,u'}+\E X_{u',v'} -\E X_{v,v'}\right|\leq \frac{\theta r^{1/3}}{10}$$
for all $u\in L(U)$ and $u'\in R(U)$. It hence follows that for $r$ and $\theta$ sufficiently large we have
\begin{equation}
\label{e:treesupnew1}
\tilde{X}_{v,v'}\geq  \inf_{u\in L(U)} \tilde{X}_{v,u}+\inf_{u'\in R(U)}\tilde{X}_{u',v'}+\sup_{u\in  L(U), u'\in R(U)}\tilde{X}_{u,u'}-\frac{\theta r^{1/3}}{10}.
\end{equation}
%
%
%
%
%
%
Let $F_1$, $F_2$, $F_3$ denote the events
$$F_1=\left\{\inf_{u\in L(U)} \tilde{X}_{v,u} \geq -\frac{\theta r^{1/3}}{50}\right\};$$
$$F_2=\left\{\inf_{u'\in R(U)} \tilde{X}_{u',v'} \geq -\frac{\theta r^{1/3}}{50}\right\};$$
$$F_3=\left\{\sup_{u\in L(U), u'\in R(U)}\tilde{X}_{u,u'} \geq \theta r^{1/3}\right\}.$$
It is clear that
$$\left\{\tilde{X}_{v,v'}\geq \frac{\theta r^{1/3}}{2}\right\}\supseteq  F_1\cap F_2 \cap F_3.$$
Observe that the events $F_1,F_2,F_3$ are increasing in point configurations and hence by the FKG inequality we have
\begin{equation}
\label{e:treesupFKG}
\P\left[\tilde{X}_{v,v'}\geq \frac{\theta r^{1/3}}{2}\right]\geq \P[F_1]\P[F_2]\P[F_3].
\end{equation}
Notice now that it follows from Proposition \ref{t:treesup} that for $r$ and $\theta$ sufficiently large we have $\P[F_1]\geq \frac{1}{2}$ and $\P[F_2]\geq \frac{1}{2}$. Also observe that by Corollary \ref{c:moddevtilde} for $r$ sufficiently large and $\theta$ sufficiently large we have for some absolute constant $c>0$ that
$$\P\left[\tilde{X}_{v,v'}\geq \frac{\theta r^{1/3}}{2}\right]\leq e^{-c\theta}.$$
The above equation, together with (\ref{e:treesupFKG}) completes the proof of the lemma.
\end{proof}

Proposition \ref{t:treesup} has the following immediate corollary.

\begin{corollary}
\label{c:treesupwidth}
Consider the parallelogram $U=U_{h,m,\ell}$ where $\frac{4}{\psi}<m<\frac{\psi}{4}$ and $\ell >1$. There exists an absolute constant $c_1>0$, $h_0>0$ and $\theta_0=\theta_0(\psi)>0$ such that we have for all $h>h_0$ and $\theta> \theta _0$
\begin{equation}
\label{e:treesupgeneralwidth}
\P\left(\sup_{(u,u')\in \mathcal{S}(U)} \tilde{X}_{u,u'}\geq \theta\sqrt{\ell} h^{1/3}\right)\leq e^{-c_1\theta}.
\end{equation}
\end{corollary}

\section{Exponential tails in transversal fluctuation}
\label{s:trans}
It was proved by Johansson in \cite{J00} that the transversal fluctuations of the longest increasing subsequence from $(0,0)$ to $(n,n)$ is of the order $n^{2/3+o(1)}$. Using the estimates proved in the previous subsections we prove the following sharper version of Johansson's result.

For $a=(a_1,a_2),b=(b_1,b_2)\in \R^2$ with $a<b$, let $\Gamma_{a,b}=\{(x,\Gamma_{a,b}(x)):x\in [a_1,b_1]\}$ be the topmost maximal increasing path from $a$ to $b$. Let $\mathcal{L}_{a,b}=\mathcal{L}=\{(x,\mathcal{L}(x)):x\in [a_1,b_1]\}$ denote the straight line segment joining $a$ and $b$. Define  
$$D(a,b)=\sup_{x\in [0,r]}|\Gamma_{a,b}(x)-\mathcal{L}(x)|;$$
i.e., $D(a,b)$ denotes the maximal transversal fluctuation of the topmost maximal path from $a$ to $b$ about the straight line segment joining $a$ and $b$. 
Similarly define 
$$\tilde{D}(a,b)=\sup_{x\in [0,r]} (\Gamma_{a,b}(x)-\mathcal{L}(x));$$

\begin{theorem}
\label{t:transversal}
Define $\Gamma_r=\Gamma_{(0,0),(r,r)}$ and $D(r)=D((0,0),(r,r))$ (resp.\ $\tilde{D}(r)=\tilde{D}((0,0),(r,r))$). Then there exist absolute positive constants $r_0$ and $k_0$ and $c_4$ such that for all $r>r_0$, $k>k_0$, we have
$$\P[D(r) \geq kr^{2/3}]\leq e^{-c_4k}$$
and
$$\P[\tilde{D}(r) \geq kr^{2/3}]\leq \frac{1}{2}e^{-c_4k}.$$
\end{theorem}

Note the by the obvious symmetry, it suffices to prove only the second result in the statement of the theorem. We shall need a few lemmas in order to prove Theorem \ref{t:transversal}. The following lemma is basic and its proof is omitted. See Figure \ref{f:order}.
 
\begin{lemma}[Polymer Ordering]
\label{l:order}
Consider points $a=(a_1,a_2), a'=(a_1,a_3), b=(b_1,b_2)$ and $b'=(b_1,b_3)$ such that $a_1<b_1$ and $a_2\leq a_3\leq b_2\leq b_3$. Let $\Gamma_{a,b}$ and $\Gamma_{a',b'}$ be as in Theorem \ref{t:transversal}. Then we have $\Gamma_{a,b}(x)\leq \Gamma_{a',b'}(x)$ for all $x\in [a_1,b_1]$.  
\end{lemma}

\begin{figure*}[ht!]
\begin{center}
\includegraphics[width=0.5\textwidth]{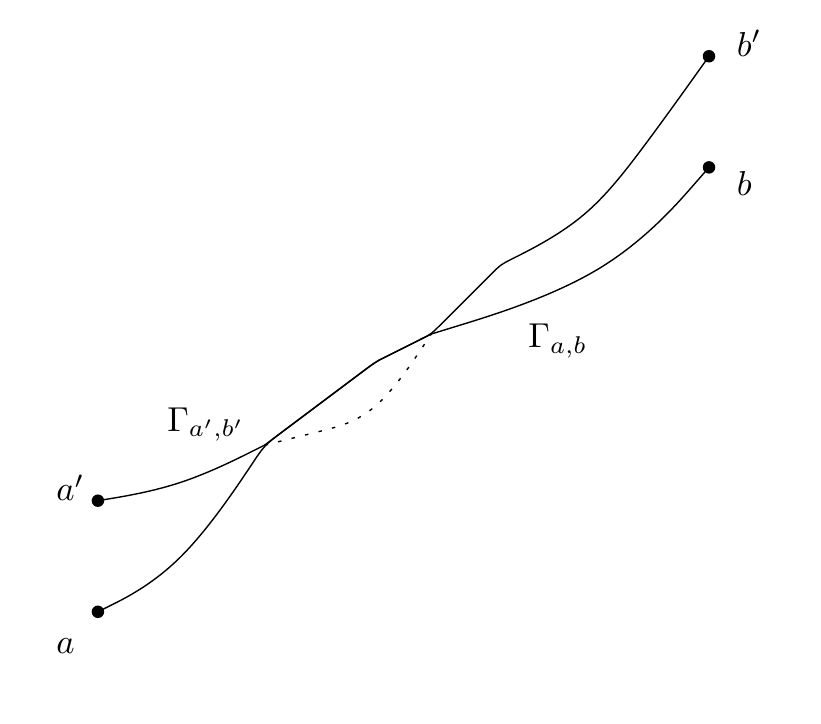}
\caption{Lemma \ref{l:order}: the dotted curve cannot be part of $\Gamma_{a',b'}$.}
\label{f:order}
\end{center}
\end{figure*}

The next lemma shows that the topmost maximal path $\Gamma_r$ cannot be too high at the midpoint of the interval $[0,r]$.

\begin{lemma}
\label{l:transversalmp}
Consider the set-up of Theorem \ref{t:transversal}. There exist constants $c>0$, $r_1>0$ and $k_1>0$ such that for all $k\geq k_1$ and for all $r>r_1$ we have 
$$\P\left[\Gamma_r\biggl(\frac{r}{2}\biggr)-\frac{r}{2}\geq kr^{2/3}\right]\leq e^{-ck}.$$
\end{lemma}

\begin{proof}
Let $A$ denote the event
$$A=\left\{\Gamma_r\biggl(\frac{r}{2}\biggr)-\frac{r}{2}\geq kr^{2/3}\right\}.$$
Observe that if $k\geq r^{1/3}$, then $\P[A]=0$ and hence we can restrict ourselves to the case $k<r^{1/3}$.
For $\ell \geq 0$, let $B_{\ell}$ denote the event $$B_{\ell}=\left\{\Gamma_r\biggl(\frac{r}{2}\biggr)-\frac{r}{2}\in [(k+\ell)r^{2/3}, (k+\ell+1)r^{2/3}]\right\}.$$
Finally let $G$ denote the event
$$G=\left\{\Gamma_r\biggl(\frac{r}{2}\biggr)\geq \frac{9r}{10}\right\}.$$
It is clear that (see Figure \ref{f:transversalmp})
$$A \subseteq \bigcup_{\ell=0}^{\lceil \frac{4r^{1/3}}{10}\rceil -k}  B_{\ell} \cup G.$$

\begin{figure}[ht!]
\begin{center}
\includegraphics[width=0.8\textwidth]{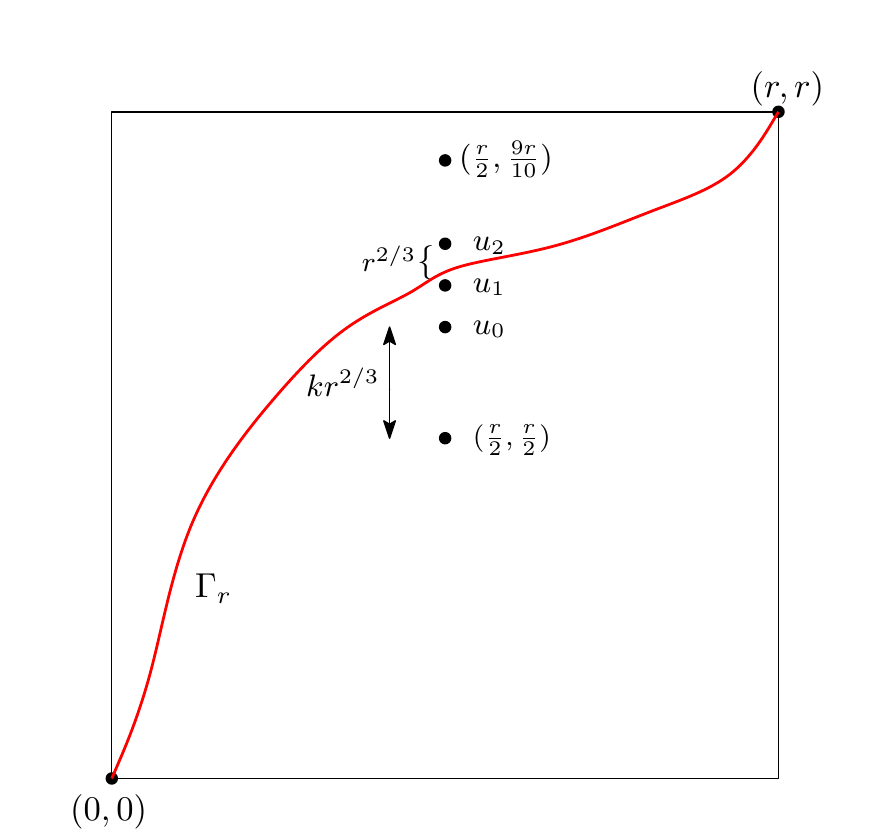}
\caption{Event $B_{1}$ as defined in the proof of Lemma \ref{l:transversalmp}}
\label{f:transversalmp}
\end{center}
\end{figure}

First let us bound $\P[G]$. Let $\mathcal{L}$ denote the line segment joining $(\frac{r}{2},\frac{9r}{10})$ and $(\frac{r}{2},r)$. It is clear that 
$$\sup_{u'\in \mathcal{L}} \biggl(X_{(0,0),u'}+X_{u',(r,r)}\biggr) \leq  X_{(0,0),(\frac{r}{2},r)}+X_{(\frac{r}{2},\frac{9r}{10}),(r,r)}$$
and hence
$$\P[G]\leq \P\left[ X_{(0,0),(\frac{r}{2},r)}+X_{(\frac{r}{2},\frac{9r}{10}),(r,r)} \geq X_{(0,0),(r,r)}\right].$$
An elementary computation as in Lemma \ref{l:penaltyproper} shows that
$$\E X_{(0,0),(\frac{r}{2},r)}+\E X_{(\frac{r}{2},\frac{9r}{10}),(r,r)}- \E X_{(0,0),(r,r)} \leq -c'r$$
for some constant $c'>0$. It follows from Theorem \ref{t:moddevuppertail} and  Theorem \ref{t:moddevlowertail} that for $r$ sufficiently large we have  
\begin{equation}
\label{e:gbound1}
\P[G]\leq e^{-cr^{2/3}}\leq e^{-ck}
\end{equation} 
for some absolute constant $c>0$ as $k\leq r^{1/3}$.

Now for the events $B_{\ell}$, observe the following. Let $u_{\ell}=(\frac{r}{2}, \frac{r}{2}+(k+\ell)r^{2/3})$. Let $L_{\ell}$ denote the line segment joining $u_{\ell}$ and $u_{\ell+1}$. Then we have
$$B_{\ell} \subseteq \left\{\sup_{u'\in L_{\ell}} X_{(0,0),u'}+X_{u',(r,r)}- X_{(0,0),(r,r)}\geq 0\right \}.$$

Notice that for $r$ sufficiently large and $k$ sufficiently large we have from Proposition \ref{t:treeinf} that
$$\P[\tilde{X}_{(0,0),(r,r)}\leq -k^{3/2}r^{1/3}]\leq e^{-ck}$$
for some constant $c>0$. Also by Proposition \ref{t:treesup} we have for $r$ and $k$ sufficiently large and $\ell \leq \frac{4r^{1/3}}{10}$ and some constant $c>0$
$$\P[\sup_{u'\in L_{\ell}} \tilde{X}_{(0,0),u'}\geq (\ell+ k)^{3/2}r^{1/3}]\leq e^{-c(k+\ell)}$$
and similarly
$$\P[\sup_{u'\in L_{\ell}} \tilde{X}_{u',(r,r)}\geq (k+\ell)^{3/2}r^{1/3}]\leq e^{-c(k+\ell)}.$$

Using Lemma \ref{l:penaltyproper} for $k$ sufficiently large and for all $u'\in L_{\ell}$
$$X_{(0,0),u'}+X_{u',(r,r)}-X_{(0,0),(r,r)} \leq \tilde{X}_{(0,0),u'}+\tilde{X}_{u',(r,r)}-\tilde{X}_{(0,0),(r,r)}-(k+\ell)^{7/4}r^{1/3}.$$

Putting together all these, we get for all $\ell \leq \lceil \frac{4r^{1/3}}{10}\rceil$
\begin{equation}
\label{e:bbound1}
\P[B_{\ell}]\leq e^{-c(k+\ell)}.
\end{equation}
Notice that we were very generous in the above calculations. We could have replaced the exponents $3/2$ by anything that is larger than $1$, and the exponent $7/4$ could have been replaced by anything smaller than $2$. However this is not important for us, and so $3/2$ was chosen arbitrarily and $7/4$ was chosen sufficiently large to beat it.

Taking a union bound over all $\ell \leq \lceil \frac{4r^{1/3}}{10}\rceil$
it follows from \eqref{e:gbound1} and \eqref{e:bbound1} that $\P[A]\leq e^{-ck}$ which completes the proof of the lemma. 
\end{proof}

Now we want to use a chaining argument to extend this bound on the transversal fluctuation of the topmost maximal path at the midpoint to a sequence of points, placed at the boundaries of dyadic sub-intervals of $[0,r]$. Let $r_1$ and $k_1$ be now given by Lemma \ref{l:transversalmp}, and let $k>10^5k_1$ be now fixed. Also fix $r$ sufficiently large so that $r^{2/3}\geq 10r_1$. Choose $j_0=j_0(k,r)>0$ such that $2^{-j_0}r=\frac{k}{10}r^{2/3}$. Without loss of generality we can assume that $j_0$ is an integer. For $j=1,2,\ldots , j_0$, define $S_j$ by
$$S_j=\{\ell r2^{-j}:\ell=0,1,\ldots 2^j\}.$$

For $j\geq 1$, define 
$$k_j=\frac{k}{10^5}\prod_{i=0}^{j-1}(1+2^{-i/10}).$$
Let $A_j$ denote the event that for all $x\in S_j$, we have $(\Gamma_r(x)-x)\leq k_jr^{2/3}$. 


\begin{lemma}
\label{l:transversal3}
Let $r$ and $k$ be as above. There exists an absolute constant $c>0$ such that for all $j$ with $j_0(k,r)\geq j\geq 1$, we have $\P[A_j^c\cap A_{j-1}]\leq 2^{-j}e^{-ck}$ where $A_0$ denotes the full set.
\end{lemma}
Notice that it is an immediate corollary of Lemma \ref{l:transversalmp} that $\P[A_1^c]\leq e^{-ck}$ and hence it remains to prove Lemma \ref{l:transversal3} for $j>1$. We postpone the proof for the moment and show how to complete the proof of Theorem \ref{t:transversal}.

\begin{lemma}
\label{l:transversal1}
If $A_j$ holds for each $j\leq j_0$, then we have  $\tilde{D}(r)\leq kr^{2/3}$.
\end{lemma}

\begin{proof}
Let $x \in [x_1,x_2]$ where $x_1$, $x_2$ are consecutive elements of $S_{j_0}$. Clearly then
$\Gamma_r(x)-x\leq (\Gamma_{r}(x_1)-x_1)\vee (\Gamma_{r}(x_2)-x_2) +\frac{k}{10}r^{2/3}$. The result follows.
\end{proof}

We are now ready to prove Theorem \ref{t:transversal}.
\begin{proof}[Proof of Theorem \ref{t:transversal}]
Denoting the full set by $A_0$, notice that 
$$\P[(\cap_{j\geq 1} A_j)^c]\leq \sum_{j\geq 1} \P[A_{j}^c\cap A_{j-1}].$$ The theorem now follows from Lemma \ref{l:transversal1} and Lemma \ref{l:transversal3}.
\end{proof}

It remains to prove Lemma \ref{l:transversal3}. This will follow from the following lemma.

\begin{lemma}
\label{l:transversal5}
In the set-up of Lemma \ref{l:transversal3},
fix $j< j_0$ and $0\leq  h \leq 2^{j}$. Let $k_j=\frac{k}{10^5}\prod_{i=0}^{j-1}(1+2^{-i/10})$ be defined as above. Consider the line $y=x+k_jr^{2/3}$. Let $u_{h}$ denote the point where this line intersects the vertical line $x=h2^{-j}r$. Let 
$$A_{h,j}=\left\{\Gamma_{u_{h},u_{h+1}}((2h+1)2^{-(j+1)}r)-(2h+1)2^{-(j+1)}r\leq k_{j+1}r^{2/3}\right\}.$$ 
Then $\P[A_{h,j}^c]\leq 4^{-j}e^{-ck}$ for some absolute constant $c>0$.
\end{lemma}

\begin{proof}
Set $r'=2^{-j}r$. By translation invariance, we have that 
$$\P[A_{h,j}^c]=\P\left[\Gamma_{r'}\biggl(\frac{r'}{2}\biggr)-\frac{r'}{2}\geq (k_{j+1}-k_j)r^{2/3}\right].$$

Observing that $(k_{j+1}-k_j)r^{2/3}=k_j2^{-j/10}2^{2j/3}(r')^{2/3}\geq \frac{k2^{j/2}}{10^{5}}(r')^{2/3}$ and $r'>r_1$, $k>10^5k_1$ it follows from Lemma \ref{l:transversalmp} that
$\P[A_{h,j}^c]\leq e^{-ck2^{j/2}}\leq 4^{-j}e^{-ck/2}$ since $k$ is sufficiently large. This completes the proof of the lemma.
\end{proof}

\begin{figure}[h!]
\begin{center}
\includegraphics[width=0.6\textwidth]{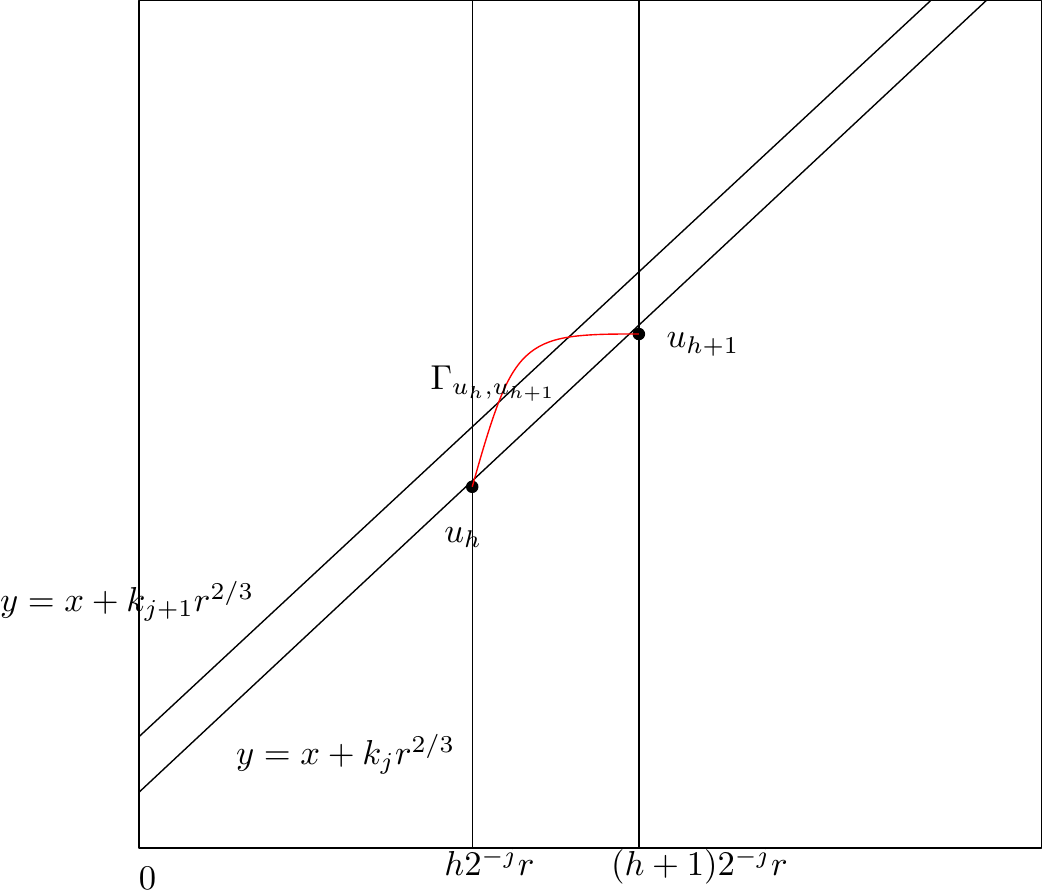}
\caption{Event $A_{h,j}^c$ in the proof of Lemma \ref{l:transversal3}}
\label{f:transversalc}
\end{center}
\end{figure}

Now we complete the proof of Lemma \ref{l:transversal3}.

\begin{proof}[Proof of Lemma \ref{l:transversal3}]
Observe that using Lemma \ref{l:order}, it follows that 
$$A_{j+1}^{c}\cap A_{j} \subseteq \bigcup_{h\in [2^j]} A_{h,j}^c.$$
The lemma now follows from Lemma \ref{l:transversal5} by taking a union bound over all $h$.
\end{proof}

Observe now that the arguments proving Theorem \ref{t:transversal} would still go through if we were looking at the transversal fluctuation of the topmost maximal path between two points $a$ and $b$ such that the line segment joining them has slope bounded away from $0$ and infinity. In particular we have the following corollary whose proof we omit.

\begin{corollary}
\label{c:transversaldiffslope}
Let $a=(a_1,a_2)<b=(b_1,b_2)\in \R^2$ be such that $m:=\frac{b_2-a_2}{b_1-a_1}\in (\frac{100}{\psi}, \frac{\psi}{100})$. Then there exist  positive constants $r_0(\psi)$ and $k_0(\psi)$ and $c_4(\psi)$ such that for all $r:=b_1-a_1>r_0$, $k>k_0$, we have
$$\P\left[D(a,b)\geq kr^{2/3}\right]\leq e^{-c_4k}.$$
\end{corollary}

\section{Bounds on constrained paths}
\label{s:const}
Our objective in this subsection is to obtain bounds on lengths of the longest increasing path between two points constrained to be contained within certain parallelograms. Clearly the constrained paths have smaller lengths than unconstrained paths, hence the upper tail results in Corollary \ref{c:moddevtilde} automatically holds in this scenario. For the lower tail, we have the following result.

\begin{lemma}
\label{l:pathbox}
Let $k>0$ be fixed. Consider the parallelogram $U=U_{r,m,k}$ where $m\in (\frac{2}{\psi}, \frac{\psi}{2})$. Let $u=(0,0)$ and let $u'=(r,mr+hr^{2/3})$ where $|h|\leq \frac{k}{2}$. Then there exist positive constants $r_0(k,\psi)$, $\theta_0(k,\psi)$ and a constant $c=c(k,\psi)>0$ such that for all $r>r_0$ and $\theta>\theta_0$ we have
$$\P[\tilde{X}_{u,u'}^{U^c}\leq -\theta r^{1/3}]\leq e^{-c\sqrt{\theta}}.$$
\end{lemma}

Observe that instead of $\theta$, we get a worse exponent of $\theta^{1/2}$ in this case which is not optimal (the exponent of $\theta$ was not optimal either), but is sufficient for our purposes.

\begin{proof}[Proof of Lemma \ref{l:pathbox}]
As usual we take $m=1$ without loss of generality.
Note that we only need to consider the case where $\theta \leq 10 r^{2/3}$, as for $\theta > 10 r^{2/3}$ and $r$ sufficiently large $$\P[\tilde{X}_{u,u'}^{U^c}\leq -\theta r^{1/3}]=\P[X_{u,u'}^{U^c}<0]=0.$$ 

Let $\theta$ now be sufficiently large and set $J=\lfloor\theta^{3/4}\rfloor $. For $0\leq j\leq J$, define $u_j=(\frac{jr}{J},\frac{jr+jhr^{2/3}}{J})$. See Figure \ref{f:constrainedp2p}. Notice that,
\begin{equation}
\label{e:p2pbox1}
\P[\tilde{X}_{u_{j},u_{j+1}}^{U^c}\leq -\frac{\theta}{J} r^{1/3}]\leq \P[\tilde{X}_{u_j,u_{j+1}}\leq -\frac{\theta}{J} r^{1/3}]+\P[A_j]
\end{equation}
where $A_j$ denotes the event the all maximal paths from $u_{j}$ to $u_{j+1}$ exit $U$. Since 
$$\frac{\theta}{J}r^{1/3}= \frac{\theta}{J^{2/3}}\cdot (\frac{r}{J})^{1/3}$$
it follows using Corollary \ref{c:moddevtilde} that for $r$ sufficiently large the first term in the right hand side of \eqref{e:p2pbox1} is bounded by $e^{-c\theta/J^{2/3}}\leq e^{-c\sqrt{\theta}}$ for some absolute constant $c>0$. Notice that we have used above that since $J\leq \theta \leq 10 r^{2/3}$, we have  $r\gg J$. Since 
$kr^{2/3}=kJ^{2/3}\cdot (\frac{r}{J})^{2/3}$
it follows similarly that for $\theta$ (and hence $J$) sufficiently large and $r$ sufficiently large the second term in the right hand side of \eqref{e:p2pbox1} is bounded by $e^{-ckJ^{2/3}}=e^{-c\theta^{1/2}}$.
Taking a union bound over all $j\in \{0,1,\ldots, J-1\}$ the result follows.
\end{proof}

\begin{figure}[h!]
\begin{center}
\includegraphics[width=\textwidth]{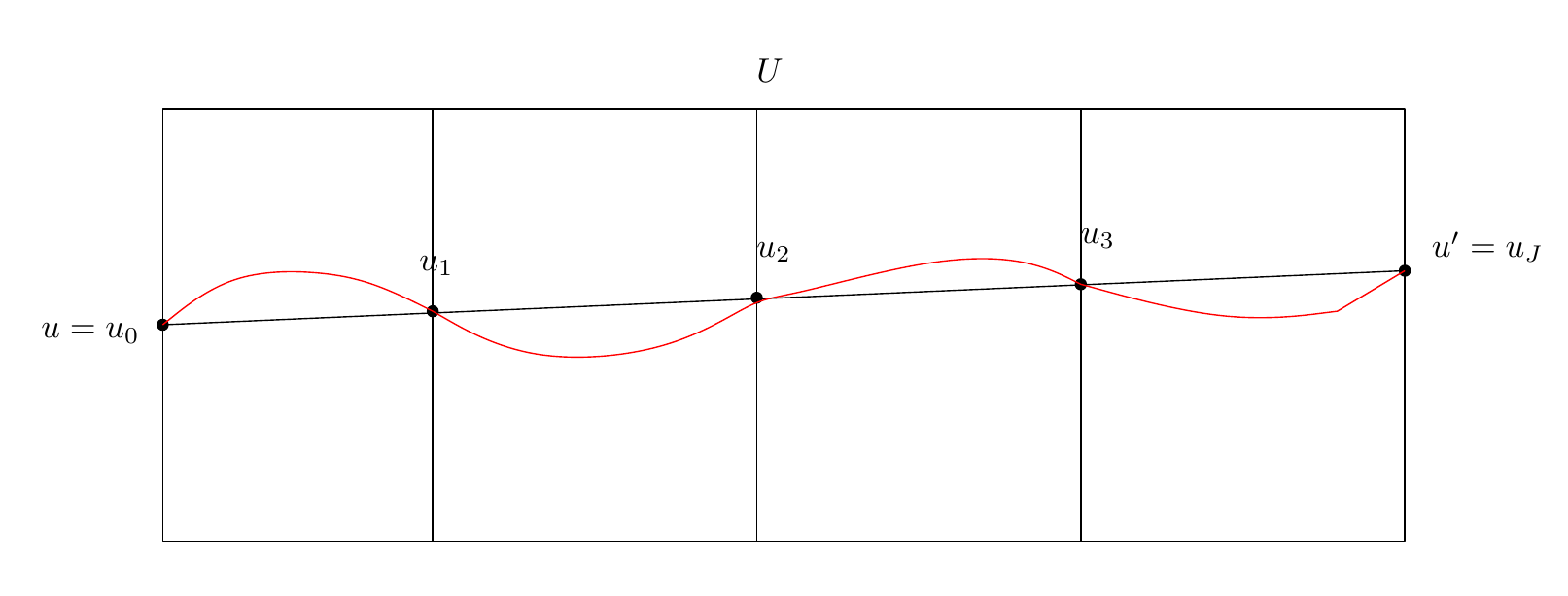}
\caption{Sequence of points $u_j$ as in the proof of Lemma \ref{l:pathbox}, between $u_j$ and $u_{j+1}$, the path is constrained to be in a very fat parallelogram, and hence the maximal path between $u_j$ and $u_{j+1}$ is unlikely to exit $U$ for any $j$}
\label{f:constrainedp2p}
\end{center}
\end{figure}

Our next objective is to prove results analogous to Proposition \ref{t:treeinf} and Proposition \ref{t:treesup} for constrained maximal paths. The result for the upper tail is a trivial corollary of Proposition \ref{t:treesup}, and for the lower tail we have the following result.

\begin{proposition}
\label{t:treebox}
Let $W$ be a fixed positive constant. Consider the parallelogram $U=U_{r,1,W}$. Then there is an absolute constant $\theta_0>0$ such that for all $\theta>\theta_0$ and a constant $c=c(W)>0$ such that and for all sufficiently large $r \geq r_0(W)$ we have

$$\P[\inf_{(u,u')\in \mathcal{S}(U)} \tilde{X}_{u,u'}^{U^c}\leq -\theta r^{1/3}]\leq e^{-c\theta^{1/3}}. $$
\end{proposition}

Proposition \ref{t:treebox} follows from the next lemma.



\begin{lemma}
\label{l:treeboxinfbasiconesided}
Consider the parallelogram $U=U_{r,m,W}$ where $m\in (\frac{2}{\psi}, \frac{\psi}{2})$. Define $u_*=(\frac{r}{2}, \frac{mr}{2})$ and $L(U)= U\cap \{x< r/8\}$. Then there exist constants $r_0$, $\theta_2$ and $c_3>0$ such that for all $r>r_0$ and $\theta>\theta_2$ we have

\begin{equation}
\label{e:tibo1box}
\P\left(\inf_{u\in L(U)} \tilde{X}_{u,u_*}^{U^c}\leq -\theta r^{1/3}\right)\leq e^{-c_3\theta^{1/3}}.
\end{equation}
\end{lemma}

Proposition \ref{t:treebox} follows from Lemma \ref{l:treeboxinfbasiconesided} exactly as in the proof of Proposition \ref{t:treeinf}. We shall omit the details. Lemma \ref{l:treeboxinfbasiconesided} is proved similarly to how \eqref{e:treeinfbasic1} is established in the proof of Lemma \ref{l:treeinfbasic}. In stead of writing out the whole proof again, we only point out the significant differences below.

As before we assume without loss of generality $m=1$. As in the proof of Lemma \ref{l:treeinfbasic} we construct a graph with vertices in $U$ such that for every $u\in L(U)$ there exists a path $u_0,u_1,\ldots ,u_{i}=u_{*}$ in this graph. Then we ask that $\tilde{X}_{u_{j},u_{j+1}}$ is not too small for any $j$, and show that on this event $\tilde{X}_{u,u_{*}}\geq -\theta r^{1/3}$. Then an upper bound on $\P[\inf_{u\in L(U)} \tilde{X}_{u,u_*}^{\partial U}\leq -\theta r^{1/3}]$ is established taking an union bound over the edges of the graph, where  in stead of Theorem \ref{t:moddevlowertail} (or the lower tail bound for Corollary \ref{c:moddevtilde}), now we use Lemma \ref{l:pathbox}.

However, there is a significant difference in how the sequence of $u_j$'s is constructed for a given point $u\in L(u)$. Observe that Lemma \ref{l:pathbox} only applies to pairs of points $(u,u')$ if the distance of both $u$ and $u'$ from the boundary of $U$ is at least $c|u-u'|^{2/3}$ for some constant $c>0$. Hence for example we cannot set $u_0$ to be the top left corner of $u$; see Figure \ref{f:constrained2}. Hence we do the following. For $u=(x,y)\in L(u)$, define $g(u)$ as follows. If $Wr^{2/3}-|y-x| \geq 10r^{1/4}$; set $g(u)=u$. If $y-x\geq Wr^{2/3}-10r^{1/4}$; set $g(u)=(x+10r^{1/4},y)$. If $y-x\leq -Wr^{2/3}+10r^{1/4}$, set $g(u)=(x,y+10r^{1/4})$. Intuitively, for points $u$, that are near the corners of $u$, we are choosing $g(u)$ to be a point slightly away from the corners of $U$. See Figure \ref{f:constrained2}. It is not too hard to observe that for $r$ and $\theta$ sufficiently large   
$$\left\{\inf_{u\in L(U)} \tilde{X}_{u,u_*}^{U^c}\geq -\theta r^{1/3} \right\}\supseteq \left\{\inf_{u\in L(U)} \tilde{X}_{g(u),u_*}^{U^c}\geq -\frac{\theta r^{1/3}}{2} \right\}.$$

\begin{figure}[h!]
\begin{center}
\includegraphics[width=0.3\textwidth]{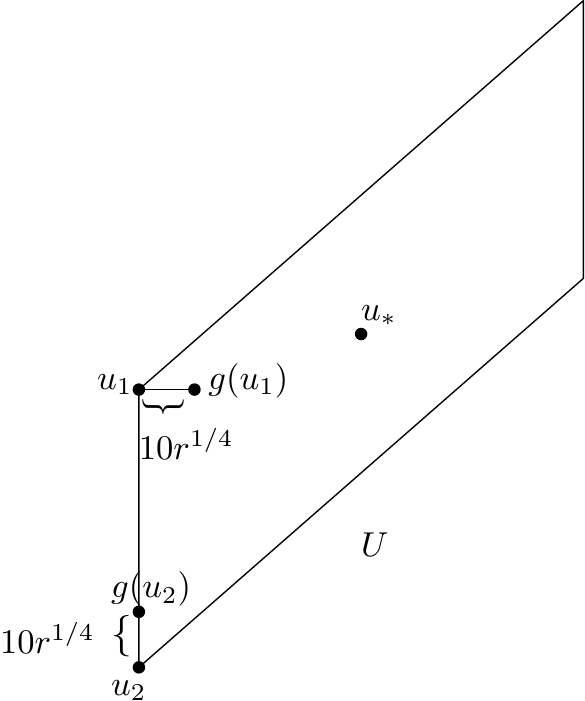}
\caption{Constructing $g(u)$'s for different $u\in L(U)$ for the proof of Lemma \ref{l:treeboxinfbasiconesided}}
\label{f:constrained2}
\end{center}
\end{figure}

The rest of the proof is identical to the corresponding part in the proof of Lemma \ref{l:treeinfbasic} after observing that if the sequence $u_0, u_1, \ldots$ is constructed as in that proof starting from $u_{-1}=g(v)$ for some $v\in L(u)$ then for all $j$ the distance from $u_j$ and $u_{j+1}$ to $\partial U$ as it least $c|u_{j+1}-u_{j}|^{2/3}$ for some absolute constant $c$. Observe that since we start with an exponent of $\theta ^{1/2}$, we end up with a worse exponent $\theta^{1/3}$, which nonetheless is sufficient for our purposes.

\section{Necessary modifications for the Exponential Case}
\label{s:discrete}
In this section we briefly describe how we can adapt the preceding arguments to prove Theorem \ref{t:maintheoremdlpp}. The argument is in essence the same, though some minor modification is necessary. In this case we  shall only consider increasing paths between points in $\Z^2$. For $u=(x,y),u'=(x',y')\in \Z^2$, we define
$$X_{u,u'}=\max_{\pi}\sum_{v\in \pi\setminus \{u'\} } \xi_{v}$$
where the maximum is taken over all increasing lattice paths $\pi$ from $u$ to $u'$.
The following Tracy-Widom Fluctuation result in this case is due to Johansson \cite{Jo99}.

\begin{theorem}
\label{t:Jo99}
Let $h>0$ be fixed. Let $v=(0,0)$ and $v_{n}=(n,\lfloor hn \rfloor)$. Let $T_{n}=X_{v,v_n}$. Then
\begin{equation}
\label{e:twlimitdiscrete}
\dfrac{T_n-(1+\sqrt{h})^2n}{h^{-1/6}(1+\sqrt{h})^{4/3}n^{1/3}} \stackrel{d}{\rightarrow} F_{TW}.
\end{equation}
\end{theorem}

As before, we define $\tilde{X}_{u,u'}=X_{u,u'}-\E X_{u,u'}$ and $\hat{X}_{u,u'}=X_{u,u'}-2d(u,u')$ since now $2d(u,u')$ is the first order term in $\E X_{u,u'}$. We have the following moderate deviation estimates from \cite{BFP12} and \cite{Baik}.

\begin{theorem}
\label{t:moddevdiscrete}
Let $\psi>1$ be fixed. Let $Z_{h,n}$ denote the last passage time from $(0,0)$ to $(n, \lfloor h n \rfloor)$ where $h \in (1/\psi, \psi)$. Then  there exist constants $N_0=N_0(\psi)$, $t_0=t_0(\psi)$ and $c=c(\psi)$ such that we have for all $n>N_0, t>t_0$ and all $h \in (\frac{1}{\psi}, \psi)$

$$\P[|Z_{h,n}- n(1+\sqrt{h})^{2}|\geq tn^{1/3}]\leq e^{-ct}.$$
\end{theorem}
Theorem \ref{t:moddevdiscrete} is not explicitly stated and proved in \cite{BFP12} in the above form. We sketch briefly below how Theorem \ref{t:moddevdiscrete} follows from the estimates in \cite{BFP12}, as explained to us by Baik \cite{Baik}. Recall the connection between DLPP with Exponential weights and TASEP with step initial conditions. The upper tail estimate in Theorem \ref{t:moddevdiscrete} is contained in Lemma 1 of \cite{BFP12}, see equation (37) there. For the lower tail, observe the following. Using the Fredholm determinant formula for the distribution function of $Z_{h,n}$ (properly centred and scaled) given by (22) in \cite{BFP12}, the lower tail follows from Proposition 3 there and the paragraph preceding it. See (56) and (57) in \cite{BFP12}.  Observe that, unlike \cite{LM01, LMS02} the exponents obtained from \cite{BFP12} are not optimal, it gives an exponent of $t$ in the upper tail and $t^{3/2}$ in the lower tail (the optimal exponents are $t^{3/2}$ and $t^3$ respectively, as mentioned before) but this is sufficient for our purposes.

Once we have Theorem \ref{t:moddevdiscrete} at our disposal, the proof proceeds in exactly similar manner, we establish all consequences of moderate deviation estimates in \S~\ref{s:moddev} (with possibly changed constants) using Theorem \ref{t:moddevdiscrete} instead of Theorem \ref{t:moddevlowertail} and Theorem \ref{t:moddevuppertail}. The proof now proceeds as before. We define key events exactly as in \S~\ref{s:events}. The results in \S~\ref{s:maxpathnice}, \S~\ref{s:gx}, \S~\ref{s:rxcond} and \S~\ref{s:locsuccess} follows as before.

To get an improved path in an analogous manner to the argument of \S~\ref{s:finish}, we do the following. Observe that
\[
\zeta_{1-\epsilon} \stackrel{d}{=} \zeta_{1} + B_{\epsilon}\zeta'_{1-\epsilon}
\]
where $\zeta_{1-\epsilon}$, $\zeta'_{1-\epsilon}$ are exponential variables with rate $(1-\epsilon)$, $\zeta_1$ is an exponential variable with rate $1$, $B_{\epsilon}$ is a $\mbox{Ber}(\epsilon)$ variable and all of these are independent. Introducing a defect on the diagonal is equivalent to reinforcing the diagonal (i.e., adding to the entries on the diagonal) with these independent variables with positive expectation. The rest of the arguments in \S~\ref{s:finish} works in the same way as before, where we reinforce on discrete lines $y=x+m$ with $m\in \Z$, and add up the improvements instead of integrating. Notice that we can get rid of the area condition in $G_x$ for the Exponential case, as for any path $\gamma$ the expected increase in length in the reinforced environment is proportional to the number of points $\gamma$ hits the reinforced line.

Another thing one needs to take care of is the following. In the Exponential set-up to make sure that the length of two augmented paths is equal to the sum of their individual lengths, we do not add up the contribution of the very last vertex. To make sure, that this does not change any of our estimates (and also to make sure that the estimates about paths conditioned not to hit certain paths work as before), we need to condition on the event that no passage time on $[0,n]^2$ is bigger than $\log ^2 n$. This event holds with high probability and hence rest of the arguments will work as before.

{\bf Remark:}
Notice that the main difference between the Exponential and the Poissonian case is that in the Poissonian case, the moderate deviation estimates for the length of a path between two corners of a rectangle does not depend on the aspect ratio of the rectangle. In the Exponential case, we can only get uniform moderate deviation estimates for rectangles with bounded aspect ratio, which forces us to work harder to avoid steep paths, or work out different estimates for steep paths. To deal with only the Poissonian case, one can get rid of all the conditions involving $\psi$ (e.g. the steepness condition) and also one can prove Proposition \ref{t:treesup} and Proposition \ref{t:treeinf} without the assumptions that the slope between the pairs of points considered are bounded. However, we worked under these assumptions in order to have a proof which can be adapted to the Exponential case with minimal changes.

Before we finish we add a brief discussion about the diffusive fluctuations for $L_n^{\lambda}$ and $T_n^{\epsilon}$.

\begin{figure*}[h]
\begin{center}
\includegraphics[width=0.5\textwidth]{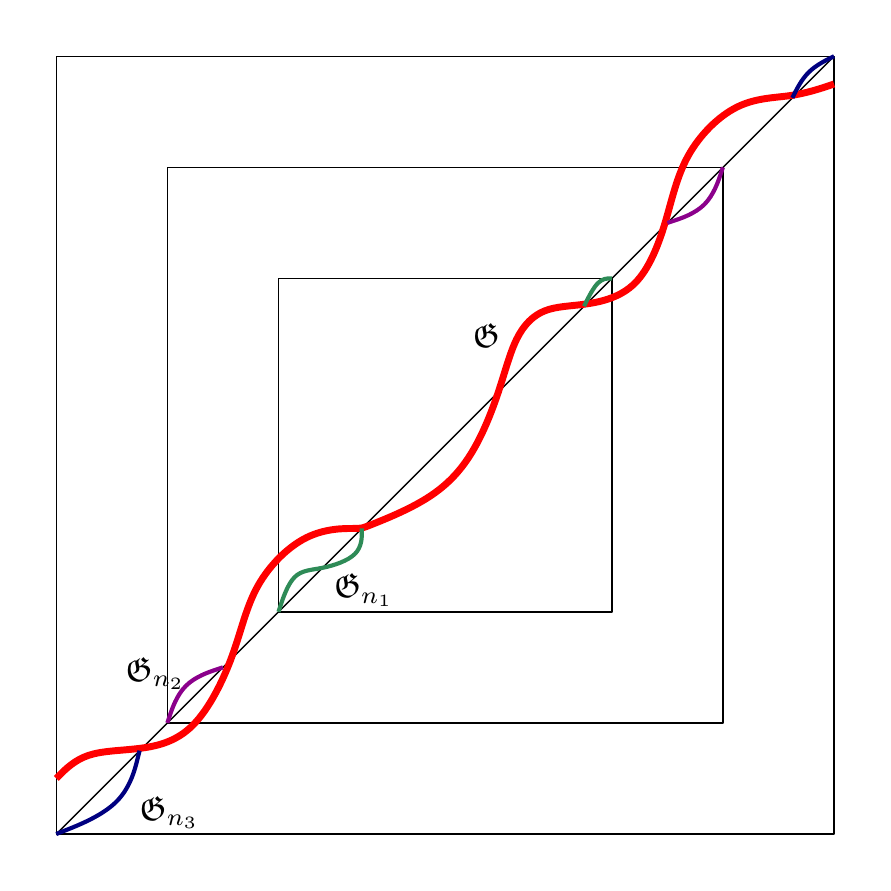}
\caption{Paths $\mathfrak{G}_{n_1}$, $\mathfrak{G}_{n_2}$, $\mathfrak{G}_{n_3}$ for $n_1<n_2<n_3$ and the limiting path $\mathfrak{G}$}
\label{f:clt}
\end{center}
\end{figure*}

{\bf Remark:} For convenience we shall only consider the case of Exponential Directed Last Passage Percolation here. Let $\epsilon>0$ be fixed. For $n\in \N$, consider the maximal path $\mathfrak{G}_n$ from $(-n,-n)$ to $(n,n)$, let $\mathfrak{T}_n$ denote its passage time. It is a consequence of Theorem \ref{t:maintheoremdlpp} that reinforcement on the diagonal leads to a pinning transition for $\mathfrak{G}_n$, and it is not too hard conclude that the typical transversal distance of $\mathfrak{G}_n$ from the diagonal is $O(1)$, and therefore as $n\to \infty$, the paths $\mathfrak{G}_n$ converges almost surely to a limiting path $\mathfrak{G}$; see Figure \ref{f:clt}.  For $i\in \Z$, let $\mathfrak{G}_{i}\in \Z$ be such that $(i,\mathfrak{G}_{i})\in \mathfrak{G}$. Let $\mathfrak{X}_i$ be the passage time corresponding to this vertex. By translation invariance of the environment, $\{\mathfrak{X}_i\}_{i\in \Z}$ is a stationary process. Also since the path is pinned to the diagonal, it can be argued that the correlation function of this process decays sufficiently fast. This implies that $\sum_{i=-n}^{n}\mathfrak{X}_i$ has diffusive fluctuations and obeys a Gaussian central limit theorem. A central limit theorem for $\mathfrak{T}_n$ (which is the same as a CLT for $T_n^{\epsilon}$) can then be inferred by observing that $\sum_{-n}^{n} \mathfrak{X}_i$ is close in distribution to $\mathfrak{T}_n$ (after rescaling), again using the fact that the paths $\mathfrak{G}_n$ are localised around the diagonal.    

\bigskip

{\bf Acknowledgements.}  The authors would like to thank Jinho Baik for many helpful discussions and in particular for pointing out and discussing the results of~\cite{BFP12}. We thank anonymous referees for numerous comments and suggestions that helped improve the exposition of the paper. V.S.  thanks H. Spohn for introducing the problem to him fifteen years ago and for many stimulating discussions since then. V.S. also thanks V. Beffara. This work was completed while R.B. was a graduate student at the Department of Statistics, University of California, Berkeley. He gratefully acknowledges the support of UC Berkeley Graduate Fellowship.

\bibliography{slowbondnew}
\bibliographystyle{plain}

\end{document}